\documentclass[a4paper]{article}

\usepackage[latin1]{inputenc}
\usepackage[T1]{fontenc}
\usepackage{amssymb,amsmath,amsthm,amscd,mathrsfs}
\usepackage[all]{xy}
\usepackage{color}
\usepackage[margin=3cm]{geometry}
\usepackage{amsmath}
\usepackage{paralist}

\newcommand{\bP}{{\mathbb P}}
\newcommand{\bQ}{{\mathbb Q}}

\newcommand{\bZ}{{\mathbb Z}}

\newcommand{\cM}{{\mathscr M}}

\newcommand{\cP}{{\mathscr P}}

\newcommand{\cW}{{\mathscr W}}

\newcommand{\dC}{{\mathcal C}}
\newcommand{\dD}{{\mathcal D}}

\newcommand{\dK}{{\mathcal K}}

\renewcommand{\phi}{\varphi}

\DeclareMathOperator{\iso}{\cong}

\DeclareMathOperator{\surj}{\twoheadrightarrow}

\DeclareMathOperator{\too}{\longrightarrow}

\DeclareMathOperator{\Mon}{Mon}

\renewcommand{\div}{{\rm div}}
\DeclareMathOperator{\Aut}{Aut}

\DeclareMathOperator{\half}{\frac{1}{2}}

\newcommand{\TODO}[1]{}

\newtheorem{thm}{Theorem}[section]
\newtheorem{defi}[thm]{Definition}
\newtheorem{prop}[thm]{Proposition}
\newtheorem{lemme}[thm]{Lemma}
\newtheorem{cor}[thm]{Corollary}

\theoremstyle{remark}
\newtheorem{remark}[thm]{Remark}
\newtheorem{nota}[thm]{Notation}
\newtheorem{rmk}[thm]{Remark}
\newtheorem{ex}[thm]{Example}

\DeclareMathOperator{\Z}{\mathbb{Z}}
\DeclareMathOperator{\R}{\mathbb{R}}
\DeclareMathOperator{\N}{\mathbb{N}}

\DeclareMathOperator{\codim}{Codim}

\DeclareMathOperator{\Pic}{Pic}
\DeclareMathOperator{\rk}{rk}

\DeclareMathOperator{\Ext}{Ext}

\DeclareMathOperator{\Ima}{Im}
\DeclareMathOperator{\Rea}{Re}

\DeclareMathOperator{\Fix}{Fix}
\DeclareMathOperator{\Supp}{Supp}

\DeclareMathOperator{\Vect}{Vect}

\DeclareMathOperator{\id}{id}

\DeclareMathOperator{\Sing}{Sing}

\DeclareMathOperator{\Q}{\mathbb{Q}}

\DeclareMathOperator{\C}{\mathbb{C}}
\DeclareMathOperator{\Pj}{\mathbb{P}}

\DeclareMathOperator{\BK}{\mathcal{B}\mathcal{K}}
\DeclareMathOperator{\coloneqq}{:=}
\DeclareMathOperator{\nutild}{\widetilde{\nu}}
\DeclareMathOperator{\rhotild}{\widetilde{\rho}}
\DeclareMathOperator{\hdel}{\hat{\delta}}
\DeclareMathOperator{\hSig}{\hat{\Sigma}}
\DeclareMathOperator{\reg}{reg}
\DeclareMathOperator{\Hdg}{Hdg}
\DeclareMathOperator{\Bl}{Bl}

\newcommand{\eq}[1][r]
{\ar@<-3pt>@{-}[#1]
\ar@<-1pt>@{}[#1]|<{}="gauche"
\ar@<+0pt>@{}[#1]|-{}="milieu"
\ar@<+1pt>@{}[#1]|>{}="droite"
\ar@/^2pt/@{-}"gauche";"milieu"
\ar@/_2pt/@{-}"milieu";"droite"}

\newcommand{\incl}[1][r]
  {\ar@<-0.2pc>@{^(-}[#1] \ar@<+0.2pc>@{-}[#1]}

\begin{document}
\title{\bf Wall divisors on irreducible symplectic orbifolds of Nikulin-type}

\author{Gr\'egoire \textsc{Menet}; Ulrike \textsc{Riess}} 

\maketitle
\begin{abstract}
  We determine the wall divisors on irreducible symplectic orbifolds which are deformation
  equivalent to a special type of examples, called Nikulin orbifolds. The Nikulin orbifolds are obtained as partial resolutions in codimension 2 of a quotient by a symplectic involution of a Hilbert scheme of 2 points on a K3 surface. 
  This builds on the previous article \cite{Menet-Riess-20} in which the theory of wall divisors was
  generalized to orbifold singularities.
\end{abstract}

\section{Introduction}
\subsection{Motivations and main results}
During the last years, many efforts have been made to extend the theory of smooth compact varieties with trivial first
Chern class to a framework of varieties admitting some singularities.
 Notably, let us cite, the generalization of the Bogomolov decomposition theorem
\cite{Bakker}.
One of the motivations for such generalizations is given by the minimal model program
in which certain singular varieties appear naturally.

More specifically, in the theory of irreducible symplectic varieties, many generalizations can be mentioned. 
One of the most important concerns the global Torelli theorem which allows to obtain geometrical information on the variety from its period (\cite{Bakker-Lehn-GlobalTorelli}, \cite{Menet-2020} and \cite{Menet-Riess-20}).

In this paper, we are considering a specific kind of singularities: quotient singularities. A complex analytic space with only quotient singularities is called an \emph{orbifold}. An orbifold $X$ is called
\emph{irreducible holomorphically symplectic} if $X\smallsetminus \Sing X$ is simply connected, admits a
unique (up to a scalar multiple), non-degenerate holomorphic 2-form and $\codim \Sing X\geq 4$ (Definition
\ref{def}). The framework of irreducible symplectic orbifolds appears to be very favorable. In particular,
general results about the K\"ahler cone have been generalized for the first time in this context (see
\cite{Menet-Riess-20}). This is particularly  important, since knowledge on the K\"ahler cone is needed to be
able to apply the global Torelli theorem (see Theorem \ref{mainHTTO}) effectively. The key tool used to study
the K\"ahler cone of irreducible symplectic orbifolds are wall divisors (originally introduced for the smooth
case in \cite{Mongardi13}). 
\begin{defi}[{\cite[Definition 4.5]{Menet-Riess-20}}]
Let $X$ be an irreducible symplectic orbifold and let $D\in\Pic(X)$. Then $D$ is called a \emph{wall divisor}
if $q(D)<0$ and $g(D^{\bot})\cap \BK_X =\emptyset$, for all $g\in \Mon^2_{\Hdg}(X)$.

\end{defi}

Here, $q$ denotes the Beauville--Bogomolov form on $H^2(X,\bZ)$. The birational K\"ahler cone $\BK_X$
and the Hodge monodromy group $\Mon^2_{\Hdg}(X)$ are recalled in Section \ref{Kahlersection} and
Definition \ref{transp}, respectively.

In particular, we recall that the K\"ahler classes can be characterized by their intersections with the wall
divisors (see Corollary \ref{cor:desrK}).

A very practical feature of wall divisors is their deformation invariance. More precisely, let $\Lambda$ be a lattice of signature $(3,\rk\Lambda-3)$ and $(X,\varphi)$ a marked irreducible symplectic orbifold with $\varphi:H^2(X,\Z)\simeq \Lambda$. Then there exists a set $\mathscr{W}_{\Lambda}\subset \Lambda$ such that for all $(Y,\psi)$ deformation equivalent to $X$, the set $\psi^{-1}(\mathscr{W}_{\Lambda})\cap H^{1,1}(Y,\Z)$ is the set of wall divisors of $Y$ (see Theorem \ref{wall}). We call the set $\mathscr{W}_{\Lambda}$ the \emph{set of wall divisors of the deformation class of $X$}. 

In this paper, we are going to provide the first description of the wall divisors of a deformation class of singular irreducible symplectic varieties. 
The most "popular" singular irreducible symplectic variety, in the literature (see \cite[Section 13, table 1, I2]{Fujiki-1983}, \cite{Marku-Tikho}, \cite{Menet-2014}, \cite{Menet-2015}, \cite[Section 3.2 and 3.3]{Menet-Riess-20}, \cite{Camere}), is denoted by $M'$ and recently named Nikulin orbifold in \cite{Camere};
it is obtained as follows. Let $X$ be an irreducible symplectic manifold of $K3^{[2]}$-type and $\iota$ a symplectic involution on $X$.
By \cite[Theorem 4.1]{Mongardi-2012}, $\iota$ has 28 fixed points and a fixed K3 surface $\Sigma$. 
We obtain $M'$ as the blow-up of $X/\iota$ in the image of $\Sigma$ (see Example \ref{exem}); we denote by $\Sigma'$ the exceptional divisor.
The orbifolds deformation equivalent to this variety will be called \emph{orbifolds of Nikulin-type}. 
We also recall that the Beauville--Bogomolov lattice of the orbifolds of Nikulin-type is $U(2)^3\oplus
E_8(-1)\oplus(-2)^2$ (see Theorem \ref{BBform}).
\begin{thm}\label{main}
Let $\Lambda:=U(2)^3\oplus E_8(-1)\oplus(-2)^2$.

The set $\mathscr{W}_{M'}$ of primitive wall divisors of Nikulin-type orbifolds is given by:
$$\mathscr{W}_{M'}=\left\{D\in\Lambda\left|
    \begin{array}{lll}
				D^2=-2, & \div(D)=1, & \\
        D^2=-4, & \div(D)=2, &\\ 
				D^2=-6, & \div(D)=2, & \text{and}\\
				D^2=-12, & \div(D)=2 & \text{if\ }D_{U(2)^3}\text{\ is divisible by\ }2
    \end{array}
    \right.\right\},$$
    where $D_{U(2)^3}$ is the projection of $D$ to the $U(2)^3$-part of the lattice.
\end{thm}
The divisibility $\div$ is defined in Section \ref{notation} below.

\begin{remark}
  Although an isometry of $\Lambda$ does not necessarily preserve the decomposition
  $\Lambda=U(2)^3\oplus E_8(-1)\oplus(-2)^2$, the condition in the last line of Theorem \ref{main} is
  intrinsic. Indeed, let $\varepsilon_1,\varepsilon_2$ generate the two $(-2)$-summands and write
  $D=u+v+a\varepsilon_1+b\varepsilon_2$ according to this decomposition. The condition $\div(D)=2$ implies that
  $v\in 2E_8(-1)$. If $u\in 2U(2)^3$, then the equality $D^2=-12$, reduced modulo 8, implies that
  $a$ and $b$ are odd. Thus
  \[
    D_{U(2)^3}\in 2U(2)^3
    \quad\Longleftrightarrow\quad
    \frac{D}{2}\equiv\frac{\varepsilon_1+\varepsilon_2}{2}\pmod{\Lambda}.
  \]
  The class on the right is the characteristic element of the discriminant form of $\Lambda$, hence is
  fixed by every isometry of $\Lambda$.
\end{remark}

Combined with the global Torelli theorem (Theorem \ref{mainHTTO}), the previous theorem can be used for studying automorphisms on orbifolds of Nikulin-type. 
As an example of application, we construct a symplectic involution on orbifolds of Nikulin-type which is not induced by a symplectic involution on a Hilbert scheme of 2 points on a K3 surface (\emph{non-standard involution}) (see Section \ref{Application}).
\begin{prop}\label{main2}
Let $X$ be an irreducible symplectic orbifold of Nikulin-type such that there exists $D\in\Pic (X)$ with $q(D)=-2$ and $\div(D)=2$. Then, there exists an irreducible symplectic orbifold $Y$ bimeromorphic to $X$ and a symplectic involution $\iota$ on $Y$ such that:
$$H^2(Y,\Z)^{\iota}\simeq U(2)^3\oplus E_8(-1)\oplus (-2)\ \text{and}\ H^2(Y,\Z)^{\iota\bot}\simeq (-2).$$
\end{prop}
The proof of this Proposition is given in Section \ref{Application}.

For the proof of Theorem \ref{main} we need to show that the following two operators are monodromy operators.
The reflections $R_D$ on the second cohomology group are defined in Section \ref{notation} below.
\begin{prop}[{Compare Corollaries \ref{Sigma'}, \ref{cor:Chiara}, and \ref{Lastmonodromy}}]\label{MonoIntroduction}
\ 
  \begin{itemize}
 \item[(i)]
 The reflection $R_{\Sigma'}$ is a monodromy operator of $M'$.
\item[(ii)]
More generally,
let $X$ be an orbifold of Nikulin-type and $\alpha\in H^2(X,\Z)$ which satisfies one of the two numerical conditions:
\vspace{-0.2cm}
$$\left\{\begin{array}{ll}
q(\alpha)=-2 & \text{and}\ \div(\alpha)=2,\ \text{or}\\
q(\alpha)=-4 & \text{and}\ \div(\alpha)=2.
\end{array}\right.$$
Then $R_{\alpha}$ is a monodromy operator.
\end{itemize}
\end{prop}
\begin{rmk}
 Note that Proposition \ref{MonoIntroduction} (i) can also be obtained from the recent result of Lehn--Mongardi--Pacienza \cite[Theorem 3.10]{Lehn2}.
\end{rmk}

\subsection{Organization of the paper and sketch of the proof}
The paper is organized as follows. In Section \ref{reminders}, we provide some reminders related to irreducible symplectic orbifolds, especially from 
\cite{Menet-Riess-20} where the theory of the K\"ahler cone have been developed. In Section \ref{M'section0}, we provide some reminders about the orbifold $M'$
especially from \cite{Menet-2015}; moreover, we investigate the monodromy operators of $M'$ inherited from the ones on the Hilbert schemes of two points on K3 surfaces.
In Section \ref{genericM'0}, we determine the wall divisors of an orbifold $M'$ obtained from a very general K3 surface endowed with a symplectic involution $(S,i)$.
As a corollary, we can prove Proposition \ref{MonoIntroduction} (i). Our main tool to determine wall divisors is Proposition \ref{extremalray} which says that the dual divisor of an extremal ray of the cone of classes of effective curves (the \emph{Mori cone}) is a wall divisor.
The proof of Theorem \ref{main} is then divided in two parts. The first part (Section \ref{extremalcurves}) consists in finding enough examples of extremal rays of Mori cones in several different $M'$-orbifolds; the second part (Section \ref{sec:monodromy-orbits}) consists in using our knowledge on the monodromy group of $M'$ to show that we have find all possible wall divisors.
Finally, Section \ref{Application} is devoted to the proof of Proposition \ref{main2}.

\subsection{Notation and convention}\label{notation}
\begin{itemize}
\item
Let $\Lambda$ be a lattice of signature $(3,\rk \Lambda -3)$. Let $x\in \Lambda$ such that $x^2< 0$. We define the \emph{reflection} $R_x$ associated to $x$ by:
$$R_x(\lambda)=\lambda-\frac{2\lambda\cdot x}{x^2}x,$$
for all $\lambda\in\Lambda$.
\item
In $\Lambda$, we define the \emph{divisibility} of an element $x\in\Lambda$  as the integer $a\in \N^*$ such
that $x\cdot \Lambda =a\Z$. We denote by $\div(x)$ the divisibility of $x$.
	\item
	Let $X$ be a comapct orbifold of dimension $n$ and $C\subset X$ a compact curve. We denote by
	$\left[C\right]_{X}\in H^{2n-2}(X,\Q)$ the class of $C$ in $X$. 
\end{itemize}
~\\
\textbf{Acknowledgements}: 
We are very grateful to the Second Japanese-European Symposium on
Symplectic Varieties and Moduli Spaces where our collaboration was initiated. The first author has been
financed by the ERC-ALKAGE grant No. 670846 and by the PRCI SMAGP (ANR-20-CE40-0026-01).
The second author is a member of the Institute for Theoretical Studies at ETH Zürich
(supported by Dr.~Max R\"ossler, the Walter Haefner Foundation and the ETH Z\"urich
Foundation).
\section{Reminders on irreducible symplectic orbifolds}\label{reminders}
\subsection{Definition}\label{basicdef}
In this paper an \emph{orbifold} is a complex space with only quotient singularities.
\begin{defi}\label{def}
An irreducible symplectic orbifold (or hyperk\"ahler orbifold) is a compact K\"ahler orbifold $X$ such that:
\begin{itemize}
\item[(i)]
$\codim\Sing X\geq4$,
\item[(ii)]
$H^{2,0}(X)=\C \sigma_X$ with $\sigma_X$ non-degenerate on $X_{\reg}:=X\smallsetminus \Sing X$,
\item[(iii)]
$\pi_1(X_{\reg})=0$.
\end{itemize}
\end{defi}
We refer to \cite[Section 6]{Campana-2004}, \cite[Section 3.1]{Menet-2020}, \cite[Section 3.1]{Fu-Menet} and \cite[Section 2.1]{Menet-Riess-20} for discussions about this definition.
\begin{ex}[{\cite[Section 3.2]{Menet-2020}}]\label{exem}
Let $X$ be a hyperk\"ahler manifold deformation equivalent to a Hilbert scheme of 2 points on a K3 surface and $\iota$ a symplectic involution on $X$.
By \cite[Theorem 4.1]{Mongardi-2012}, $\iota$ has 28 fixed points and a fixed K3 surface $\Sigma$. 
We denote by $M'$ the blow-up of $X/\iota$ in the image of $\Sigma$. The orbifold $M'$ is irreducible
symplectic (see \cite[Proposition 3.8]{Menet-2020}).
\end{ex}
\begin{defi}\label{Nikulin}
An orbifold $M'$ constructed as before is called a \emph{Nikulin orbifold}. An irreducible symplectic orbifold deformation equivalent to a Nikulin orbifold is called an orbifold of \emph{Nikulin-type}. 
\end{defi}
\subsection{Moduli space of marked irreducible symplectic orbifolds}\label{per}
Let $X$ be an irreducible symplectic orbifold.
As explained in \cite[Section 3.4]{Menet-2020}, $H^2(X,\Z)$ is endowed with a quadratic form of signature $(3,b_2(X)-3)$ called the Beauville--Bogomolov form and denoted by $q_{X}$ (the bilinear associated form is denoted by $(\cdot,\cdot)_{q_X}$ or $(\cdot,\cdot)_{q}$ when there is no ambiguity). 
Let $\Lambda$ be a lattice of signature $(3,\rk \Lambda-3)$. We denote $\Lambda_{\mathbb{K}}:=\Lambda\otimes \mathbb{K}$ for $\mathbb{K}$ a field.
A \emph{marking} of $X$ is an isometry $\varphi: H^{2}(X,\Z)\rightarrow \Lambda$.
Let $\mathcal{M}_{\Lambda}$ be the set of isomorphism classes of marked irreducible symplectic orbifolds $(X,\varphi)$ with $\varphi:H^2(X,\Z)\rightarrow\Lambda$. As explained in \cite[Section 3.5]{Menet-2020}, this set can be endowed with a non-separated complex structure such that the \emph{period map}:
$$\xymatrix@R0cm@C0.5cm{\ \ \ \ \ \ \ \ \mathscr{P}:& \mathcal{M}_{\Lambda}\ar[r]& \mathcal{D}_{\Lambda}\\
&(X,\varphi)\ar[r]&\varphi(\sigma_X)}$$
is a local isomorphism with $\mathcal{D}_{\Lambda}:=\left\{\left.\alpha\in \mathbb{P}(\Lambda_{\C})\ \right|\ \alpha^2=0,\ \alpha\cdot\overline{\alpha}>0\right\}$. The complex manifold $\mathcal{M}_{\Lambda}$ is called \emph{the moduli space of marked irreducible symplectic orbifolds of Beauville--Bogomolov lattice $\Lambda$}.

Moreover there exists a \emph{Hausdorff reduction} of $\mathcal{M}_{\Lambda}$.
\begin{prop}[\cite{Menet-2020}, Corollary 3.25]
There exists a Hausdorff reduction $\overline{\mathcal{M}_{\Lambda}}$ of $\mathcal{M}_{\Lambda}$ such that the period map $\mathscr{P}$ factorizes through $\overline{\mathcal{M}_{\Lambda}}$:
$$\xymatrix{\mathcal{M}_{\Lambda}\ar@/^1pc/[rr]^{\mathscr{P}}\ar@{->>}[r]& \overline{\mathcal{M}_{\Lambda}}\ar[r]& \mathcal{D}_{\Lambda}.}$$
Moreover, two points in $\mathcal{M}_{\Lambda}$ map to the same point in $\overline{\mathcal{M}_{\Lambda}}$ if and only if they are non-separated in $\mathcal{M}_{\Lambda}$.
\end{prop}
\subsection{Global Torelli theorems}
\begin{thm}[\cite{Menet-2020}, Theorem 1.1]\label{mainGTTO}
Let $\Lambda$ be a lattice of signature $(3,b-3)$, with $b\geq3$. Assume that $\mathcal{M}_{\Lambda}\neq\emptyset$ and let $\mathcal{M}_{\Lambda}^{o}$ be a connected component of $\mathcal{M}_{\Lambda}$. Then the period map:
$$\mathscr{P}: \overline{\mathcal{M}_{\Lambda}}^{o}\rightarrow \mathcal{D}_{\Lambda}$$
is an isomorphism. 
\end{thm}
There also exists a Hodge version of this theorem, which we state in the following.
By a \emph{deformation family of irreducible symplectic orbifolds}, we mean a proper and flat morphism
$s:\mathcal{X}\rightarrow B$ over a connected analytic base $B$, all of whose fibers are irreducible
symplectic orbifolds. Such a family is locally trivial by \cite[Proposition 3.10]{Menet-2020}; in particular,
$R^2s_*\Z$ is a local system on $B$.
\begin{defi}\label{transp}
Let $X_1$ and $X_2$ be two irreducible symplectic orbifolds. An isometry $f:H^{2}(X_{1},\Z)\rightarrow H^{2}(X_{2},\Z)$ is called a \emph{parallel transport operator} if there exists a deformation family $s:\mathcal{X}\rightarrow B$, two points $b_{i}\in B$, two isomorphisms $\psi_{i}:X_{i}\rightarrow \mathcal{X}_{b_{i}}$, $i=1,2$ and a continuous path $\gamma:\left[0,1\right]\rightarrow B$ with $\gamma(0)=b_{1}$, $\gamma(1)=b_{2}$ and such that the parallel transport in the local system $R^2s_{*}\Z$ along $\gamma$ induces the morphism $\psi_{2*}\circ f\circ\psi_{1}^{*}: H^{2}(\mathcal{X}_{b_{1}},\Z)\rightarrow H^{2}(\mathcal{X}_{b_{2}},\Z)$.

Let $X$ be an irreducible symplectic orbifold.
If $f:H^{2}(X,\Z)\rightarrow H^{2}(X,\Z)$ is a parallel transport operator from $X$ to $X$ itself, $f$ is called a \emph{monodromy operator}. If moreover $f$ sends a holomorphic 2-form to a holomorphic 2-form, $f$ is called a \emph{Hodge monodromy operator}. We denote by $\Mon^2(X)$ the group of monodromy operators and by $\Mon^2_{\Hdg}(X)$ the group of Hodge monodromy operators.
\end{defi}
\begin{rmk}
If $(X,\varphi)$ and $(X',\varphi')$ are in the same connected component $\mathcal{M}_{\Lambda}^{o}$ of $\mathcal{M}_{\Lambda}$, then $\varphi^{-1}\circ\varphi'$ is a parallel transport operator.
Indeed, one can join the two points by a path in $\mathcal{M}_{\Lambda}^{o}$ and cover this path by finitely
many Kuranishi neighborhoods. By construction of $\mathcal{M}_{\Lambda}$, these neighborhoods carry local
universal deformation families and the markings are locally constant with respect to parallel transport.
Composing the resulting parallel transport operators along the path gives $\varphi^{-1}\circ\varphi'$.
\end{rmk}
\begin{thm}[\cite{Menet-Riess-20}, Theorem 1.1]\label{mainHTTO}
Let $X$ and $X'$ be two irreducible symplectic orbifolds.
\begin{itemize}
\item[(i)]
The orbifolds $X$ and $X'$ are bimeromorphic if and only if there exists a parallel transport operator $f:H^2(X,\Z)\rightarrow H^2(X',\Z)$ which is an isometry of integral Hodge structures.
\item[(ii)]
Let $f:H^2(X,\Z)\rightarrow H^2(X',\Z)$ be a parallel transport operator, which is an isometry of integral Hodge structures. There exists an isomorphism $\widetilde{f}:X\rightarrow X'$ such that $f=\widetilde{f}_*$ if and only if $f$ maps some K\"ahler class on $X$ to a K\"ahler class on $X'$.
\end{itemize}
\end{thm}

\subsection{K\"ahler cone} \label{Kahlersection}
Let $X$ be an irreducible symplectic orbifold of dimension $n$. We denote by $\mathcal{K}_X$ the K\"ahler cone of $X$. We denote by $\mathcal{C}_X$ the connected component of $\left\{\left.\alpha\in H^{1,1}(X,\R)\right|\ q_X(\alpha)>0\right\}$ which contains $\mathcal{K}_X$; it is called the \emph{positive cone}.
Let $\BK_X$ be the \emph{birational K\"ahler cone} which is the union $\cup f^{*}\mathcal{K}_{X'}$ for $f$ running through all birational maps between $X$ and any irreducible symplectic orbifold $X'$. In \cite[Definition 4.5]{Menet-Riess-20}, we define the wall divisors in the same way as Mongardi in \cite[Definition 1.2]{Mongardi13}.
\begin{defi}
Let $X$ be an irreducible symplectic orbifold and let $D\in\Pic(X)$. Then $D$ is called a \emph{wall divisor} if $q(D)<0$ and $g(D^{\bot})\cap \BK_X =\emptyset$, for all $g\in \Mon^2_{\Hdg}(X)$.
\end{defi}
We denote by $\mathscr{W}_X$ the set of primitive wall divisors of $X$ (non divisible in $\Pic X$).
By \cite[Corollary 4.8]{Menet-Riess-20}, we have the following theorem. 
\begin{thm}\label{wall}
Let $\Lambda$ be a lattice of signature $(3,\rk \Lambda -3)$ and $\mathcal{M}_{\Lambda}^{o}$ a connected component of the associated moduli space of marked irreducible symplectic orbifolds. Then there exists a set $\mathscr{W}_\Lambda\subset \Lambda$ such that for all $(X,\varphi)\in \mathcal{M}_{\Lambda}^{o}$:
 $$\mathscr{W}_X=\varphi^{-1}(\mathscr{W}_\Lambda)\cap H^{1,1}(X,\Z).$$
\end{thm}
\begin{defi}
The set $\mathscr{W}_\Lambda$ will be called the \emph{set of primitive wall divisor} of the deformation class of $X$.
\end{defi}
\begin{ex}[\cite{Mongardi13}, Proposition 2.12 and \cite{HuybrechtsK3}, Theorem 5.2 Chapter 8]\label{examplewall}
If $\mathcal{M}_{\Lambda}^{o}$ is a connected component of the moduli space of marked K3 surface, then:
$$\mathscr{W}_\Lambda=\left\{\left.D\in \Lambda\ \right|\ D^2=-2\right\}.$$
If $\mathcal{M}_{\Lambda}^{o}$ is a connected component of the moduli space of marked irreducible symplectic manifolds equivalent by deformation to a Hilbert scheme of 2 points on a K3 surface, then: $$\mathscr{W}_\Lambda=\left\{\left.D\in \Lambda\ \right|\ D^2=-2\right\}\cup\left\{\left.D\in \Lambda\ \right|\ D^2=-10\ \text{and}\ D\cdot\Lambda\subset2\Z \right\}.$$
\end{ex}
\begin{rmk}\label{dualclass}
Let $\beta\in H^{n-1,n-1}(X,\Q)$.
We can associate to $\beta$ its \emph{dual class} $\beta^{\vee}\in H^{1,1}(X,\Q)$ defined as follows. By \cite[Corollary 2.7]{Menet-2020} and since the Beauville--Bogomolov form is integral and non-degenerate (see \cite[Theorem 3.17]{Menet-2020}), we can find $\beta^{\vee}\in H^{2}(X,\Q)$ such that for all $\alpha\in H^2(X,\C)$:
 $$(\alpha,\beta^{\vee})_q=\alpha\cdot \beta,$$ 
where the dot on the right hand side is the cup product. Since $(\beta^{\vee},\sigma_X)_q=\beta\cdot \sigma_X=0$, we have $\beta^{\vee}\in H^{1,1}(X,\Q)$.
\end{rmk}

We also define the \emph{Mori cone} as the cone of classes of effective curves in $H^{n-1,n-1}(X,\Z)$.
\begin{prop}[\cite{Menet-Riess-20}, Proposition 4.12]\label{extremalray}
Let $X$ be an irreducible symplectic orbifold. Let $R$ be an extremal ray of the Mori cone of $X$ of negative self-intersection. Then the line $\Q R^{\vee}$ contains a primitive wall divisor; equivalently, $R^{\perp}\subset H^2(X,\Z)$ is the orthogonal complement of a wall divisor.
\end{prop}
It induces a criterion for K\"ahler classes.
\begin{defi}Given an irreducible symplectic orbifold $X$ endowed with
  a K\"ahler class $\omega$.
    Define $\cW_X^+\coloneqq \{D\in \cW_X\,|\, (D,\omega)_q>0\}$, i.e.~for every wall divisor, we choose the
  primitive representative in its line, which pairs positively with the K\"ahler cone.
\end{defi}
\begin{cor}[\cite{Menet-Riess-20}, Corollary 4.14]\label{criterionwall}\label{cor:desrK}
  Let $X$ be an irreducible symplectic orbifold.
  Then
  $$
  \dK_X=\{\alpha\in \dC_X \,|\, (\alpha, D)_q>0\ \forall D\in \cW_X^+\}.$$
\end{cor}
Finally, we recall the following proposition about the birational K\"ahler cone.
\begin{prop}[\cite{Menet-Riess-20}, Corollary 4.17]\label{caca}
Let $X$ be an irreducible symplectic orbifold. Then $\alpha\in H^{1,1}(X,\R)$ is in the closure $\overline{\BK}_X$ of the birational K\"ahler cone $\BK_X$ if and only if $\alpha\in\overline{\mathcal{C}}_X$ and $(\alpha,[D])_q\geq 0$ for all uniruled divisors $D\subset X$.
\end{prop}
\subsection{Twistor space}\label{Twist}
Let $\Lambda$ be a lattice of signature $(3,\rk-3)$. We denote by "$\cdot$" its bilinear form.
A \emph{positive three-space} is a 3-dimentional subspace $W\subset \Lambda\otimes\R$ such that $\cdot_{|W}$ is positive
definite. For any positive three-space, we define the associated \emph{twistor line} $T_W\subset \mathcal{D}_{\Lambda}$ by:
$$T_W:=\mathcal{D}_{\Lambda}\cap \mathbb{P}(W\otimes\C).$$
A twistor line is called \emph{generic} if $W^{\bot}\cap \Lambda=0$. A point of $\alpha\in \mathcal{D}_{\Lambda}$ is called \emph{very general} if $\alpha^{\bot}\cap \Lambda=0$.
\begin{thm}[\cite{Menet-2020}, Theorem 5.4]\label{Twistor}
  Let $(X,\varphi)$ be a marked irreducible symplectic orbifold with $\varphi:H^2(X,\Z)\rightarrow \Lambda$. Let $\alpha$ be a K\"ahler class on $X$, and $W_\alpha\coloneqq\Vect_{\R}(\varphi(\alpha),$ $\varphi(\Rea \sigma_X),\varphi(\Ima \sigma_X))$. 
Then:
\begin{itemize}
\item[(i)]
There exists a metric $g$ and three complex structures (see \cite[Section 5.1]{Menet-2020} for the definition) $I$, $J$ and $K$ in quaternionic relation on $X$ such that:
$$\alpha= \left[g(\cdot,I\cdot)\right]\ \text{and}\ g(\cdot,J\cdot)+ig(\cdot,K\cdot)\in H^{0,2}(X).$$
\item[(ii)]
Let
$$T(\alpha):=\left\{\left.(a,b,c)\in\R^3\ \right|\ a^2+b^2+c^2=1\right\}\simeq\mathbb{P}^1$$
be the twistor sphere associated with $\alpha$, parametrizing the complex structures $aI+bJ+cK$.
There exists a deformation of $X$
$$\mathscr{X}\rightarrow T(\alpha)$$
such that the period map $\mathscr{P}:T(\alpha)\rightarrow T_{W_\alpha}$ provides an isomorphism. Moreover,
for each $s=(a,b,c)\in T(\alpha)$, the associated fiber $\mathscr{X}_s$ is an orbifold diffeomorphic to $X$
endowed with the complex structure $aI+bJ+cK$.
\end{itemize}
\end{thm}
\begin{rmk}
Note that if the irreducible symplectic orbifold $X$ of the previous theorem is endowed with a marking then
all the fibers of $\mathscr{X}\rightarrow T(\alpha)$ are naturally endowed with a marking. Therefore, the period map $\mathscr{P}:T(\alpha)\rightarrow T_{W_\alpha}$ is well-defined.
\end{rmk}
\begin{rmk}\label{twistorinvo}
Let $X$ be an irreducible symplectic orbifold endowed with a finite group $G$ of symplectic automorphisms (i.e. $G$ fixes the holomorphic 2-form of $X$). Let $\alpha$ be a K\"ahler class of $X$ fixed by $G$ and $\mathscr{X}\rightarrow T(\alpha)$ the associated twistor space. Then $G$ extends to a group of automorphisms of $\mathscr{X}$ and restricts on each fiber to a group of symplectic automorphisms. Indeed, since $G$ is symplectic, it fixes all the complex structures $I$, $J$, and $K$.
\end{rmk}

We provide the following lemma which will be used several times in this paper. It is a generalization of \cite[Lemma 2.17]{Menet-Riess-20}.
\begin{lemme}\label{lem:connected+}
  Let $\Lambda'\subseteq \Lambda$ be a sublattice of rank $b'$, which also has signature $(3,b'-3)$.
  Consider the inclusion of period domains
$    \dD_{\Lambda'} \subseteq\dD_\Lambda$.
  Suppose that a very general point $(\widetilde{X},\widetilde{\varphi})\in \cM_\Lambda \cap
  \cP^{-1}\dD_{\Lambda'}$
  (i.e.~$(\widetilde{X},\widetilde{\varphi})$ with $\widetilde{\varphi}(\Pic(\widetilde{X}))^\perp=\Lambda'$) satisfies that
  $\dK_{\widetilde{X}}=\dC_{\widetilde{X}}$.
  Let $(X,\varphi)$ and  $(Y,\psi)\in \mathcal{M}_{\Lambda}^{\circ}$ be any two marked irreducible symplectic
  orbifolds which satisfy $\cP(X,\varphi)\in \dD_{\Lambda'}$ and $\cP(Y,\psi)\in \dD_{\Lambda'}$ and for which $\varphi(\dK_X)\cap
  \Lambda'_{\R} \neq \emptyset$ and $\psi(\dK_Y)\cap \Lambda'_{\R}\neq \emptyset$. Then $(X,\varphi)$ and $(Y,\psi)$ can
  be connected by a sequence of twistor spaces whose associated twistor lines are contained in $\dD_{\Lambda'}$ and are generic with respect to $\Lambda'$.
	That is: there exists a sequence of such twistor spaces $f_i:\mathscr{X}_i\rightarrow \mathbb{P}^1\simeq T(\alpha_i)$ with $(x_i,x_{i+1})\in\mathbb{P}^1\times\mathbb{P}^1$, $i\in \left\{0,...,k\right\}$, $k\in \mathbb{N}$ such that:
	\begin{itemize}
	\item
$f^{-1}_0(x_0)\simeq (X,\varphi),\ f^{-1}_i(x_{i+1})\simeq f^{-1}_{i+1}(x_{i+1})\ \text{and}\ f^{-1}_{k}(x_{k+1})\simeq (Y,\psi),$
for all $0\leq i\leq k-1$. 
\item
$\mathscr{P}(T(\alpha_i))\subset\dD_{\Lambda'}$ for all $0\leq i\leq k$.
\end{itemize}
\end{lemme}
\begin{proof}
We split the proof in two steps.
  
\textbf{First case}: We assume that $(X,\varphi)$ and $(Y,\psi)$ are very general in $\cM_\Lambda\cap
  \cP^{-1}\dD_{\Lambda'}$ (hence $\dC_X=\dK_X$ and $\dC_Y=\dK_Y$). 
By \cite[Proposition 3.7]{Huybrechts12}  the period domain $\dD_{\Lambda'}$
is connected by generic twistor lines. Note that the proof of \cite[Proposition 3.7]{Huybrechts12} in fact shows
that the twistor lines can be chosen in a such a way that they intersect in very general points of $\dD_{\Lambda'}$.
In particular, we can connect  $\mathscr{P}(Y,\psi )$ and $\mathscr{P}(X,\varphi)$ by
such generic twistor lines in $\dD_{\Lambda'}$.
  Since for a very general element $(\widetilde{X},\widetilde{\varphi})$ of $\cM_\Lambda\cap
  \cP^{-1}\dD_{\Lambda'}$ we know
  $\dK_{\widetilde{X}}=\dC_{\widetilde{X}}$, Theorem \ref{Twistor} shows that all these twistor line can be lifted to
twistor spaces. Moreover, by Theorem \ref{mainHTTO} (ii) the period map $\cP$ is injective on the set of points $(\widetilde{X},\widetilde{\varphi})\in\cM_\Lambda$ such that $\dK_{\widetilde{X}}=\dC_{\widetilde{X}}$. Therefore, all these twistor spaces intersect and connect $(X,\varphi)$ to $(Y,\psi)$.

\textbf{Second case}: If $(X,\varphi)$ is not very general, we consider 
a very general K\"ahler class $\alpha\in\dK_X \cap \Lambda'_{\R}\neq\emptyset$. Then the associated twistor space $\mathscr{X}\rightarrow T(\alpha)$
have a fiber which is a very general marked irreducible symplectic orbifold in $\cM_\Lambda\cap
  \cP^{-1}\dD_{\Lambda'}$. Hence we are back to the first case.

\end{proof}
\section{The Nikulin orbifolds}\label{M'section0}
\subsection{Construction and description of Nikulin orbifolds}\label{M'section}
In order to enhance the readability, we recall the construction of the Nikulin orbifold from Example \ref{exem} and Definition \ref{Nikulin}.
Let $X$ be a (smooth) irreducible symplectic 4-fold deformation equivalent to the Hilbert scheme of two points
on a K3 surface (called \emph{manifold of $K3^{[2]}$-type}). 
Suppose that $X$ admits a symplectic involution $\tilde{\iota}$. 
By \cite[Theorem 4.1]{Mongardi-2012}, $\tilde{\iota}$ has 28 fixed points and a fixed K3 surface $\Sigma$. We
define $M:=X/\tilde{\iota}$ the quotient and $r:M'\rightarrow M$ the blow-up in the image of $\Sigma$.
As mentioned in Example \ref{exem}, the orbifolds $M'$ constructed in this way are irreducible symplectic orbifolds (see \cite[Proposition
  3.8]{Menet-2020}) and are named \emph{Nikulin orbifolds}.

A concrete example of such $X$ can be obtained in the following way: let $S$ be a K3 surface endowed with a symplectic involution $\iota$. It induces a symplectic involution $\iota^{[2]}$ on $S^{[2]}$ the Hilbert scheme of two points on $S$.
Then the fixed surface $\Sigma$ of $\iota^{[2]}$ is the following:
\begin{equation}\label{eq:sigma}
  \Sigma=\left\{ \left.\xi\in S^{[2]}\ \right|\ \Supp \xi=\left\{s,\iota(s)\right\}, s\in S\right\}.
\end{equation}

\begin{remark}\label{rem:Sigma}
Let us describe this surfaces $\Sigma$:
consider as usual $S\times S \overset{\widetilde{\nu}}{\longleftarrow} \widetilde{S\times S}
\overset{\widetilde{\rho}}{\too} S^{[2]}$, where $\nutild$ is the blow-up of the diagonal $\Delta_S \subseteq S\times S$, and
$\rhotild$ the double cover induced by permutation of the two factors.
Consider the surface $S_\iota \coloneqq \{(s,\iota(s))\,|\,s\in S\}\subseteq S\times S$, which is preserved by
the involution $\iota\times\iota$. Restricted to $S_\iota$ the permutation of the two factors in $S\times S$
corresponds to the action of $\iota$ on $S$ (via the isomorphism $S_\iota \iso S$ induced by the first
projection), and thus $S_\iota \cap \Delta_S$ corresponds to the fixed points of $\iota$ in $S$. Therefore, the strict transform
$\widetilde{S_{\iota}}$ of $S_\iota$ is isomorphic to the blow-up $\Bl_{\Fix\iota}S$
of $S$ in the fixed points of $\iota$.
Denote
\begin{equation*}
  \Sigma\coloneqq\rhotild(\widetilde{S_\iota})\iso \Bl_{\Fix\iota}S / \overline{\iota},
\end{equation*}
where $\overline\iota$ is the involution on $\Bl_{\Fix\iota}S$ which is induced by $\iota$.
Then $\Sigma$ is a K3 surface, which is fixed by $\iota^{[2]}$ and admits the description in \eqref{eq:sigma}
by construction. 
\end{remark}

\bigskip

Note that the existence of a symplectic involution on a K3 surface or on a $K3^{[2]}$-type manifold can be checked purely on the
level of lattices.
We will need the following lemma.
\begin{lemme}\label{pfff}
Let $X$ be a K3 surface or an irreducible symplectic manifold of $K3^{[2]}$-type.
  Assume that there is a primitive embedding $E_8(-2)\hookrightarrow\Pic X$. Then there is no wall divisor
  in $E_8(-2)$.
In particular under the additional assumption that $\Pic X\simeq E_8(-2)$, then $\mathcal{C}_X=\mathcal{K}_X$.
\end{lemme}
\begin{proof}
All elements of $E_8(-2)$ are of square divisible by 4. 
Hence by Example \ref{examplewall},
$E_8(-2)$ cannot contain any wall divisor. Then the lemma follows from Corollary \ref{cor:desrK}.
\end{proof}
\begin{prop}\label{involutionE8}
  Let $X$ be a K3 surface or a manifold of $K3^{[2]}$-type.
  Then there exists a symplectic involution $\iota$ on $X$ if and only if $X$ satisfies the following
  conditions:
  \begin{compactenum}[(i)]
  \item \label{it1iota} There exists a primitive embedding $E_8(-2)\hookrightarrow\Pic X$.
  \item \label{it2iota} The intersection $\dK_X \cap E_8(-2)^\perp\neq \emptyset$.
  \end{compactenum}
  In this case the pullback $\iota^*$ to $H^2(X,\bZ)$ acts 
    on $E_8(-2)$ as $-\id$ and trivially on $E_8(-2)^{\bot}$.
\end{prop}
\begin{proof}
  Let us start by fixing a symplectic involution $\iota$. The identifications of the fixed and anti-invariant
  lattices of $\iota$ with $E_8(-2)^\perp$ and $E_8(-2)$, respectively, are shown in
  \cite[Section 1.3]{Sarti-VanGeemen} and \cite[Theorem 5.2]{Mongardi-2012}. This readily implies (i).
  To observe (ii), pick any K\"ahler class $\alpha \in \dK_X$. Since $\iota$ is an isomorphism, $\iota^*(\alpha)$
  is also a K\"ahler class. Therefore $\alpha + \iota^*(\alpha)\in E_8(-2)^\perp$ is a K\"ahler class.

  For the other implication assume (i) and (ii). We consider the involution $i$ on $E_8(-2)\oplus E_8(-2)^{\bot}$ defined by $-\id$ on $E_8(-2)$ and $\id$ on $E_8(-2)^{\bot}$.
By \cite[Corollary 1.5.2]{Nikulin}, $i$ extends to an involution on $H^2(X,\Z)$. 

By \cite[Section 9.1.1]{Markman11}, $i$ is a monodromy operator.
Moreover, by (ii), we can find a K\"ahler class of $X$ in $E_8(-2)^{\bot}$. It follows from the global
Torelli theorem (see \cite[Theorem 1.3 (2)]{Markman11} or Theorem \ref{mainHTTO} (ii)) that there exists a symplectic automorphism
$\iota$
on
$X$ such that 
$\iota^*=i$. However by \cite[Propositions 10]{Beauville1982}, we know that the natural map $\Aut(X)\rightarrow \mathcal{O}(H^2(X,\Z))$ is an injection. Hence $\iota$ is necessarily an involution.
\end{proof}
\begin{remark}\label{rem:invariant-lattice-K3}
Let $S$ be a K3 surface endowed with a symplectic involution $\iota$. One can choose an identification
\[
H^2(S,\bZ)\simeq U^3\oplus E_8(-1)\oplus E_8(-1)
\]
under which $\iota^*$ acts trivially on $U^3$ and exchanges the two copies of $E_8(-1)$. Consequently,
\[
H^2(S,\bZ)^{\iota^*}
=U^3\oplus\{(e,e)\mid e\in E_8(-1)\}
\simeq U^3\oplus E_8(-2),
\]
whereas
\[
\left(H^2(S,\bZ)^{\iota^*}\right)^{\bot}
=\{(e,-e)\mid e\in E_8(-1)\}
\simeq E_8(-2).
\]
\end{remark}

\begin{remark} Fix a primitive embedding of $E_8(-2)$ in the K3$^{[2]}$-lattice $\Lambda\coloneqq U^3 \oplus E_8(-1)^2 \oplus (-2)$.
  Let $\cM_{\rm K3^{[2]}}^\iota$ be the moduli space of marked K3$^{[2]}$-type manifolds endowed with a
  symplectic involution such that the anti-invariant lattice is identified with the chosen $E_8(-2)$.
  Denote by $\Lambda^{{\iota}}\iso U^3\oplus E_8(-2)\oplus (-2)$  the orthogonal complement of $E_8(-2)$.
  From Proposition \ref{involutionE8} we observe that the  period map restricts to 
  $$\cP^\iota \colon
  \cM_{\rm K3^{[2]}}^\iota \to
  \mathcal{D}_{\Lambda^{\iota}}:=\left\{\left.\sigma\in \mathbb{P}(\Lambda^{\iota}\otimes\C)\ \right|\ \sigma^2=0,\ \sigma\cdot\overline{\sigma}>0\right\}.$$
Fix a period $\sigma\in\mathcal{D}_{\Lambda^\iota}$. By the global Torelli theorem, the points in the fiber
$(\cP^\iota)^{-1}(\sigma)$ correspond to the K\"ahler-type chambers in the positive cone contained in
$\sigma^\perp\cap\Lambda^\iota_{\R}$: marked manifolds with the same period are bimeromorphic, and they are
isomorphic precisely when the induced Hodge isometry maps a K\"ahler class to a K\"ahler class. These chambers
are cut out by the hyperplanes orthogonal to wall divisors. Moreover, if $D$ is a wall divisor and
$x\in\Lambda^\iota_{\R}$, then
$$
  (x,D)=(x,\operatorname{pr}_{\Lambda^\iota}(D)).
$$
Hence the induced walls in $\Lambda^\iota_{\R}$ are the hyperplanes orthogonal to the images of wall divisors
under the orthogonal projection
$\operatorname{pr}_{\Lambda^\iota}:\Lambda_{K3^{[2]}}\rightarrow\Lambda^\iota$. By Lemma \ref{pfff}, no wall
divisor is contained in $(\Lambda^\iota)^\perp\simeq E_8(-2)$, so these projections are non-zero.
\end{remark}

\subsection{The lattice of Nikulin orbifolds starting from $S^{[2]}$}\label{M'S2}
From now on we restrict ourselves to the case $X=S^{[2]}$ for a suitable K3 surface $S$ with an involution $\iota$.
We consider the following commutative diagram: 
\begin{equation}
\xymatrix{
  \ar@(dl,ul)[]^{\iota^{[2]}_1}N_1\ar[r]^{r_1}\ar[d]^{\pi_1}&\ar[d]^{\pi}S^{[2]}\ar@(dr,ur)[]_{\iota^{[2]}}\\
 M' \ar[r]^{r} & M,} 
\label{diagramM'}
\end{equation}
where $\pi : S^{[2]}\longrightarrow S^{[2]}/\iota^{[2]}=:M$ is the quotient map, $r_1$ is the blow-up in $\Sigma$ of $S^{[2]}$,  
$\iota^{[2]}_1$ is the involution induced by $\iota^{[2]}$ on $N_1$, $\pi_1 : N_1\longrightarrow
N_1/\iota^{[2]}_1\simeq M'$ is the quotient map and $r$ is the blow-up in $\pi(\Sigma)$ of $M$. We also denote
by $j:H^2(S,\Z)\hookrightarrow H^2(S^{[2]},\Z)$ the natural Hodge isometric embedding (see \cite[Proposition 6 Section 6 and Remark 1 Section 9]{Beauville1983}.

We fix the following notation for important divisors:
\begin{itemize}
\item
$\Delta$ the class of the diagonal divisor in $S^{[2]}$ and $\delta:=\frac{1}{2}\Delta$;
\item
$\delta_1:=r_1^*(\delta)$ and $\Sigma_1$ the exceptional divisor of $r_1$;
\item
$\delta':=\pi_{1*}r_1^*(\delta)$ and $\Sigma':=\pi_{1*}(\Sigma_1)$ the exceptional divisor of $r$.
\end{itemize} 

Here we use the definition of the push-forward given in \cite{Aguilar-Prieto}. In particular, $\pi_*$ satisfies the following equations (see \cite[Theorem 5.4 and Corollary 5.8]{Aguilar-Prieto}):
\begin{equation}
\pi_*\circ\pi^*=2\id\ \text{and}\ \pi^*\circ\pi_*=\id+\iota^{[2]*}.
\label{Smith}
\end{equation}
As a consequence, we have (see \cite[Lemma 3.6]{Menet-2018}):
\begin{equation}
\pi_*(\alpha)\cdot\pi_*(\beta)=2\alpha\cdot \beta,
\label{Smith2}
\end{equation}
with
$\alpha \in H^k(S^{[2]},\Z)^{\iota^{[2]}}$ and $\beta \in H^{8-k}(S^{[2]},\Z)^{\iota^{[2]}}$, $k\in \left\{0,...,8\right\}$.
Of course, the same equations are also true for $\pi_{1*}$.

\begin{rmk}\label{Smithcomute}
Note that the commutativity of diagram (\ref{diagramM'}) and equations (\ref{Smith}) imply
$\pi_{1*}r_1^*(x)=r^*\pi_*(x)$ for all $x\in H^{2}(S^{[2]},\Z)$. Indeed,
\begin{align*}
\pi_1^*\pi_{1*}r_1^*(x)
  &=\left(\id+\iota_1^{[2]*}\right)r_1^*(x)\\
  &=r_1^*\left(\id+\iota^{[2]*}\right)(x)\\
  &=r_1^*\pi^*\pi_*(x)\\
  &=\pi_1^*r^*\pi_*(x).
\end{align*}
Here the second equality follows from the equivariance of $r_1$, and the last one from the commutativity of
diagram (\ref{diagramM'}). Since $\pi_1^*:H^2(M',\Z)\rightarrow H^2(N_1,\Z)$ is injective, the claimed
identity follows.
\end{rmk}
We denote by $q_{M'}$ and $q_{S^{[2]}}$ respectively the Beauville--Bogomolov form of $M'$ and $S^{[2]}$. We can also define a Beauville--Bogomolov form on $M$ by:
$$q_{M}(x):=q_{M'}(r^*(x)),$$
for all $x\in H^2(M,\Z)$. We recall the following theorem.
\begin{thm}\label{BBform}
\begin{itemize}
\item[(i)]
The Beauville--Bogomolov lattice of $M'$ is given by
$(H^2(M',\Z),q_{M'})\simeq U(2)^3\oplus E_8(-1)\oplus(-2)^2$ where the Fujiki constant is equal to 6.
\item[(ii)]
$q_{M}(\pi_*(x))=2q_{S^{[2]}}(x)$ for all $x\in H^2(S^{[2]},\Z)^{\iota^{[2]}}$.
\item[(iii)]
$q_{M'}(\delta')=q_{M'}(\Sigma')=-4$.
\item[(iv)]
$(r^*(x),\Sigma')_{q_{M'}}=0$ for all $x\in H^{2}(M,\Z)$.
\item[(v)]
$H^2(M',\Z)=r^*\pi_*(j(H^2(S,\Z)))\oplus^{\bot}\Z\frac{\delta'+\Sigma'}{2}\oplus^{\bot}\Z\frac{\delta'-\Sigma'}{2}$.
\end{itemize}
\end{thm}
\begin{proof}
This theorem corresponds to several results in \cite{Menet-2015}. We want to emphasize that our notation is slightly different from that of \cite{Menet-2015}.
In \cite{Menet-2015}, $r_1$, $r$, $\delta'$, and $\Sigma'$ are respectively denoted by $s_1$, $r'$, $\overline{\delta}'$, and $\overline{\Sigma}'$.
Statement (i) is \cite[Theorem 2.5]{Menet-2015}. Knowing that the Fujiki constant is equal to 6 and Remark \ref{Smithcomute}, statement (ii) is \cite[Proposition 2.9]{Menet-2015}.
Similarly, statements (iii) and (iv) are respectively given by \cite[Propositions 2.10 and 2.13]{Menet-2015}. Finally, statement (v) is provided by \cite[Theorem 2.39]{Menet-2015}.
\end{proof}
\begin{rmk}
In the previous theorem the Beauville--Bogomolov lattice of $M'$ is obtained as follows: 
\begin{itemize}
\item
$r^*\pi_*(j(H^2(S,\Z)))\simeq U(2)^3\oplus E_8(-1)$,
\item
$\Z\frac{\delta'+\Sigma'}{2}\oplus^{\bot}\Z\frac{\delta'-\Sigma'}{2}
\simeq (-2)^2.$
\end{itemize}
\end{rmk}
We recall that the divisibility $\div$ of a lattice element is defined in Section \ref{notation}.
\begin{rmk}\label{div}
Theorem \ref{BBform} shows that $\div(\Sigma')=\div(\delta')=2$.
\end{rmk}
\subsection{Monodromy operators inherited from $\Mon^2(S^{[2]})$}
We keep the notation from the previous subsection.
The monodromy group is defined in Section \ref{notation}.
\begin{prop}\label{MonoM'}
Let $f\in \Mon^2(S^{[2]})$, (resp. $f\in \Mon^2_{\Hdg}(S^{[2]})$) be a monodromy operator such that
$f\circ\iota^{[2]*}=\iota^{[2]*}\circ f$ on $H^2(S^{[2]},\bZ)$.
We consider $f':H^2(M',\Z)\rightarrow H^2(M',\Z)$ such that $f'(\Sigma')=\Sigma'$ and:
$$f'(r^*(x))=\frac{1}{2}r^*\circ \pi_*\circ f \circ \pi^*(x),$$
 for all $x\in H^2(M,\Z)$. Then $f'\in \Mon^2(M')$, (resp. $f'\in \Mon^2_{\Hdg}(M')$).
\end{prop}
\begin{proof}
Let $\varphi$ be a marking of $S^{[2]}$. Since $f$ is a monodromy operator, we know that $(S^{[2]},\varphi)$ and $(S^{[2]},\varphi\circ f)$ are in the same connected component of their moduli space (see Section \ref{per} for the definition of the moduli space). We consider
$$\Lambda^{\iota^{[2]}}:=\varphi\left(H^{2}(S^{[2]},\Z)^{\iota^{[2]}}\right).$$
We know that $\Lambda^{\iota^{[2]}}\simeq U^3\oplus E_8(-2)\oplus (-2)$ which is a lattice of signature
$(3,12)$ (see for instance \cite[Proposition 2.6]{Menet-2015}). As in Section \ref{M'section}, we can consider the associated period domain:
$$\mathcal{D}_{\Lambda^{\iota^{[2]}}}:=\left\{\left.\sigma\in
\mathbb{P}(\Lambda^{\iota^{[2]}}\otimes\C)\ \right|\ \sigma^2=0,\ \sigma\cdot\overline{\sigma}>0\right\}.$$
By Lemma \ref{pfff}, a very general K3$^{[2]}$-type manifold mapping to  $\mathcal{D}_{\Lambda^{\iota^{[2]}}}$ satisfies that the K\"ahler cone is the
entire positive cone. Furthermore, by Proposition \ref{involutionE8} \eqref{it2iota} the intersection
$\varphi(\dK_{S^{[2]}})\cap\Lambda^{\iota^{[2]}}_{\R}\neq \emptyset$ and therefore also $\varphi\circ f(\dK_{S^{[2]}})\cap
\Lambda^{\iota^{[2]}}_{\R}\neq \emptyset$ is non-empty. We can apply Lemma \ref{lem:connected+} to see that  $(S^{[2]},\varphi)$ and $(S^{[2]},\varphi\circ f)$ can be connected by a sequence of twistor spaces $\mathscr{X}_i\rightarrow \mathbb{P}^1$. By construction and Remark \ref{twistorinvo}, all these twistor spaces are endowed with an involution $\mathscr{I}_i$ which restricts on each fiber to a symplectic involution. Hence we can consider for each twistor space the blow-up $\widetilde{\mathscr{X}_i/\mathscr{I}_i}\rightarrow \mathscr{X}_i/\mathscr{I}_i$ of the quotient $\mathscr{X}_i/\mathscr{I}_i$ in the codimension 2 component of its singular locus. We obtain $\widetilde{\mathscr{X}_i/\mathscr{I}_i}\rightarrow\mathbb{P}^1$ a sequence of families of orbifolds deformation equivalent to $M'$. This sequence of families provides a monodromy operator of $M'$ that we denote by $f'$.

We need to verify that $f'$ satisfies the claimed properties. First note that by construction $f'(\Sigma')=\Sigma'$.
For each of the twistor families used above, we choose a differentiable trivialization which is equivariant
with respect to the involution; such a trivialization can be obtained by choosing an invariant connection.
Composing the corresponding parallel transport maps gives a diffeomorphism
$u:S^{[2]}\rightarrow S^{[2]}$ such that $u^*=f$ and
$u\circ\iota^{[2]}=\iota^{[2]}\circ u$.

The equivariance of $u$ implies that the formula
$\overline{u}(\pi(x)):=\pi(u(x))$ defines a homeomorphism $\overline{u}:M\rightarrow M$. Moreover, $u$
preserves the unique two-dimensional component $\Sigma$ of the fixed locus of $\iota^{[2]}$, so
$\overline{u}$ preserves $\pi(\Sigma)$. Since the blow-ups of the quotients were performed fiberwise along
this component, the same equivariant parallel transport induces a homeomorphism
$\overline{u}':M'\rightarrow M'$ satisfying $r\circ\overline{u}'=\overline{u}\circ r$. We therefore obtain
the following commutative diagram:
$$\xymatrix{S^{[2]}\ar[r]^{u}\ar[d]^{\pi}& S^{[2]}\ar[d]^{\pi} \\
M\ar[r]^{\overline{u}}&M \\
M'\ar[u]^{r}\ar[r]^{\overline{u}'} & M'.\ar[u]^{r}}$$
Note that, by construction $f'=\overline{u}'^*$. We can use the commutativity of the previous diagram to check
that $f'$ satisfies the properties from the proposition.
Let $x\in H^2(M,\Z)$. We have:
\begin{equation} 
f'(r^*(x))=\overline{u}'^*(r^*(x))=r^*(\overline{u}^*(x)).
\label{calculf}
\end{equation}
Moreover:
$$\pi^*(\overline{u}^*(x))=u^*(\pi^*(x)).$$
Taking the image by $\pi_*$ and using \eqref{Smith}, we obtain that:
$$2\overline{u}^*(x)=\pi_*u^*\pi^*(x).$$
Combining this last equation with (\ref{calculf}), we obtain the statement 
of the proposition.

It is only left to prove that if $f\in\Mon^2_{\Hdg}(S^{[2]})$ then also $f'\in \Mon^2_{\Hdg}(M')$.
The maps $\pi$ and $r$ are holomorphic maps between K\"ahler orbifolds, hence induce morphisms $\pi^*$ and $r^*$
which respect the Hodge structure. Then $\pi_*$ respects the Hodge structure because of (\ref{Smith}). Since
$f'$ is a composition of morphisms which respect the Hodge structure, we therefore obtain that $f'\in \Mon^2_{\Hdg}(M')$.
\end{proof}
\begin{rmk}
The argument of the previous proposition applies more generally under the following assumptions. Let $X$ be
an irreducible symplectic manifold endowed with a finite symplectic group action by $G$, and let $Y$ be a
partial resolution in codimension 2 of $X/G$. Assume that a monodromy operator
$f\in\Mon^2(X)$ commuting with the action of $G$ can be realized by a $G$-equivariant deformation family, and
that the chosen partial resolution of the quotient extends fiberwise over this family. Then the construction
above induces a monodromy operator on $Y$. 
\end{rmk}
\begin{cor}\label{Rdelta}
The reflection $R_{\delta'}$ as defined in Section \ref{notation} is an element of the monodromy group $\Mon^2_{\Hdg}(M')$.
\end{cor}
\begin{proof}
By \cite[Section 9]{Markman11}, we know that $R_{\delta}\in \Mon^2(S^{[2]})$. By Proposition \ref{MonoM'} and Theorem \ref{BBform} (iv) , we only have to check that:
$$R_{\delta'}(r^*(x))=\frac{1}{2}r^*\circ \pi_*\circ R_{\delta} \circ \pi^*(x),$$
for all $x\in H^{2}(M,\Z)$. We have:
$$R_{\delta} \circ \pi^*(x)=\pi^*(x)-\frac{2(\delta,\pi^*(x))_{q_{S^{[2]}}}}{q_{S^{[2]}}(\delta)}\delta.$$
Taking the image by $\pi_*$, applying (\ref{Smith}) and Theorem \ref{BBform} (ii), we obtain:
\begin{align*}
\pi_*\circ R_{\delta} \circ \pi^*(x)=2x-\frac{4(\pi_*(\delta),2x)_{q_M}}{2q_{M}(\pi_*\delta)}\pi_*(\delta)
=2\left(x-\frac{2(\pi_*(\delta),x)_{q_M}}{q_M(\pi_*\delta)}\pi_*(\delta)\right).
\end{align*}
Then dividing by 2, taking the image by $r^*$, and using $q_M=q_{M'}\circ r^*$ (compare Section \ref{M'S2}) concludes the computation.
\end{proof}

\section{A first example: the very general Nikulin orbifolds}\label{genericM'0}
\subsection{Wall divisors of a Nikulin orbifold constructed from a K3 surface without effective curves}\label{genericM'}
Let $S$ be a K3 surface admitting a symplectic involution, which does not contain any
  effective curves.
Let us explain why such a surface exists. By the surjectivity of the period map, one can choose a marked K3
surface $S$ with $\Pic S\simeq E_8(-2)$. Since every non-zero element of $E_8(-2)$ has square at most $-4$,
Example \ref{examplewall} and Corollary \ref{cor:desrK} give $\mathcal{K}_S=\mathcal{C}_S$. Since $E_8(-2)$ is negative definite, the positive cone
contains a class orthogonal to $E_8(-2)$; consequently
$\mathcal{K}_S\cap E_8(-2)^\perp\neq\emptyset$. Proposition \ref{involutionE8} then provides a symplectic
involution on $S$.
Then, we consider $M'$ the Nikulin orbifold associated to $S^{[2]}$ and the induced involution
$\iota^{[2]}$ as in Section \ref{M'section} (we keep the same notation as earlier in this section).

\begin{prop}\label{exwalls}
The primitive wall divisors of $M'$ are $\delta'$ and $\Sigma'$ which are both of square $-4$ and divisibility $2$.
\end{prop}
This section is devoted to the proof of this proposition.
The idea of the proof is to study the curves in $M'$ and use Proposition \ref{extremalray}. 

Consider the following diagram: 
\begin{equation}
\xymatrix{
  &S^{[2]}\ar[r]^{\nu}& S^{(2)} & \\
  &\widetilde{S\times S}\ar[u]^{\widetilde{\rho}}\ar[r]^{\widetilde{\nu}}&\ar[ld]_{p_1}S^2\ar[u]^{\rho}\ar[rd]^{p_2}& \\
&S & & S,}
\label{S2}
\end{equation}
where $p_1$, $p_2$ are the projections, $\rho$ the quotient map and $\nu$ the blow-up in the diagonal in $S^{(2)}$. 
By assumption $S$ does not contain any effective curve. Hence considering the image by the projections $p_1$, $p_2$ and $\rho$,
we deduce that $S^{(2)}$ does not contain any effective curve either. Hence all effective curves in $S^{[2]}$ are contracted by
$\nu$, i.e.~fibers of the exceptional divisor $\Delta\rightarrow \Delta_{S^{(2)}}$, where $\Delta_{S^{(2)}}$ is the diagonal
in $S^{(2)}$.
We denote such a curve by  $\ell_{\delta}^s$, where $s\in S$ keeps track of the point $(s,s)\in S^{(2)}$. To
simplify the notation, we
 denote the cohomology class $\ell_{\delta}\coloneqq[\ell_{\delta}^s]$, since it does not depend
on $s \in S$. 

Our next goal is to determine the irreducible curves in $N_1$.
Recall that $r_1\colon N_1 \to S^{[2]}$ is the blow-up in the fixed surface $\Sigma$. Let $C$ be an irreducible curve in $N_1$. There are three cases:
\begin{itemize}
 \item[(i)]
 The image of $C$ by $r_1$ does not intersect $\Sigma$ and is of class $\ell_{\delta}$.
 Therefore, $C$ is of class $r_1^*(\ell_{\delta})$. 
  \item[(ii)]
   The image of $C$ by $r_1$ is contained in $\Sigma$ and of class $\ell_{\delta}$.
   \item[(iii)]
   The image of $C$ by $r_1$ is a point. Then $C$ is of class $\ell_{\Sigma}$ 
   (the class of a fiber of the exceptional divisor $\Sigma_1\longrightarrow \Sigma$).
\end{itemize}
Note that $\ell_{\delta}^s$ is contained in $\Sigma$ if $s\in S$ is a fixed point of the involution $\iota$,
and otherwise the intersection $\ell_{\delta}^s\cap \Sigma=\emptyset$ is empty (this follows from the description
of $\Sigma$ in Remark \ref{rem:Sigma}). Therefore, there cannot be a case, where the image of
$\ell_{\delta}^s$ intersects $\Sigma$ in a zero-dimensional locus. 

It remains to understand the case (ii). 
\begin{lemme}\label{Rdeltalemma}
Consider a curve 
$\ell_{\delta}^s$ contained in $\Sigma$ (i.e.~when $s\in S$ is a fixed point of $\iota$). 
The surface $H_0:=r_1^{-1}(\ell_{\delta}^s)$ is isomorphic to $\mathbb{P}^1\times\mathbb{P}^1$.
\end{lemme}
\begin{proof}
The surface $H_0$ is a Hirzebruch surface.
Since $r_1$ is the blow-up along $\Sigma$, observe that $H_0 \iso\bP(\mathcal{N}_{\Sigma/S^{[2]}}|_{\ell_\delta^s})$.
Therefore, we need to compute $\mathcal{N}_{\Sigma/S^{[2]}}|_{\ell_{\delta}^s}$.

Keeping the notation from Remark \ref{rem:Sigma}, recall that
$\widetilde{S_{\iota}}:=\widetilde{\rho}^{-1}(\Sigma)$ and
$S_{\iota}:=\widetilde{\nu}(\widetilde{S_{\iota}})$. For simplicity, we also denote by $\ell_{\delta}^s$ the
preimage of $\ell_{\delta}^s$ by $\widetilde{\rho}$. (Note for this, that $\widetilde{\rho}$ restricts to an
isomorphism over the preimage of $\Delta$, and therefore, it makes sense to identify $\ell_{\delta}^s$ with its preimage).

Observe that:
$$\mathcal{N}_{\Sigma/S^{[2]}}|_{\ell_{\delta}^s}
\iso \widetilde{\rho}^{*}(\mathcal{N}_{\Sigma/S^{[2]}})|_{\ell_{\delta}^s}
\iso\mathcal{N}_{\widetilde{S_{\iota}}/\widetilde{S\times S}}|_{\ell_{\delta}^s}
\iso\widetilde{\nu}^{*}(\mathcal{N}_{S_{\iota}/S\times S}|_{s})
\iso\mathcal{O}_{\ell_{\delta}^s}\oplus\mathcal{O}_{\ell_{\delta}^s},$$
where we identify $s\in S\iso S_\iota$. For the third isomorphism, note that $S_\iota$ and the diagonal
intersect transversely at $(s,s)$. Locally, the blow-up is the product of $S_\iota$ with the blow-up of the
normal slice at the origin, so the normal bundle of the strict transform is the pull-back of
$\mathcal N_{S_\iota/S\times S}$. It follows that
$H_0\simeq \mathbb{P}^1\times \mathbb{P}^1$.
\end{proof}
Under the isomorphism $H_0\simeq\mathbb{P}^1\times\mathbb{P}^1$, the restriction
$r_1|_{H_0}:H_0\rightarrow\ell_\delta^s$ is one of the two projections. The cone of effective curves of
$H_0$ is generated by the classes of the two rulings. The first ruling is a fiber of $r_1$ and has class
$\ell_\Sigma$, whereas the second ruling maps isomorphically to $\ell_\delta^s$ and has class
$r_1^*(\ell_\delta)$ in $N_1$. Since curves of the first ruling belong to case (iii), an extremal curve in
case (ii) necessarily has class $r_1^*(\ell_\delta)$.
\begin{rmk}\label{mainldelta}
In particular, considering cases (i) and (ii), we see that the extremal curves $C$ such that $r_1(C)=\ell_{\delta}^s$ for some $s\in S$ have classes $r_1^*(\ell_{\delta})$.
\end{rmk}

Hence, we obtain that the extremal curves in $N_1$ have classes $r_1^*\ell_{\delta}$ and
$\ell_{\Sigma}$.
This implies that the extremal curves in $M'$ have classes $\pi_{1*}r_1^*\ell_{\delta}$ and $\pi_{1*}\ell_{\Sigma}$. 

We can compute their dual divisors. 
\begin{lemme}\label{dualdeltasigma}
The dual divisors in $H^2(M',\Q)$ of $\pi_{1*}r_1^*\ell_{\delta}$ and $\pi_{1*}\ell_{\Sigma}$ are respectively 
$\frac{1}{2}\delta'$ and $\frac{1}{2}\Sigma'\in H^2(M',\bQ)$
\end{lemme}
\begin{proof}
Write 
$$H^2(N_1,\Z)=r_1^*\nu^*H^2(S^{(2)},\Z)\oplus \Z\left[\delta_1\right]\oplus\Z\left[\Sigma_1\right].$$
We denote by $p_{\delta_1}:H^2(N_1,\Z)\rightarrow \Z[\delta_1]\iso \bZ$ and $p_{\Sigma_1}:H^2(N_1,\Z)\rightarrow
\Z[\Sigma_1]\iso \bZ$ the projections.
 Let $x\in H^2(N_1,\Z)^{\iota_1^{[2]}}$.
Since $\ell_\delta\cdot\alpha=0$ for all $\alpha \in r_1^*\nu^*H^2(S^{(2)},\Z)\oplus\bZ [\Sigma_1]$, we
have  $$\ell_{\delta}\cdot x=(\ell_{\delta}\cdot \delta_1)p_{\delta_1}(x)=-p_{\delta_1}(x)$$ and similarly
$\ell_{\Sigma}\cdot x=(\ell_{\Sigma}\cdot \Sigma_1)p_{\Sigma_1}(x)=-p_{\Sigma_1}(x)$.
It follows from (\ref{Smith2}):
$$\pi_{1*}(\ell_{\delta})\cdot\pi_{1*}(x)=-2p_{\delta_1}(x)\ \text{and}\ \pi_{1*}(\ell_{\Sigma})\cdot\pi_{1*}(x)=-2p_{\Sigma_1}(x).$$

Therefore, Theorem \ref{BBform} (iii) shows that $\frac{1}{2}\delta'$ and $\frac{1}{2}\Sigma'\in H^2(M',\bQ)$ are the duals of
$\pi_{1*}(\ell_{\delta})$ and $\pi_{1*}(\ell_{\Sigma})$ respectively.
\end{proof}
\begin{proof}[Proof of Proposition \ref{exwalls}]
By Proposition \ref{extremalray}, the previous lemma proves that $\delta'$ and $\Sigma'$ are wall divisors in $M'$.
Their claimed numerical properties are given by Theorem \ref{BBform} (iii) and Remark \ref{div}.

It remains to show that $\delta'$ and $\Sigma'$ are the only wall divisors in $M'$. 
Let us assume for contradiction that there is another wall divisor $D$. 
By Theorem \ref{BBform} (v), we have 
$D=a\delta'+b\Sigma'+K$, with $(a,b)\in \Z[\frac{1}{2}]\times\N[\frac{1}{2}]$, (up to replacing $D$ by $-D$) and $K$ a divisor orthogonal to $\delta'$ and $\Sigma'$.

Since $\delta'$ and $\Sigma'$ correspond to the duals of the extremal rays of the Mori cone,
all classes $\alpha\in \mathcal{C}_{M'}$ such that $(\alpha,\delta')_q>0$ and $(\alpha,\Sigma')_q>0$ are K\"ahler
classes by \cite[Theorem 4.1]{Menet-Riess-20}.

Hence, we cannot have $a=b=0$. Indeed, $D$ would be orthogonal to the K\"ahler classes with orthogonal projection on $\Pic M'$ equal to $ -(\Sigma' + \delta')$  which is impossible by definition of wall divisors. 

Therefore, $a$ or $b$ are non trivial. 
It follows that $K=0$. 
Indeed, if $K\neq0$, as  before, the class $ -(\Sigma' + \delta')-\frac{4(b+a)}{(K,K)}K$ is the projection on $\Pic M'$ of a
K\"ahler class; however it is orthogonal to $D$ which is impossible.

Now, we can assume that $a\neq 0$ and $b\neq 0$ (indeed, if $a=0$ or $b=0$, then $D\in \bZ\delta'$ or $D\in \bZ\Sigma'$).
If $a<0$, $D$ would be orthogonal to $ a\Sigma' - b\delta'$ the projection on $\Pic M'$ of a K\"ahler class; this is impossible by definition of wall divisors.
Hence we assume that $a>0$. By Corollary \ref{Rdelta}, $R_{\delta'}$ is a monodromy operator.
Moreover, $R_{\delta'}(D)=-a\delta'+b\Sigma'$; as previously $R_{\delta'}(D)$ is orthogonal
to some K\"ahler class. This gives a contradiction and thus concludes the proof. 
\end{proof}
\subsection{Application: an example of non-natural symplectic involution on a Nikulin orbifold}\label{inv0M'}
Using the results from the previous subsection, we will prove the existence of a non-natural symplectic
involution on the $M'$ of the previous section. We recall that the reflections $R_x$ with $x\in H^2(M',\Z)$ are defined in Section \ref{notation}.
\begin{prop}\label{involution}
Let $(S,\iota)$ be a very general K3 surface endowed with a symplectic involution (that is $\Pic S\simeq E_8(-2)$).
Let $M'$ be a Nikulin orbifold constructed from $(S,\iota)$ as in Section \ref{M'S2}. 
There exists $\kappa'$ a symplectic involution on $M'$ such that
$\kappa'^*=R_{\frac{1}{2}(\delta'-\Sigma')}$.
\end{prop}
\begin{proof}
We consider the following involution on $S\times S$:
$$\xymatrix@R0pt{\kappa: \hspace{-1.5cm}& S\times S\ar[r] & S\times S\\
& (x,y)\ar[r]&(x,\iota(y)).}$$
We consider $$V:=S\times S\smallsetminus \left(\Delta_{S^{2}}\cup S_{\iota}\cup(\Fix \iota\times \Fix \iota)\right).$$
We denote by $\sigma_2$ the involution on $S\times S$ which exchanges the two K3 surfaces and $\iota\times\iota$ the involution which acts as $\iota$ diagonally on $S\times S$. Then we consider 
$$U:=V/\left\langle \sigma_2,\iota\times\iota\right\rangle.$$
This set can be seen as an open subset of $M'$ and $V$ can also be seen as an open subset of $\widetilde{S\times S}$.
Moreover the map $\pi_1\circ\widetilde{\rho}_{|V}:V\rightarrow U$ is a four to one non-ramified cover. For simplicity, we denote $\gamma:=\pi_1\circ\widetilde{\rho}_{|V}$.

First, we want to prove that $\kappa$ induces an involution $\kappa'$ on $U$ with a commutative diagram:
\begin{equation}
\xymatrix{V\ar[d]_{\gamma}\ar[r]^{\kappa}&V\ar[d]^{\gamma}\\
U\ar[r]^{\kappa'} & U.}
\label{kappadiagram}
\end{equation}
If such a map $\kappa'$ existed, it would necessarily satisfy the following equation:
$$\kappa'\circ \gamma=\gamma\circ \kappa.$$
Since the map $\gamma$ is surjective, to claim that the previous equation provides a well-defined map from $U$ to $U$, we have to verify that: 
$$\kappa'\circ\gamma(x_1,y_1)=\kappa'\circ\gamma(x_2,y_2),$$ whenever
$\gamma(x_1,y_1)=\gamma(x_2,y_2)$.
That is:
\begin{equation}
\gamma\circ \kappa(x_1,y_1)=\gamma\circ \kappa(x_2,y_2),
\label{welldefined}
\end{equation}
for all $((x_1,y_1),(x_2,y_2))\in S^4$ such that $\gamma(x_1,y_1)=\gamma(x_2,y_2)$. Let $(x,y)\in S^2$.
We have: $$\gamma^{-1}(\gamma(x,y))=\left\{(x,y),(y,x),(\iota(x),\iota(y)),(\iota(y),\iota(x))\right\}.$$
We also have: $$\kappa(\gamma^{-1}(\gamma(x,y)))=\left\{(x,\iota(y)),(y,\iota(x)),(\iota(x),y),(\iota(y),x)\right\}=\gamma^{-1}(\gamma(x,\iota(y)))=\gamma^{-1}(\gamma(\kappa(x,y))).$$
This shows (\ref{welldefined}). Hence $\kappa'$ is set-theoretically well-defined. Since $\gamma$ is a four-to-one non-ramified cover, it is a local isomorphism; therefore, $\kappa'$ inherits the properties of $\kappa$. In particular, $\kappa'$ is a holomorphic symplectic involution.

It follows that $\kappa'$ induces a bimeromorphic symplectic involution on $M'$. By \cite[Lemma 3.2]{Menet-Riess-20}, $\kappa'$ extends to a bimeromorphic symplectic involution which is an isomorphism in codimension 1 (we still denote by $\kappa'$ this involution).
In particular, $\kappa'^*$ is now well-defined on $H^2(M',\Z)$ (see \cite[Lemma 3.4]{Menet-Riess-20}).
Now, we are going to prove that $\kappa'$ extends to a regular involution. 
We recall from Theorem \ref{BBform} (v) that:
$$H^2(M',\Z)=r^*\pi_*(j(H^2(S,\Z)))\oplus^{\bot}\Z\frac{\delta'+\Sigma'}{2}\oplus^{\bot}\Z\frac{\delta'-\Sigma'}{2}.$$
Since $\kappa'$ is symplectic $\kappa'$ acts trivially on $r^*\pi_*(j(H^2(S,\Z)))$. Indeed $\Pic M'=\Z\frac{\delta'+\Sigma'}{2}\oplus\Z\frac{\delta'-\Sigma'}{2}$. Moreover, $\kappa$ exchanges $\Delta_{S^{(2)}}$ and $S_{\iota}$. Hence by continuity and commutativity of diagram (\ref{kappadiagram}), we have that $\kappa'$ exchanges the divisors $\delta'$ and $\Sigma'$. 
By Proposition \ref{exwalls} and Corollary \ref{criterionwall}, it follows that $\kappa'$ sends K\"ahler classes to K\"ahler classes. Hence by \cite[Proposition 3.3]{Menet-Riess-20}, $\kappa'$ extends to an involution on $M'$.
\end{proof}
\begin{rmk}
Let $Q:=M'/\langle\kappa'\rangle$. Since $\kappa'^*$ is the reflection in
$\frac{1}{2}(\delta'-\Sigma')$, its invariant subspace in $H^2(M',\Q)$ has dimension $15$. Hence
$b_2(Q)=15$. However $Q$ is not terminal. The eight fixed points of $\iota$ give rise to eight fixed surfaces of $\kappa'$,
whose images are the codimension-two components of the singular locus of $Q$. At a general point of each
of these components, the analytic germ of $Q$ is isomorphic to
$\C^2\times(\C^2/\langle-\id\rangle)$. At the non-general points, the analytic germs are of the form
$$
  \C^4/(\langle-\id\rangle\wr\mathfrak{S}_2)
  \simeq \bigl(\C^2/\langle-\id\rangle\bigr)^{(2)}.
$$
These germs admit the smooth symplectic resolution
$$
  \left(\widetilde{\C^2/\langle-\id\rangle}\right)^{[2]}
  \longrightarrow \bigl(\C^2/\langle-\id\rangle\bigr)^{(2)},
$$
where $\widetilde{\C^2/\langle-\id\rangle}$ denotes the minimal resolution. Away from the non-general
points, this morphism restricts to the resolution of the transversal $A_1$-singularities described above.
The local resolutions therefore glue and provide a smooth symplectic terminalization
$\widetilde{Q}\rightarrow Q$.

We can moreover determine its deformation type. Set $\overline{S}:=S/\langle\iota\rangle$ and let
$\widehat{S}\rightarrow\overline{S}$ be its minimal resolution, which is a K3 surface. At the level of
function fields, the construction of $\kappa'$ gives
$$
  \C(Q)=\C(S\times S)^{\langle\sigma_2,\iota\times\iota,\kappa\rangle}
  =\C(S\times S)^{\langle\iota\rangle\wr\mathfrak{S}_2}
  =\C\bigl(\overline{S}^{(2)}\bigr).
$$
Consequently, $\widetilde{Q}$ is bimeromorphic to $\widehat{S}^{[2]}$. Since bimeromorphic irreducible
symplectic manifolds are deformation equivalent by \cite[Theorem 4.6]{Huybrechts:HK:basic-results},
$\widetilde{Q}$ is of $K3^{[2]}$-type.
\end{rmk}
\begin{cor}\label{Sigma'}
Let $M'$ be a Nikulin orbifold.
Then: $$R_{\Sigma'}\in \Mon^2(M').$$
\end{cor}
\begin{proof}
Let $(X,\widetilde{\iota})$ be any manifold of $K3^{[2]}$-type endowed with a symplectic involution. Let $(S,\iota)$ be a very general K3 surface endowed with a symplectic involution. 
With exactly the same argument as in proof of Proposition \ref{MonoM'}, we can connect $X$ and $S^{[2]}$ by
a sequence of twistor spaces; each twistor space being endowed with an involution which restricts to a symplectic involution on its fibers. 
This sequence of twistor spaces provides a sequence of twistor spaces between $M'_X$ and $M'_{S^{[2]}}$ the irreducible symplectic orbifolds associated to $X$ and $S^{[2]}$ respectively. This sequence of twistor spaces provides a parallel transport operator $f:H^2(M'_X,\Z)\rightarrow H^2(M'_{S^{[2]}},\Z)$ which sends $\Sigma'_X$ to $\Sigma'_{S^{[2]}}$ (respectively the exceptional divisors of the blow-ups $M'_X\rightarrow X/\widetilde{\iota}$ and $M'_{S^{[2]}}\rightarrow S^{[2]}/\iota^{[2]}$).

By Proposition \ref{involution}, $\kappa'^*\in\Mon^2(M'_{S^{[2]}})$. 
Moreover by Corollary \ref{Rdelta}, $R_{\delta'}\in \Mon^2(M'_{S^{[2]}})$. 
Hence $R_{\Sigma'_{S^{[2]}}}=\kappa'^*\circ R_{\delta'}\circ \kappa'^*\in\Mon^2(M'_{S^{[2]}})$. Therefore $R_{\Sigma'_{X}}=f^{-1}\circ R_{\Sigma'_{S^{[2]}}}\circ f\in\Mon^2(M'_X)$.
\end{proof} 
\begin{rmk}\label{RdeltaSigma}
Let $(S,\iota)$ be a K3 surface endowed with a symplectic involution.
Let $M'$ be a Nikulin orbifold constructed from $(S,\iota)$ as in Section \ref{M'S2}.
The previous proof also shows that $R_{\frac{1}{2}(\delta'-\Sigma')}$ is a parallel transport operator.
\end{rmk}
\section{In search of wall divisors in special examples}\label{extremalcurves}
In this section we study some explicit examples of K3 surfaces with symplectic involutions and their
associated Nikulin orbifolds. This will be used in Section \ref{endsection} in order to determine which divisors on
Nikulin-type orbifolds are wall divisors. 
\subsection{Overview and notation}\label{5nota}
\subsubsection*{Cases considered}
We consider the following three special examples:
\begin{itemize}
\item[(i)] a K3 surface containing a unique rational curve, which is preserved by the involution
(Section \ref{onecurve});
\item[(ii)] a K3 surface containing two rational curves exchanged by the involution
(Section \ref{sec:twocurves});
\item[(iii)] a K3 surface endowed with a particular elliptic fibration compatible with the involution
(Section \ref{sec:elliptic}).
\end{itemize}
\subsubsection*{Common notation}
Throughout the section, we keep the notation introduced in Section \ref{genericM'}. In particular,
$$
  \nu:S^{[2]}\longrightarrow S^{(2)},\qquad
  r_1:N_1\longrightarrow S^{[2]},\qquad
  \pi_1:N_1\longrightarrow M'
$$
denote respectively the Hilbert--Chow morphism, the blow-up of the fixed K3 surface, and the quotient map.

In the first example, $C$ denotes the unique rational curve on $S$. In the second example, we choose one of
the two rational curves and denote it by $C$; the other one is then $\iota(C)$. The following notation is
attached to this choice of $C$. In the second example, the analogous objects associated with $\iota(C)$ are
obtained by applying $\iota$.
\begin{nota}\label{notacurves}
\begin{itemize}
\item
We denote by $D_C$ the divisor
$$
  D_C:=\left\{\left.\xi\in S^{[2]}\ \right|\ \Supp\xi\cap C\neq\emptyset\right\}
$$
and set $D_C':=\pi_{1*}r_1^*(D_C)$.
\item
We denote by $\overline{C^{(2)}}$ the strict transform of $C^{(2)}\subset S^{(2)}$ by $\nu$ and by
$\overline{\overline{C^{(2)}}}$ the strict transform of $\overline{C^{(2)}}\subset S^{[2]}$ by $r_1$.
We denote by
$j^{C}:\overline{C^{(2)}}\hookrightarrow S^{[2]}$ and
$\overline{j^{C}}:\overline{\overline{C^{(2)}}}\hookrightarrow N_1$ the embeddings. Note that
$\overline{\overline{C^{(2)}}}\simeq\overline{C^{(2)}}\simeq C^{(2)}\simeq \mathbb{P}^2$.
\item
We recall that $\Delta_{S^{(2)}}$ is the diagonal in $S^{(2)}$ and $\Delta$ the diagonal in $S^{[2]}$.
We also denote by $\Delta_{S^2}$ the diagonal in $S\times S$ and by
$\Delta_{\widetilde{S^2}}:=\widetilde{\nu}^{-1}(\Delta_{S^{2}})$ the exceptional divisor. Furthermore, we set
$$
  \Delta_{S^{(2)}}^C:=C^{(2)}\cap \Delta_{S^{(2)}},\qquad
  \Delta_{S^{2}}^{C}:=C^{2}\cap \Delta_{S^2},
$$
and
$$
  \Delta_{S^{(2)}}^{\iota,C}:=\left\{\left.\{x,\iota(x)\}\ \right|\ x\in C\right\},\qquad
  \Delta_{S^{2}}^{\iota,C}:=\left\{\left.(x,\iota(x))\ \right|\ x\in C\right\}.
$$
\item
We denote $H_2:=\nu^{-1}(\Delta_{S^{(2)}}^C)$ and by $\overline{H_2}$ its strict transform by $r_1$.
We denote by $j^{H_2}:H_2\hookrightarrow S^{[2]}$ and
$j^{\overline{H_2}}:\overline{H_2}\hookrightarrow N_1$ the embeddings.
\end{itemize}
\end{nota}
This notation is summarized in the following diagram:
$$\xymatrix{&\ar@{^{(}->}[dd]_{j^{\overline{H_2}}}\overline{H_2}\ar@{=}[r]& \ar@{^{(}->}[d]H_2\ar@/_1pc/@{_{(}->}[dd]_{j^{H_2}}\ar[r]   &   \ar@{^{(}->}[d]\Delta_{S^{(2)}}^C & &\\
 & & \Delta\ar[r]\ar@{^{(}->}[d]  &   \Delta_{S^{(2)}}\ar@{^{(}->}[d]& &\\
& N_1\ar[r]^{r_1}  & S^{[2]}\ar[r]^{\nu}  &  S^{(2)}& \ar@{^{(}->}[l] C^{(2)} &\ar@{^{(}->}[l] \Delta_{S^{(2)}}^{\iota,C}\\
\overline{\overline{C^{(2)}}}\ar@{^{(}->}[ru]^{\overline{j^{C}}}\ar@{=}[r]& \overline{C^{(2)}}\ar@{^{(}->}[ru]^{j^{C}}& \widetilde{S^{2}}\ar[u]^{\widetilde{\rho}} \ar[r]^{\widetilde{\nu}}&  S^{2}\ar[u]^{\rho} &  \ar@{^{(}->}[l] C^{2}\ar[u]^{2:1} &\ar@{^{(}->}[l]  \Delta_{S^{2}}^{\iota,C}\ar[u]^{2:1}\\
& & \Delta_{\widetilde{S^2}}\ar@{^{(}->}[u]\ar[r]  & \ar@{^{(}->}[u]   \Delta_{S^{2}}   &&
}$$
\subsubsection*{Strategy}
The first two examples are treated by the same strategy. 
\begin{itemize}
\item
We first classify the relevant curves in
$S^{[2]}$, 
\item
then determine the classes of their strict transforms in $N_1$, 
\item
and finally push these classes
forward to $M'$ and compute their dual divisors. 
\end{itemize}
Proposition \ref{extremalray} then provides the corresponding
wall divisors. We carry out this analysis in detail in the first example and reuse it in the second one.

The elliptic example requires a different argument. Instead of determining all the extremal curves, we use
the elliptic fibration to construct a particular divisor and to control the form of an extremal curve class.
The numerical restrictions established in Section \ref{endsection} will then identify the relevant wall
divisor.
\subsection{When the K3 surface $S$ used to construct the Nikulin orbifold contains a unique rational curve}\label{onecurve}
\subsubsection*{Objective}
In this section we assume that $\Pic S\simeq E_8(-2)\oplus^{\bot} (-2)$. By Riemann--Roch, $S$ contains a unique irreducible curve; moreover it is a rational curve. We denote this curve by $C$. In particular, in this case, $\mathcal{K}_S\cap
E_8(-2)^{\bot}\neq\emptyset$. Hence, by Proposition \ref{involutionE8} there exists a symplectic involution
$\iota$ on $S$ whose anti-invariant lattice is the $E_8(-2)\subset \Pic(S)$. Moreover, the curve $C$ is fixed by $\iota$. The objective of this section is to determine wall divisors of the Nikulin orbifold $M'$ obtained from the pair $(S^{[2]},\iota^{[2]})$ (see Section \ref{M'section}). 
\begin{lemme}\label{Hirzebruch}
The surface $H_2$ is a Hirzebruch surface isomorphic to $\mathbb{P}(\mathcal{E})$, where $\mathcal{E}:=\mathcal{O}_C(2)\oplus\mathcal{O}_C(-2)$. Let $f$ be a fiber of $\mathbb{P}(\mathcal{E})$. There exists a section $C_0$ which
satisfies:  $\Pic H_2=\Z C_0\oplus\Z f$, $C_0^2=-4$, $f^2=0$ and $C_0\cdot f=1$.   
\end{lemme}
\begin{proof}
Let $\sigma_2$ be the involution on $S\times S$ which exchanges the two K3 surfaces and $\widetilde{\sigma_2}$ the induced involution on $\widetilde{S\times S}$. The involution $\widetilde{\sigma_2}$ acts trivially on $\Delta_{\widetilde{S^2}}$. It follows that $\widetilde{\rho}$ induces an isomorphism $\Delta_{\widetilde{S^2}}\simeq \Delta$. In particular, it shows that:
$$H_2\simeq\mathbb{P}(\mathcal{N}_{\Delta_{S^{2}}/S\times S}|_{\Delta_{S^{2}}^{C}}).$$
We consider the following commutative diagram:
\begin{equation}
\xymatrix{\Delta_{S^{2}}\eq[r]& S\\
\Delta_{S^{2}}^{C}\ar@{^{(}->}[u]\eq[r] & \ar@{^{(}->}[u]C.}
\label{delta0}
\end{equation}
Under the isomorphism induced by (\ref{delta0}), observe that:
$$\mathcal{N}_{\Delta_{S^2}/S\times S}|_{\Delta_{S^{2}}^{C}}\simeq T_S|_C.$$
 To compute $T_S|_C$, we consider the following exact sequence:
$$\xymatrix{0\ar[r]&\ar[r]T_C\ar[r]&T_S|_C\ar[r]&\mathcal{N}_{C/S}\ar[r]&0.}$$
 We have $T_C=\mathcal{O}_C(2)$ and $\mathcal{N}_{C/S}=\mathcal{O}_C(-2)$. 
Moreover, $\Ext^1(\mathcal{O}_C(-2),\mathcal{O}_C(2))=H^1(C,\mathcal{O}_C(4))=0$. 
Hence: $$T_S|_C=\mathcal{O}_C(-2)\oplus\mathcal{O}_C(2).$$
As a consequence $H_2\iso \bP(\mathcal{E})$ as claimed.

Therefore, by \cite[Chapter V, Proposition 2.3 and Proposition 2.9]{Hartshorne},
we know that $\Pic H_2=\Z C_0\oplus\Z f$, with $C_0^2=-4$, $f^2=0$ and $C_0\cdot f=1$; 
 $C_0$ being the class of a specific section and $f$ the class of a fiber.
\end{proof}
\begin{lemme}\label{intersection}
We have $\ell_{\delta}\cdot\delta=-1$ and $C\cdot D_C=-2$ in $S^{[2]}$.
\end{lemme}
\begin{proof}
We denote by $\widetilde{\ell_{\delta}}$ a fiber associated to the exceptional divisor $\Delta_{\widetilde{S^2}}\rightarrow \Delta_{S^{2}}$.
We know that $\widetilde{\ell_{\delta}}\cdot \Delta_{\widetilde{S^2}}=-1$. We can deduce for instance from \cite[Lemma 3.6]{Menet-2018} (see \ref{Smith2}) that $\ell_{\delta}\cdot \Delta=-2$. That is $\ell_{\delta}\cdot \delta=-1$.

Similarly, we have $(C\times S+S\times C)\cdot(s\times C+C\times s)=-4$. By \cite[Lemma 3.6]{Menet-2018}:
\begin{align*}
&\rho_*(C\times S+S\times C)\cdot\rho_*(s\times C+C\times s)=-8\\
&\rho_*(C\times S)\cdot\rho_*(s\times C)=-2.
\end{align*}
Then taking the pull-back by $\nu$, we obtain $C\cdot D_C=-2$.
\end{proof}
\subsubsection*{Curves in $S^{[2]}$}
The first step is to determine the curves contained in $S^{[2]}$. Before that, we can say the following about curves in $S^{(2)}$.
\begin{lemme}\label{S2curves}
There are only two cases for irreducible curves in $S^{(2)}$:
\begin{itemize}
\item[(1)]
the curves $C^s:=\rho(C\times s)=\rho(s\times C)$ with $s\notin C$;
\item[(2)]
curves in $C^{(2)}\simeq \mathbb{P}^2$.
\end{itemize}
\end{lemme}
\begin{proof}
In $S\times S$, considering the images of curves by the projections $p_1$ and $p_2$ of diagram (\ref{S2}), there are only two possibilities: 
\begin{itemize}
\item
$s\times C$ or $C\times s$, with $s$ a point in $S$,
\item
curves contained in $C\times C$.
\end{itemize}
Then, we obtain all the curves in $S^{(2)}$ as images by $\rho$ of curves in $S\times S$. 
\end{proof}
It follows four cases in $S^{[2]}$.
\begin{lemme}\label{curveS2}
We have the following four cases for irreducible curves in $S^{[2]}$:
\begin{itemize}
\item[(0)]
Curves which are fibers of the exceptional divisor $\Delta\rightarrow \Delta_{S^{(2)}}$. As in Section \ref{genericM'}, we denote these curves by $\ell_{\delta}^s$,
where $s=\nu(\ell_{\delta}^s)$ and we denote their classes by $\ell_{\delta}$.
\item[(1)]
Curves which are strict transforms of $C^s$ with $s\notin C$. We denote the class of these curves by $C$. Note that $C=\nu^*(\left[C^s\right])$ for $s\notin C$.
\item[(2a)]
Curves contained in $H_2$. 
\item[(2b)]
Curves contained in $\overline{C^{(2)}}$. 
\end{itemize}
\end{lemme}
\begin{proof}
Let $\gamma$ be an irreducible curve in $S^{[2]}$. By Lemma \ref{S2curves}, there are three cases for $\nu(\gamma)$:
\begin{itemize}
\item[(0)]
$\nu(\gamma)$ is a point and $\gamma$ is a fiber of the exceptional divisor;
\item[(1)]
$\nu(\gamma)=C^s$, with $s\notin C$.
\item[(2)]
$\nu(\gamma)\subset C^{(2)}$.
\end{itemize}
Moreover case (2) can be divided in 2 sub-cases: $\nu(\gamma)=\Delta_{C}$ or $\nu(\gamma)\nsubseteq\Delta_{C}$.
This provides cases (2a) and (2b).
\end{proof}
Now, we are going to determine the classes of the extremal curves in cases (2a) and (2b) in two lemmas.
\begin{lemme}\label{C0}
We have $j^{H_2}_*(C_0)=2(C-\ell_{\delta})$ and $j^{H_2}_*(f)=\ell_{\delta}$.
\end{lemme}
\begin{proof}
It is clear that $j^{H_2}_*(f)=\ell_{\delta}$, we are going to compute $j^{H_2}_*(C_0)$. 
Necessarily, $j^{H_2}_*(C_0)=aC+b\ell_{\delta}$. We can consider the intersection with $\Delta$ and $D_C$ and use Lemma \ref{intersection} to determine $a$ and $b$.
We consider the following commutative diagram:
$$\xymatrix{H_2\ar@{=}[d]\ar@{^{(}->}[r]^{\widetilde{j^{H_2}}}&\widetilde{S\times S}\ar[d]^{\widetilde{\rho}}\\
H_2\ar@{^{(}->}[r]^{j^{H_2}}&S^{[2]}.}$$
By commutativity of the diagram, we have:
\begin{equation}
 j^{H_2}_*(C_0)=\widetilde{\rho}_*\widetilde{j^{H_2}}_*(C_0).
\label{heuuu}
\end{equation}
By \cite[Propositions 2.6 and 2.8]{Hartshorne}, we have:
$$C_0\cdot \widetilde{j^{H_2}}^*(\Delta_{\widetilde{S^2}})=C_0\cdot \mathcal{O}_{\mathcal{E}}(-1)=-C_0\cdot(C_0+2f)=4-2=2.$$
By projection formula that is:
$$\widetilde{j^{H_2}}_*(C_0)\cdot \Delta_{\widetilde{S^2}}=2.$$
Taking the push-forwards by $\widetilde{\rho}$, we obtain by \cite[Lemma 3.6]{Menet-2018} (see \ref{Smith2}):
$$\widetilde{\rho}_*\widetilde{j^{H_2}}_*(C_0)\cdot\widetilde{\rho}_*(\Delta_{\widetilde{S^2}})=4.$$
Hence by (\ref{heuuu}): 
$$j^{H_2}_*(C_0)\cdot\Delta=4.$$
Hence by lemma \ref{intersection}:
$$b=-2.$$
We have $D_C=\nu^*(\rho_*(C\times S))$. So by projection formula:
$$D_C\cdot j^{H_2}_*(C_0)=\rho_*(C\times S)\cdot\left[\Delta_{S^{(2)}}^C\right]=\rho_*(C\times S)\cdot\rho_*(s \times C+C\times s)=2\rho_*(C\times S)\cdot\rho_*(s \times C).$$
Taking the pull-back by $\nu$ of the last equality, we obtain:
$$D_C\cdot j^{H_2}_*(C_0)=2D_C\cdot C.$$
So $a=2$ which concludes the proof.
\end{proof}
\begin{lemme}\label{stricttransform}
We have $j^C_*(d)=C-\ell_{\delta}$, where $d$ is the class of a line in $\overline{C^{(2)}}\simeq\mathbb{P}^2$.
\end{lemme}
\begin{proof}
We consider the following commutative diagram:
$$\xymatrix{\overline{C^{(2)}}\eq[d]\ar@{^{(}->}[r]^{j^C}&S^{[2]}\ar[d]^{\nu}\\
C^{(2)}\ar@{^{(}->}[r]^{j^C_0}&S^{(2)}.}$$
Let $\gamma$ be an irreducible curve in $\overline{C^{(2)}}$. Since $\overline{C^{(2)}}$ is the strict transform of $C^{(2)}$ by $\nu$,
$j^{C}(\gamma)$ is the strict transform of $j^{C}_0(\gamma)$ by $\nu$.
Hence to compute $j^C_*(d)$ for $d$ the class of a line, it is enough to find a curve in $C^{(2)}$ with class $d$ and determine its strict transform by $\nu$. For instance, $C^s$ with $s\in C$ satisfies $\left[C^s\right]=d$ in $C^{(2)}$. Moreover, $C^s$ intersects $\Delta_{S^{(2)}}$ transversely in one point. It follows that $j^C_*(d)=C-\ell_{\delta}$.
\end{proof}
\subsubsection*{Curves in $N_1$}
\begin{lemme}\label{curvesN1}
We have the following cases for irreducible curves in $N_1$:
\begin{itemize}
\item[(00)]
Curves contracted to a point by $r_1$. 
They are fibers of the exceptional divisor $\Sigma_1\rightarrow \Sigma$ and their class is $\ell_{\Sigma}$.
\item[(0)]
Curves sent to $\ell_{\delta}^s$ by $r_1$. An extremal such a curve has class $r_1^*(\ell_{\delta})$ by Lemma \ref{Rdeltalemma}.
\item[(1)]
Curves sent to $C^{s}$ by $r_1$ with $s\notin C$.
They are curves of class $r_1^*(C)$.
\item[(2a.i)]
Curves contained in $r_1^{-1}(H_2\cap \Sigma)$.
\item[(2a.ii)]
Curves contained in $\overline{H_2}$, the strict transform of $H_2$ by $r_1$.
\item[(2b.i)]
Curves contained in $r_1^{-1}(\overline{C^{(2)}}\cap \Sigma)$.
\item[(2b.ii)]
Curves contained in $\overline{\overline{C^{(2)}}}$, the strict transform of $\overline{C^{(2)}}$ by $r_1$.
\end{itemize}
\end{lemme}
\begin{proof}
Let $\gamma$ be an irreducible curve in $N_1$. If $r_1(\gamma)$ is a point, we are in case (00). If $r_1(\gamma)$ is a curve, we are in one of the cases of Lemma \ref{curveS2}.

If $r_1(\gamma)=\ell_{\delta}^s$ for some $s\in S$. This is cases (i) and (ii) of Section \ref{genericM'}. 
It follows from Remark \ref{mainldelta} that the extremal curves in case (0) have classes $r_1^*(\ell_{\delta})$.
If $r_1(\gamma)=C^s$ with $s\notin C$, then $C^s$ does not intersect $\Sigma$ and we have $\left[\gamma\right]=r_1^*(C)$. The last four cases appear when $r_1(\gamma)\subset H_2$ or $r_1(\gamma)\subset \overline{C^{(2)}}$.
\end{proof}
Now we are going to determine the classes of the curves in cases (2a.i), (2a.ii), (2b.i) and (2b.ii).
\begin{lemme}\label{H2Sigma}
The extremal curves in $r_1^{-1}(H_2\cap \Sigma)$ are of classes $\ell_{\Sigma}$ or $r_1^*(\ell_{\delta})$.
\end{lemme}
\begin{proof}
Since $C$ is the unique curve contained in $S$, the involution $\iota$ on $S$ restricts to $C$. Since $\iota$ is a symplectic involution, $\iota$ does not act trivially on $C$. Moreover, since $C\simeq \mathbb{P}^1$, $\iota_{|C}$ has two fixed points $x$ and $y$. It follows that $\iota^{(2)}_{|\Delta_{C}}$ also has two fixed points $(x,x)$ and $(y,y)$. Hence $H_2\cap \Sigma=\ell_{\delta}^x\cup \ell_{\delta}^y$.
The surfaces $r_1^{-1}(\ell_{\delta}^x)$ and $r_1^{-1}(\ell_{\delta}^y)$ are Hirzebruch surfaces and by Lemma \ref{Rdeltalemma}, they are isomorphic to $\mathbb{P}^1\times\mathbb{P}^1$. Then the extremal curves of these Hirzebruch surfaces will have classes $\ell_{\Sigma}$ or $r_1^*(\ell_{\delta})$ in $N_1$.
\end{proof}
\begin{lemme}
Let $C_0$ be the class of the section in $\overline{H_2}$ obtained in Lemma \ref{Hirzebruch}.
 Then $j^{\overline{H_2}}_*(C_0)=2\left(r_1^*(C)-r_1^*(\ell_{\delta})-\ell_{\Sigma}\right)$.
\end{lemme}
\begin{proof}
As explained in the proof of the previous lemma, $H_2\cap \Sigma$ corresponds to two fibers of the Hirzebruch surface $H_2$. Hence $j^{H_2}(C_0)$ and $\Sigma$ intersect in two points. The class $j^{\overline{H_2}}_*(C_0)$ corresponds to the class of the strict transform by $r_1$ of $j^{H_2}(C_0)$. By Lemma \ref{C0}, $\left[j^{H_2}(C_0)\right]=2(C-\ell_{\delta})$. Hence $j^{\overline{H_2}}_*(C_0)=2\left(r_1^*(C)-r_1^*(\ell_{\delta})-\ell_{\Sigma}\right)$.
\end{proof}
\begin{lemme}\label{deltaiota}
The curve $\Delta_{S^{(2)}}^{\iota,C}$ is a line in $C^{(2)}\simeq \mathbb{P}^2$.
\end{lemme}
\begin{proof}
The map $$\xymatrix{\eq[d]C^2\ar[r]^{\rho}_{2:1}&C^{(2)}\eq[d]\\
\mathbb{P}^1\times \mathbb{P}^1 &\mathbb{P}^2}$$
is a two to one ramified cover. We recall that $\Delta_{S^{2}}^{\iota,C}=\left\{\left.(x,\iota(x))\right|\ x\in C\right\}$. 
We have: $$\left[\Delta_{S^{2}}^{\iota,C}\right]_{S^2}=\left[C\times s\right]_{S^2}+\left[s\times C \right]_{S^2}.$$
It follows: $$\rho_*\left(\left[{\Delta_{S^{2}}^{\iota,C}}\right]_{S^2}\right)=2\left[C^s\right]_{S^{(2)}}.$$
However $\rho: \Delta_{S^{2}}^{\iota,C}\rightarrow \Delta_{S^{(2)}}^{\iota,C}$ is a two to one cover; so $\left[\Delta_{S^{(2)}}^{\iota,C}\right]_{S^{(2)}}=\left[C^s\right]_{S^{(2)}}$.
Hence $\left[\Delta_{S^{(2)}}^{\iota,C}\right]_{C^{(2)}}=\left[C^s\right]_{C^{(2)}}$ and $\left[\Delta_{S^{(2)}}^{\iota,C}\right]_{C^{(2)}}$ is the class of a line in $C^{(2)}$.
\end{proof}
\begin{lemme}
The surface $r_1^{-1}(\overline{C^{(2)}}\cap \Sigma)$ is a Hirzebruch surface that we denote by $H_1$. 
Let $f$ be a fiber of $H_1$. There exists a section $D_0$ which
satisfies:  $\Pic H_1=\Z D_0\oplus\Z f$, $D_0^2=-2$, $f^2=0$ and $D_0\cdot f=1$.
Moreover, $j^{H_1}_*(D_0)=r_1^*(C)-\ell_{\delta}-\ell_{\Sigma}$, where $j^{H_1}:H_1\hookrightarrow N_1$ is the embedding. 
\end{lemme}
\begin{proof}
We denote by $\zeta$ the curve $\overline{C^{(2)}}\cap \Sigma$. This curve is the strict transform of $\Delta_{S^{(2)}}^{\iota,C}$ by $\nu$. 
In particular, it is a rational curve by Lemma \ref{deltaiota}.
Moreover by Lemmas \ref{deltaiota} and \ref{stricttransform}, we have: 
\begin{equation}
\left[\zeta\right]_{S^{[2]}}=C-\ell_{\delta}.
\label{gammadelta}
\end{equation}
To understand which Hirzebruch surface $H_1$ is, we are going to compute $\mathcal{N}_{\Sigma/S^{[2]}}|_{\zeta}=\mathcal{O}_{\zeta}(-k)\oplus\mathcal{O}_{\zeta}(k)$.
We consider $\widetilde{\zeta}:=\widetilde{\rho}^{-1}(\zeta)$. 
We have: $$\widetilde{\rho}^*(\mathcal{N}_{\Sigma/S^{[2]}})|_{\widetilde{\zeta}}=\widetilde{\nu}^*(\mathcal{N}_{S_{\iota}/S^2}|_{\Delta_{S^{2}}^{\iota,C}}).$$
As in the proof of Lemma \ref{Hirzebruch}, we have:
$$\mathcal{N}_{S_{\iota}/S^2}|_{\Delta_{S^{2}}^{\iota,C}}\simeq T_S|_C\simeq \mathcal{O}_C(-2)\oplus\mathcal{O}_C(2).$$
Hence $$\widetilde{\rho}^*(\mathcal{O}_{\zeta}(-k)\oplus\mathcal{O}_{\zeta}(k))=\widetilde{\rho}^*(\mathcal{N}_{\Sigma/S^{[2]}})|_{\widetilde{\zeta}}=\mathcal{O}_{\widetilde{\zeta}}(-2)\oplus\mathcal{O}_{\widetilde{\zeta}}(2).$$
Since $\widetilde{\rho}:\widetilde{\zeta}\rightarrow \zeta$ is a two-to-one cover, we obtain $k=1$, that is:
$$\mathcal{N}_{\Sigma/S^{[2]}}|_{\zeta}=\mathcal{O}_{\zeta}(-1)\oplus\mathcal{O}_{\zeta}(1).$$
By \cite[Chapter V, Proposition 2.3 and Proposition 2.9]{Hartshorne}, there exists a section $D_0$ such that
 $\Pic H_1=\Z D_0\oplus\Z f$, with $D_0^2=-2$, $f^2=0$ and $D_0\cdot f=1$; 
 $f$ being the class of a fiber.
Moreover, by (\ref{gammadelta}) and using the projection formula, we know that:
$$j^{H_1}_*(D_0)=r_1^*(C)-\ell_{\delta}+a\ell_{\Sigma}.$$
To compute $a$, we only need to compute $D_0\cdot j^{H_1*}(\Sigma_1)$.
We apply the same method used in the proof of Lemma \ref{C0}.
By \cite[Propositions 2.6 and 2.8]{Hartshorne}, we have:
$$D_0\cdot \widetilde{j^{H_1}}^*(\Sigma_1)=D_0\cdot \mathcal{O}_{\mathbb{P}(\mathcal{O}_{\zeta}(-1)\oplus\mathcal{O}_{\zeta}(1))}(-1)=-D_0\cdot(D_0+f)=2-1=1.$$
This proves that $a=-1$.
\end{proof}
\begin{lemme}
Let $d$ be the class of a line in $\overline{\overline{C^{(2)}}}$. Then $\overline{j^{C}}_*(d)=r_1^*(C)-r_1^*(\ell_{\delta})-\ell_{\Sigma}$. 
\end{lemme}
\begin{proof}
For instance $\overline{j^{C}}_*(d)$ corresponds to the class of the strict transform of $C^s$ by $r_1\circ \nu$ for $s\in C$. Let $\overline{C^s}$ be the strict transform of $C^s$ by $\nu$. 
As we have already seen in Lemma \ref{stricttransform}, $\left[\overline{C^s}\right]_{S^{[2]}}=j^{C}_*(d)=C-\ell_{\delta}$. The intersection $\overline{C^{(2)}}\cap\Sigma$ corresponds to the strict transform of $\Delta_{S^{(2)}}^{\iota,C}$ by $\nu$ and by Lemma \ref{deltaiota}, it has the class of a line in $\overline{C^{(2)}}$. Hence $\overline{C^s}$ intersects $\Sigma$ transversely in one point and we obtain:
$\overline{j^{C}}_*(d)=r_1^*(C-\ell_{\delta})-\ell_{\Sigma}$. 
\end{proof}
\subsubsection*{Conclusion on wall divisors}
\begin{lemme}
The extremal curves of $M'$ have classes $\pi_{1*}r_1^*\ell_{\delta}$, $\pi_{1*}\ell_{\Sigma}$ and $\pi_{1*}(r_1^*(C)-r_1^*(\ell_{\delta})-\ell_{\Sigma})$. 
\end{lemme}
\begin{proof}
Our previous investigation on curves in $N_1$ show that the 
extremal curves in $N_1$ have classes $r_1^*\ell_{\delta}$,
$\ell_{\Sigma}$ and $r_1^*(C)-r_1^*(\ell_{\delta})-\ell_{\Sigma}$.
This implies that the extremal curves in $M'$ have classes $\pi_{1*}r_1^*\ell_{\delta}$, $\pi_{1*}\ell_{\Sigma}$ and $\pi_{1*}(r_1^*(C)-r_1^*(\ell_{\delta})-\ell_{\Sigma})$. 
\end{proof}
We can compute their dual divisors to obtain wall divisors with Proposition \ref{extremalray}.
\begin{prop}\label{walldiv1}
The divisors $\delta'$, $\Sigma'$, $D_C'$ and $D_C'-\frac{1}{2}(\delta'+\Sigma')$ are wall divisors in $M'$.
Moreover, they satisfy the following numerical properties:
\begin{itemize}
\item
$q_{M'}(\delta')=q_{M'}(\Sigma')=q_{M'}(D_C')=-4$
and $\div(\delta')=\div(\Sigma')=\div(D_C')=2$;
\item
$q_{M'}\left(D_C'-\frac{1}{2}(\delta'+\Sigma')\right)=-6$ and $\div\left[D_C'-\frac{1}{2}(\delta'+\Sigma')\right]=2$.
\end{itemize}
\end{prop}
\begin{proof}
By Lemma \ref{dualdeltasigma}, $\frac{1}{2}\delta'$ and $\frac{1}{2}\Sigma'\in H^2(M',\bQ)$ are the duals of
$\pi_{1*}(r_1^*(\ell_{\delta}))$ and $\pi_{1*}(\ell_{\Sigma})$ respectively. Moreover by Lemma \ref{intersection}, we know that $D_C\cdot C=-2$,
hence by (\ref{Smith2}): $$D_C'\cdot \pi_{1*}(r_1^*(C))=-4.$$ 
Moreover, since $D_C=j(C)$ where $j$ is the isometric embedding $H^2(S,\Z)\hookrightarrow H^2(S^{[2]},\Z)$, we have: $q_{S^{[2]}}(D_C)=-2$. So by Theorem \ref{BBform} (ii): 
$$q_{M'}(D_C')=-4.$$
We obtain that $D_C'$ is the dual of $\pi_{1*}(r_1^*(C))$. Then $D_C'-\frac{1}{2}(\delta'+\Sigma')$ is the dual of $\pi_{1*}(r_1^*(C)-r_1^*(\ell_{\delta})-\ell_{\Sigma})$.

By Proposition \ref{extremalray} this proves that $\delta'$, $\Sigma'$ and $D_C'-\frac{1}{2}(\delta'+\Sigma')$ are wall divisors in $M'$.
Their claimed numerical properties are given by Theorem \ref{BBform} (iii) and Remark \ref{div}.

It remains to show that $D_C'$ is a wall divisor. By Proposition \ref{caca}, since $D_C'$ is a uniruled divisor, we have $(D_C',\alpha)_{q_{M'}}\geq0$ for all $\alpha\in\mathcal{B}\mathcal{K}_{M'}$. Since $\mathcal{B}\mathcal{K}_{M'}$ is open, it follows that $(D_C',\alpha)_{q_{M'}}>0$ for all $\alpha\in\mathcal{B}\mathcal{K}_{M'}$. 
Now, we assume that there exist $g\in \Mon^2_{\Hdg}(M')$ and $\alpha\in\mathcal{B}\mathcal{K}_{M'}$ such that $(g(D_C'),\alpha)_{q_{M'}}=0$ and we will find a contradiction. Since $g\in \Mon^2_{\Hdg}(M')$ and $\Pic(M')=\Z D_C'\oplus \Z \frac{\delta'+\Sigma'}{2}\oplus\Z\frac{\delta'-\Sigma'}{2}$, there are only 6 possibilities: 
$$g(D_C')=\left\{
    \begin{array}{ll}
        \pm D_C' & \text{or}\\
        \pm \delta' & \text{or}\\
				\pm \Sigma'. &
    \end{array}
\right.$$
Since  $(D_C',\alpha)_{q_{M'}}\neq0$,  $(\delta',\alpha)_{q_{M'}}\neq0$ and  $(\Sigma',\alpha)_{q_{M'}}\neq0$. This leads to a contradiction. 
\end{proof}
\subsection{When the K3 surface $S$ used to construct the Nikulin orbifold contains two rational curves swapped by the involution}\label{sec:twocurves}
\subsubsection*{Framework}
Let $\Lambda_{K3}:=U^3\oplus^\bot E_8(-1)\oplus^\bot E_8(-1)$ be the K3 lattice. We fix for all this section three embeddings in $\Lambda_{K3}$ of three lattices $\mathcal{U}\simeq U^3$, $E_1\simeq E_8(-1)$ and $E_2\simeq E_8(-1)$ such that $\Lambda_{K3}\simeq \mathcal{U}\oplus^\bot E_1\oplus^\bot E_2$.
We consider $i$ the involution on $\Lambda_{K3}$ which exchanges $E_1$ and $E_2$ and fixes the lattice $\mathcal{U}$.
We consider $C\in E_1$ such that $C^2=-2$. We define $E^a:=\left\{\left.e-i(e)\right|\ e\in E_1\right\}\simeq E_8(-2)$.
By the surjectivity of the period map (see for instance Theorem \ref{mainGTTO}), we can choose a K3 surface $S$ such that 
$$\Pic S=C\oplus E^a.$$
Then $\Pic S$ contains only two rational curves: one of class $C$ an the other of class $i(C)=C-(C-i(C))$.
It follows from Example \ref{examplewall} and Corollary \ref{cor:desrK} that there exists $\alpha\in \mathcal{K}_S$ invariant under the action of $i$.
Hence Proposition \ref{involutionE8} provides an involution $\iota$ on $S$ such that $\iota^*=i$.
Of course, we can refer to older results on K3 surfaces to construct $\iota$; however, we thought it simplest for the reader to refer to results stated in this paper.

As in Section \ref{onecurve}, the objective is to determine wall divisors of the Nikulin orbifold $M'$ obtained from the pair $(S^{[2]},\iota^{[2]})$ (see Section \ref{M'section}). 

We keep the notation introduced above and in Section \ref{5nota}.
In particular, we denote by $C$ and $\iota(C)$ the two curves in $S$.
\subsubsection*{Curves in $S^{[2]}$}
First, we determine the curves in $S^{(2)}$.
\begin{lemme}\label{S2curves2}
There are 5 cases for irreducible curves in $S^{(2)}$:
\begin{itemize}
\item[(1)]
the curves $C^s:=\rho(C\times s)=\rho(s\times C)$ with $s\notin C\cup\iota(C)$;
\item[(2)]
the curves $\iota(C)^s:=\rho(\iota(C)\times s)=\rho(s\times \iota(C))$ with $s\notin C\cup\iota(C)$;
\item[(3)]
curves in $C^{(2)}\simeq \mathbb{P}^2$;
\item[(4)]
curves in $\iota(C)^{(2)}\simeq \mathbb{P}^2$;
\item[(5)]
curves in $\iota(C)\times C=C\times \iota(C)\simeq \mathbb{P}^1\times\mathbb{P}^1$.
\end{itemize}
\end{lemme}
\begin{proof}
Same proof as Lemma \ref{S2curves}.
\end{proof}
It follows four cases in $S^{[2]}$.
\begin{lemme}\label{curveS22}
We have the following four cases for irreducible curves in $S^{[2]}$:
\begin{itemize}
\item[(0)]
Curves which are fibers of the exceptional divisor $\Delta\rightarrow \Delta_{S^{(2)}}$. As in Section \ref{genericM'}, we denote these curves by $\ell_{\delta}^s$,
where $s=\nu(\ell_{\delta}^s)$ and we denote their classes by $\ell_{\delta}$.
\item[(1)]
Curves which are strict transforms of $C^s$ with $s\notin C\cup\iota(C)$. We denote the class of these curves by $C$. Note that $C=\nu^*(\left[C^s\right])$ for $s\notin C$.
\item[(2)]
Curves which are strict transforms of $\iota(C^s)$ with $s\notin C\cup\iota(C)$. The class of these curves is $\iota^*(C)$.
\item[(3a)]
Curves contained in $H_2$. 
\item[(3b)]
Curves contained in $\overline{C^{(2)}}$. 
\item[(4a)]
Curves contained in $\iota(H_2)$. 
\item[(4b)]
Curves contained in $\iota(\overline{C^{(2)}})$. 
\item[(5)]
Curves in $\iota(C)\times C$.
\end{itemize}
\end{lemme}
\begin{proof}
The proof is similar as the one of Lemma \ref{curveS2}. We only remark in addition that $\nu^{-1}(\iota(C)\times C)\simeq\iota(C)\times C$ because $C$ and $\iota(C)$ do not intersect; hence $\iota(C)\times C$ does not intersect $\Delta_{S^{(2)}}$. So by an abuse of notation, we still denote $\nu^{-1}(\iota(C)\times C)$ by $\iota(C)\times C$.
\end{proof}
\subsubsection*{Curves in $N_1$}
\begin{lemme}\label{curvesN12}
We have the following cases for irreducible curves in $N_1$:
\begin{itemize}
\item[(00)]
Curves contracted to a point by $r_1$. 
They are fibers of the exceptional divisor $\Sigma_1\rightarrow \Sigma$ and their class is $\ell_{\Sigma}$.
\item[(0)]
Curves sent to $\ell_{\delta}^s$ by $r_1$. An extremal such a curve has class $r_1^*(\ell_{\delta})$ by Lemma \ref{Rdeltalemma}.
\item[(1)]
Curves sent to $C^{s}$ by $r_1$ with $s\notin C\cup \iota(C)$.
They are curves of class $r_1^*(C)$.
\item[(2)]
Curves sent to $\iota(C^{s})$ by $r_1$ with $s\notin C\cup \iota(C)$.
They are curves of class $r_1^*(\iota^*(C))$.
\item[(3a)]
Curves contained in $\overline{H_2}$, the strict transform of $H_2$ by $r_1$.
\item[(3b)]
Curves contained in $\overline{\overline{C^{(2)}}}$, the strict transform of $\overline{C^{(2)}}$ by $r_1$.
\item[(4a)]
Curves contained in $\iota(\overline{H_2})$, the strict transform of $\iota(H_2)$ by $r_1$.
\item[(4b)]
Curves contained in $\iota(\overline{\overline{C^{(2)}}})$, the strict transform of $\iota(\overline{C^{(2)}})$ by $r_1$.
\item[(5a)]
Curves contained in $\overline{\iota(C)\times C}$, the strict transform of $\iota(C)\times C$ by $r_1$.
\item[(5b)]
Curves contained in $r_1^{-1}(\overline{\iota(C)\times C}\cap\Sigma)$.
\end{itemize}
\end{lemme}
\begin{proof}
The proof is similar to the proof of Lemma \ref{curvesN1}; the difference is that $H_2$, $\iota(H_2)$, $\overline{C^{(2)}}$  and $\iota(\overline{C^{(2)}})$ do not intersect $\Sigma$. Only $\iota(C)\times C$ intersects $\Sigma$.
\end{proof}
Now, we are going to determine the classes of all these curves.
\begin{lemme}
\begin{itemize}
\item[(3a)]
The extremal curves of $\overline{H_2}$ have classes $2(r_1^*(C)-r_1^*(\ell_{\delta}))$ and $r_1^*(\ell_{\delta})$ in $N_1$.
\item[(3b)]
The extremal curves of $\overline{\overline{C^{(2)}}}$ have class $r_1^*(C)-r_1^*(\ell_{\delta})$ in $N_1$.
\item[(4a)]
The extremal curves of $\iota(\overline{H_2})$ have classes $2(\iota^*(r_1^*(C))-r_1^*(\ell_{\delta}))$ and $r_1^*(\ell_{\delta})$ in $N_1$.
\item[(4b)]
The extremal curves of $\iota(\overline{\overline{C^{(2)}}})$ have class $\iota^*(r_1^*(C))-r_1^*(\ell_{\delta})$ in $N_1$.
\item[(5a)]
The extremal curves of $\overline{\iota(C)\times C}$ have class $r_1^*(C)-\ell_{\Sigma}$ and $r_1^*(\iota^*(C))-\ell_{\Sigma}$ in $N_1$.
\end{itemize}
\end{lemme}
\begin{proof}
\begin{itemize}
\item
Since $H_2$ and $\iota(H_2)$ do not intersect $\Sigma$, (3a) and (4a) are consequences of Lemma \ref{C0}.
\item
Similarly, since $\overline{C^{(2)}} $ and $\iota(\overline{C^{(2)}})$ do not intersect $\Sigma$, (3b) and (4b) are consequences of Lemma \ref{stricttransform}.
\item
We have $\overline{\iota(C)\times C}\simeq \iota(C)\times C\simeq \mathbb{P}^1\times\mathbb{P}^1$. 
Let $j^{\iota}:\overline{\iota(C)\times C}\hookrightarrow N_1$ be the embedding in $N_1$. We want to compute the classes
$j^{\iota}_*(\left\{x\right\}\times \mathbb{P}^1)$ and $j^{\iota}_*(\mathbb{P}^1\times \left\{x\right\})$, where $\left\{x\right\}$ is just the class of a point in $\mathbb{P}^1$.
This corresponds to compute the strict transform by $r_1$ of $C^s$ with $s\in\iota(C)$ and the strict transform of $\iota(C^s)$ with $s\in C$. Since $C^s$ and $\iota(C^s)$ intersect $\Sigma$ in one point, we obtain our result.
\end{itemize}
\end{proof}
\begin{lemme}\label{H2'}
The surface $r_1^{-1}(\overline{\iota(C)\times C}\cap\Sigma)$ is a Hirzebruch surface that we denote by $H_2'$. 
Let $f$ be a fiber of $H_2'$. There exists a section $C_0'$ which
satisfies:  $\Pic H_2'=\Z C_0'\oplus\Z f$, $C_0'^2=-4$, $f^2=0$ and $C_0'\cdot f=1$.
Moreover, $j^{H_2'}_*(C_0')=r_1^*(C)+r_1^*(\iota^*(C))-2\ell_{\Sigma}$, where $j^{H_2'}:H_2'\hookrightarrow N_1$ is the embedding.
\end{lemme}
\begin{proof}
We denote by $\zeta$ the curve $\overline{\iota(C)\times C}\cap\Sigma$. This curve is the strict transform of $\Delta_{S^{(2)}}^{\iota,C}$ by $\nu$. 
Since $\Delta_{S^{(2)}}^{\iota,C}$ does not intersect $\Delta$, its class in $S^{[2]}$ is:
\begin{equation}
\left[\zeta\right]=C+\iota^*(C).
\label{gammadelta2}
\end{equation}
To understand which Hirzebruch surface $H_2'$ is, we are going to compute $\mathcal{N}_{\Sigma/S^{[2]}}|_{\zeta}$.
Let $\Delta_{\widetilde{S^2}}^{\iota,C}$ be the strict transform of $\Delta_{S^{2}}^{\iota,C}$ by $\widetilde{\nu}$.
We have $\zeta\simeq \Delta_{\widetilde{S^2}}^{\iota,C}\simeq \Delta_{S^{2}}^{\iota,C}$.

Hence: $$\mathcal{N}_{\Sigma/S^{[2]}}|_{\zeta}=\widetilde{\rho}^*(\mathcal{N}_{\Sigma/S^{[2]}})|_{\Delta_{\widetilde{S^2}}^{\iota,C}}=\widetilde{\nu}^*(\mathcal{N}_{S_{\iota}/S^2}|_{\Delta_{S^{2}}^{\iota,C}}).$$
As in the proof of Lemma \ref{Hirzebruch}, we have:
$$\mathcal{N}_{S_{\iota}/S^2}|_{\Delta_{S^{2}}^{\iota,C}}\simeq T_S|_C\simeq \mathcal{O}_C(-2)\oplus\mathcal{O}_C(2).$$
Hence $$\mathcal{N}_{\Sigma/S^{[2]}}|_{\zeta}=\mathcal{O}_{\zeta}(-2)\oplus\mathcal{O}_{\zeta}(2).$$

By \cite[Chapter V, Proposition 2.3 and Proposition 2.9]{Hartshorne},
we know that there exists a section $C_0'$ of $H_2'$ such that $\Pic H_2'=\Z C_0'\oplus\Z f$, with $C_0'^2=-4$, $f^2=0$ and $C_0'\cdot f=1$; 
 $f$ being the class of a fiber.

By (\ref{gammadelta2}) and using the projection formula, we know that:
$$j^{H_2'}_*(C_0')=r_1^*(C)+r_1^*(\iota^*(C))+a\ell_{\Sigma}.$$
To compute $a$, we only need to compute $C_0'\cdot j^{H_2'*}(\Sigma_1)$.
We apply the same method used in the proof of Lemma \ref{C0}.
By \cite[Propositions 2.6 and 2.8]{Hartshorne}, we have:
$$C_0'\cdot \widetilde{j^{H_2'}}^*(\Sigma_1)=C_0'\cdot \mathcal{O}_{\mathbb{P}(\mathcal{O}_{\zeta}(-2)\oplus\mathcal{O}_{\zeta}(2))}(-1)=-C_0'\cdot(C_0'+2f)=4-2=2.$$
This proves that $a=-2$.
 
\end{proof}
\subsubsection*{Conclusion on wall divisors}
\begin{lemme}\label{extrem2}
The extremal curves of $M'$ have classes $\pi_{1*}r_1^*\ell_{\delta}$, $\pi_{1*}\ell_{\Sigma}$, $\pi_{1*}(r_1^*(C)-r_1^*(\ell_{\delta}))$ and $\pi_{1*}(r_1^*(C)-\ell_{\Sigma})$. 
\end{lemme}
\begin{proof}
It is obtained by taking the image by $\pi_{1*}$ of the classes of the extremal curves in $N_1$.
\end{proof}
We can compute their dual divisors to obtain wall divisors with Proposition \ref{extremalray}.
\begin{prop}\label{prop:twocurves}
The divisors $\delta'$, $\Sigma'$, $D_C'$, $2D_C'-\delta'$ and $2D_C'-\Sigma'$ are wall divisors in $M'$.
Moreover, they satisfy the following numerical properties:
\begin{itemize}
\item
$q_{M'}(\delta')=q_{M'}(\Sigma')=-4$
and $\div(\delta')=\div(\Sigma')=2$;
\item
$q_{M'}(D_C')=-2$ 
and $\div(D_C')=1$;
\item
$q_{M'}\left(2D_C'-\delta'\right)=q_{M'}\left(2D_C'-\Sigma'\right)=-12$ and $\div\left[2D_C'-\delta'\right]=\div\left[2D_C'-\Sigma'\right]=2$.
\end{itemize}
\end{prop}
\begin{proof}
By Lemma \ref{dualdeltasigma}, $\frac{1}{2}\delta'$ and $\frac{1}{2}\Sigma'\in H^2(M',\bQ)$ are the duals of
$\pi_{1*}(r_1^*(\ell_{\delta}))$ and $\pi_{1*}(\ell_{\Sigma})$ respectively. Moreover by Lemma \ref{intersection}, 
we know that $D_C\cdot C=-2$,
hence by (\ref{Smith2}): $$\pi_{1*}(r_1^*(D_C+\iota^*(D_C)))\cdot \pi_{1*}(r_1^*(C+\iota^*(C)))=-8.$$ 
So $$D_C'\cdot \pi_{1*}(r_1^*(C))=-2.$$
Moreover, since $D_C=j(C)$ where $j$ is the isometric embedding $H^2(S,\Z)\hookrightarrow H^2(S^{[2]},\Z)$, we have: $q_{S^{[2]}}(D_C)=-2$. So by Theorem \ref{BBform} (ii): 
$$q_{M'}(\pi_{1*}(r_1^*(D_C+\iota^*(D_C))))=-8.$$
Hence:
\begin{equation}
q_{M'}(D_C')=-2.
\label{DC'2}
\end{equation}
We obtain that $D_C'$ is the dual of $\pi_{1*}(r_1^*(C))$. Then $D_C'-\frac{1}{2}\delta'$ is the dual of $\pi_{1*}(r_1^*(C)-r_1^*(\ell_{\delta}))$ and $D_C'-\frac{1}{2}\Sigma'$ is the dual of $\pi_{1*}(r_1^*(C)-\ell_{\Sigma})$.

By Proposition \ref{extremalray} this proves that $\delta'$, $\Sigma'$, $2D_C'-\delta'$ and $2D_C'-\Sigma'$ are wall divisors in $M'$.
Their claimed numerical properties are given by Theorem \ref{BBform}(iii), Remark \ref{div} and (\ref{DC'2}).

It remains to prove that $D_C'$ is a wall divisor. The proof is very similar to the one of Proposition \ref{walldiv1}. For the same reason, we have $(D_C',\alpha)_{q_{M'}}>0$ for all $\alpha\in\mathcal{B}\mathcal{K}_{M'}$. 
Now, we assume that there exists $g\in \Mon^2_{\Hdg}(M')$ and $\alpha\in\mathcal{B}\mathcal{K}_{M'}$ such that $(g(D_C'),\alpha)_{q_{M'}}=0$ and we will find a contradiction. Since $C\in E_1$, we have $\div(D_C')=1$. Since $g\in \Mon^2_{\Hdg}(M')$ and $\Pic(M')=\Z D_C'\oplus \Z \frac{\delta'+\Sigma'}{2}\oplus\Z\frac{\delta'-\Sigma'}{2}$, it follows that there are only 2 possibilities: 
$$g(D_C')=\pm D_C',$$
because $\div(D_C')=1$ and $\div(\frac{\delta'+\Sigma'}{2})=2$.
Since $(D_C',\alpha)_{q_{M'}}\neq0$, this leads to a contradiction. 
\end{proof}
\begin{rmk}\label{Remark:twocurves}
Note that in this case, $D_C'-\frac{\delta'+\Sigma'}{2}$ is not a wall divisor. Indeed, by Lemma
\ref{extrem2}, the class $-2D_C'-\delta'-\Sigma'$ is the projection on $\Pic M'$ of a K\"ahler class. However,
observe that $\left(D_C'-\frac{\delta'+\Sigma'}{2},-2D_C'-\delta'-\Sigma'\right)_q=0$. 
\end{rmk}
\subsection{Wall divisors on a Nikulin orbifold constructed from a specific elliptic K3 surface}\label{sec:elliptic}
As before, we consider the K3 lattice $\Lambda_{K3}:=U^3\oplus^\bot E_8(-1)\oplus^\bot E_8(-1)$ with  
the three embedded lattices $\mathcal{U}\simeq U^3$, $E_1\simeq E_8(-1)$ and $E_2\simeq E_8(-1)$ such that $\Lambda_{K3}\simeq \mathcal{U}\oplus^\bot E_1\oplus^\bot E_2$.
The involution $i$ on $\Lambda_{K3}$ is still the involution which exchanges $E_1$ and $E_2$ and fixes the lattice $\mathcal{U}$.
As before, we keep $E^a:=\left\{\left.e-i(e)\right|\ e\in E_1\right\}\simeq E_8(-2)$. For simplicity, we denote $E_8(-2):=\left\{\left.e+i(e)\right|\ e\in E_1\right\}$ which is the invariant lattice. 
Let $L_1\in \mathcal{U}$ such that $L_1^2=2$. We use the standard realization
$$E_1\simeq E_8(-1)=\left\{(x_1,\ldots,x_8)\in\Z^8\cup(\Z+\tfrac12)^8\ \middle|\ \sum x_i\in2\Z\right\}$$
with the negative Euclidean form, and denote by $\epsilon_1,\ldots,\epsilon_8$ the standard orthonormal vectors
of the ambient space. Notice that the vectors $\epsilon_j$ do not themselves belong to $E_8(-1)$. We choose
$e_2^{(0)}=2\epsilon_1\in E_8(-1)$. Thus $(e_2^{(0)})^2=-4$, while
$\epsilon_1\pm\epsilon_j\in E_8(-1)$ are roots for $2\leq j\leq8$.
Using the surjectivity of the period map (see for instance Theorem \ref{mainGTTO}), we choose a K3 surface $S$ satisfying:
$$\Pic S=\Z L_1\oplus^\bot(\Z e_2^{(0)} \oplus E^a).$$
The divisor $L_1$ or $-L_1$ is effective; for simplicity in the notation, we assume that $L_1$ is effective.
Note that the direct sum $\Z e_2^{(0)} \oplus E^a$ is not orthogonal. 
We denote:
$$v_{K3}:=2L_1+e_2,$$
with $e_2:=e_2^{(0)}+i(e_2^{(0)})$. We have $v_{K3}^2=0$ and $\Pic S\supset \Z v_{K3} \oplus^{\bot} E^a$.
The orthogonal complement $\Z e_2^{(0)}\oplus E^a$ contains no element of square $-2$. Consequently, every
effective $(-2)$-class has a positive $L_1$-component, and hence intersects $L_1$ positively. Thus $L_1$ is
ample. Moreover, $v_{K3}$ is primitive and nef. Indeed, if
$R=nL_1+me_2^{(0)}+(e-i(e))$ is an effective $(-2)$-class, then $n>0$, and the equality $R^2=-2$ implies
$m\leq n$; hence $v_{K3}\cdot R=4(n-m)\geq0$. Therefore the primitive isotropic class $v_{K3}$ defines an
elliptic fibration $f$.

As before, it follows from Example \ref{examplewall} and Corollary \ref{cor:desrK} that there exists $\alpha\in \mathcal{K}_S$ invariant under the action of $i$. Hence Proposition \ref{involutionE8} provides an involution $\iota$ on $S$ such that $\iota^*=i$. We consider $M'$ constructed from the pair $(S,\iota)$.

\begin{lemme}\label{lem:elliptic-base}
The involution $\iota$ acts trivially on the base of the elliptic fibration $f:S\rightarrow\Pj^1$; equivalently,
$f\circ\iota=f$.
\end{lemme}
\begin{proof}
For $2\leq j\leq8$, set
$$r_j:=\epsilon_1+\epsilon_j,\qquad s_j:=\epsilon_1-\epsilon_j,$$
so that $r_j$ and $s_j$ are roots and $r_j+s_j=e_2^{(0)}$. The classes
$$R_j:=L_1+r_j+i(s_j),\qquad i(R_j)=L_1+s_j+i(r_j)$$
belong to $\Pic S$, have square $-2$, and satisfy $R_j+i(R_j)=v_{K3}$. Since $L_1$ is ample and
$L_1\cdot R_j=2$, Riemann--Roch shows that $R_j$ is effective. Moreover, it is represented by an irreducible
rational curve: any non-trivial decomposition into effective divisors would have $L_1$-degree at least $4$,
since the $L_1$-degree of every divisor is even. The same holds for $i(R_j)$. Thus each pair
$R_j+i(R_j)$ is a reducible fiber of the elliptic fibration, and these seven fibers are distinct and invariant
under $\iota$. The induced involution on the base $\Pj^1$ therefore has at least seven fixed points. Since a
non-trivial involution of $\Pj^1$ has exactly two fixed points, it must be the identity.
\end{proof}

The K3 surface $S$ contains a $(-2)$-curve of class $L_1+e_2^{(0)}$. We denote this curve $\gamma$. The class of $\iota(\gamma)$ is $L_1+i(e_2^{(0)})$. Hence $\gamma\cup\iota(\gamma)$ has class $v_{K3}$ and provides a fiber of the elliptic fibration $f:S\rightarrow \Pj^1$. Moreover, we have:
$$\left[\gamma\right]\cdot \left[\iota(\gamma)\right]=(L_1+e_2^{(0)})\cdot(L_1+i(e_2^{(0)}))=2.$$
We denote by $\overline{\gamma}$ the class $\nu^*(\left[\gamma^s\right])$, with
$\gamma^s:=\gamma\times\left\{s\right\}$. 
We also denote $D_{\gamma}:=j(\gamma)$ and $D_{\gamma}':={\pi_1}_*(r_1^*(D_{\gamma}))$.

We consider the following divisor in $S^{(2)}$:
$$A:=\left\{\left.\left\{x,y\right\}\in S^{(2)}\right|\ f(x)=f(y)\right\},$$
with $f$ the elliptic fibration $S\rightarrow \Pj^1$ associated to $v_{K3}$. By Lemma
\ref{lem:elliptic-base}, the divisor $A$ contains the surface $\Sigma$.
We denote by $A'$ the image by $\pi_1$ of the strict transform of $A$ by $r_1\circ\nu$.
\begin{lemme}\label{ellipticlemma}
\begin{itemize}
We have:
\item[(i)]
The class of the strict transform of $\gamma^s$ by $r_1\circ \nu$ is $r_1^{*}(\overline{\gamma}-\ell_{\delta})-\ell_{\Sigma}$, for $s\in \gamma\cap \iota(\gamma)$.
\item[(ii)]
The dual of $\pi_{1*}(r_1^*(\overline{\gamma}))$ is $D_{\gamma}'$.
\item[(iii)]
The divisor $D_{\gamma}'$ has square 0 and divisibility 1.
\item[(iv)]
The divisor $A'$ has class $D_{\gamma}'-\frac{\delta'+\Sigma'}{2}$.
\end{itemize}
\end{lemme}
\begin{proof}
\begin{itemize}
\item[(i)]
The curve $\gamma^s$ intersects $\Delta_{S^{(2)}}$ at $\{s,s\}$ and $\Sigma$ at $\{s,\iota(s)\}$. If $s\neq\iota(s)$, these points are distinct and both intersections have multiplicity $1$. Suppose now that $s=\iota(s)$. We may choose local coordinates centered at $s$ in which $\iota(x,y)=(-x,-y)$. Since $\gamma\cdot\iota(\gamma)=2$, we may write locally
$$\gamma=\{y=ax^2+O(x^3)\},\qquad \iota(\gamma)=\{y=-ax^2+O(x^3)\}$$
with $a\neq0$. To make the local computation explicit, write $(x,y)$ and $(u,v)$ for the coordinates of the two factors of $S\times S$, and set
$$a_1=\frac{x+u}{2},\quad a_2=\frac{y+v}{2},\quad d_1=\frac{x-u}{2},\quad d_2=\frac{y-v}{2}.$$
The diagonal is defined by $d_1=d_2=0$. In the chart of its blow-up where we set $t=d_2/d_1$, the strict transform of the lift $\gamma\times\{s\}$ of $\gamma^s$ is parametrized by
$$a_1=d_1=\frac{x}{2},\qquad a_2=d_2=\frac{ax^2+O(x^3)}{2},\qquad t=ax+O(x^2).$$
On the other hand, the graph of $\iota(x,y)=(-x,-y)$ is defined in these coordinates by $a_1=a_2=0$; hence its strict transform is still defined by $a_1=a_2=0$. Along the strict transform of $\gamma\times\{s\}$, the ideal $(a_1,a_2)$ therefore pulls back to $(x)$ and vanishes to order $1$. Thus this curve meets the strict transform of the graph, and consequently the strict transform of $\Sigma$, with multiplicity $1$. In both cases, the successive strict transforms subtract once the classes $\ell_\delta$ and $\ell_\Sigma$, which proves the statement.
\item[(ii)]
Let $w\in \Z L_1\oplus^\bot E_8(-2)$ be an invariant element under the action of $\iota$ such that $\gamma\cdot w=:k$.
Then 
\begin{equation}
  \overline{\gamma}\cdot j(w)=\gamma\cdot w
  =k. 
\label{dualgamma}
\end{equation}
We set $w':=\pi_{1*}(r_1^*(j(w)))$.
It follows by (\ref{Smith2}) that $\pi_{1*}(r_1^*(\overline{\gamma}+\iota^{[2]*}(\overline{\gamma})))\cdot w'=4k$. Hence  
\begin{equation}
\pi_{1*}(r_1^*(\overline{\gamma}))\cdot w'=2k.
\label{gammaprime}
\end{equation}

On the other hand, by Theorem \ref{BBform} (ii), Remark \ref{Smithcomute} and (\ref{dualgamma}), we have: 
$$(\pi_{1*}r_1^{*}(D_\gamma+\iota^{[2]*}(D_\gamma)),\pi_{1*}(r_1^*(j(w))))_{q_{M'}}=4k.$$
Hence:
\begin{equation}
(D_\gamma',w')_{q_{M'}}=2k.
\label{w}
\end{equation}
We obtain that $D_\gamma'$ is the dual of $\pi_{1*}(r_1^*(\overline{\gamma}))$.
\item[(iii)]
We have by Theorem \ref{BBform} (ii) and Remark \ref{Smithcomute} $$q_{M'}(\pi_{1*}r_1^{*}(D_\gamma+\iota^{[2]*}(D_\gamma)))=2q_{S^{[2]}}(D_\gamma+\iota^{[2]*}(D_\gamma))=0.$$
Hence $q_{M'}(\pi_{1*}r_1^{*}(D_\gamma))=0$.
To prove that $\div(D_\gamma')=1$,
we choose a specific $w$: $w=w^{(0)}+i(w^{(0)})$ such that $w^{(0)}\in E_1$ and $w^{(0)}\cdot e_2^{(0)}=1$; then $w\cdot\gamma=1$. Then by (\ref{w}) we have $(D_\gamma',w')_{q_{M'}}=2$. However $w'$ is divisible by 2. We obtain our result.
\item[(iv)]
Since $A$ is invariant under the action of $\iota$ and considering the intersection with $\rho_*(w\times \left\{pt\right\})$, we see that the class of $A$ is given by $\rho_*(\left[S\times v_{K3}\right])$. Then the strict transform $\widetilde{A}$ by $r_1\circ \nu$ has class $r_1^*(j(v_{K3})-\delta)-\Sigma_1$ because $A$ contains $\Delta_{S^{(2)}}$ and, by Lemma \ref{lem:elliptic-base}, also $\Sigma$. We recall that $\pi_{1*}r_1^*(j(v_{K3}))=2 D_\gamma'$. Therefore and since $\widetilde{A}\rightarrow A'$ is a double cover, $A'$ has class $\frac{1}{2}(2D_\gamma'-\delta'-\Sigma')$.
\end{itemize}
\end{proof}
  \begin{lemme}\label{monolemma}
	The reflection $R_{A'}$ is a monodromy operator and $R_{A'}(\Sigma')=2D_{\gamma}'-\delta'$.
	\end{lemme}
	\begin{proof}
	The reflection $R_{A'}$ is a monodromy operator by \cite[Theorem 3.10]{Lehn2}.
	$$R_{A'}(\Sigma')=\Sigma'-\frac{2(\Sigma',A')_q}{q(A')}A'=\Sigma'+2A'=2D_\gamma'-\delta'.$$
	\end{proof}
\begin{lemme}\label{extremallemma}
 There exists an extremal curve of $M'$ that is written $a\pi_{1*}r_1^*(\overline{\gamma})+\frac{b}{2}\pi_{1*}r_1^*(\ell_{\delta})+\frac{c}{2}\pi_{1*}(\ell_{\Sigma})$ with $(a,b,c)\in\Z^3$ and $a>0$.
\end{lemme}
\begin{proof}

We also denote by $L_1$ a curve in $S$ of class $L_1$. We set
$L_1^s=L_1\times\{s\}$ and $\overline{L}_1$ the class
$\nu^*([L_1^s])$.

Let
\[
\vartheta
=
a\pi_{1*}r_1^*(\overline{\gamma})
+\frac{b}{2}\pi_{1*}r_1^*(\ell_\delta)
+\frac{c}{2}\pi_{1*}(\ell_\Sigma)
+d\pi_{1*}r_1^*(\overline{L}_1)
\]
be the class of an irreducible curve in $H^{3,3}(M',\mathbb Z)$,
with possibly $a,b,c,d$ in $\mathbb Q$.

We first observe that $a\in\mathbb Z$. Indeed, as in the proof of
Lemma~\ref{ellipticlemma}, we can choose $w$ such that $k=1$ in equation~\eqref{gammaprime}.
Moreover, $w'$ is divisible by $2$. We obtain
\[
\vartheta\cdot\frac{1}{2}w'
=
a\pi_{1*}r_1^*(\overline{\gamma})\cdot\frac{1}{2}w'
=
a.
\]
Therefore $a\in\mathbb Z$.

We now consider the pull-back of $\vartheta$ by $\pi_1$. By
equation~\eqref{Smith}, we have
\[
\pi_1^*\pi_{1*}r_1^*(\overline{\gamma})
=
r_1^*(\overline{\gamma})
+\iota_1^{[2]*}r_1^*(\overline{\gamma})
=
r_1^*\bigl(
\overline{\gamma}
+\iota^{[2]*}(\overline{\gamma})
\bigr).
\]
On the other hand, since $L_1$ is invariant under $\iota$, we have
\[
\pi_1^*\pi_{1*}r_1^*(\overline{L}_1)
=
2r_1^*(\overline{L}_1).
\]
  The terms involving $\ell_\delta$ and $\ell_\Sigma$ are contracted
  after taking the push-forward by $\nu\circ r_1$. We thus obtain an
  effective curve cycle $B$ on $S^{(2)}$  of class $2d\left[L_1^s\right]+a\left(\left[\gamma^s\right]+\left[\iota(\gamma)^s\right]\right)$. Consider the incidence
  correspondence
  \[
  \xymatrix{
  &S\times S\ar[dl]_{p_1}\ar[dr]^{\rho}&\\
  S&&S^{(2)},
  }
  \]
  and associate to $B$ the effective curve cycle
  $p_{1*}\rho^*B$ on $S$. Since this correspondence sends the class of
  $C^s$ to the class of $C$, the class of $p_{1*}\rho^*B$ is
  \[
  a\bigl(\gamma+\iota(\gamma)\bigr)+2dL_1
  =av_{K3}+2dL_1.
  \]
  Since $v_{K3}$ is nef, its intersection with this effective cycle is
  non-negative. Therefore
  \[
  0\leq (av_{K3}+2dL_1)\cdot v_{K3}=8d,
  \]
  and hence $d\geq0$.

Since there exist classes of curves with $d=0$ and since for all classes of curves $d\ge0$, there exist an extremal curve with $d=0$. 
Moreover, for such a curve we have $a>0$, since
\[
av_{K3}\cdot L_1=4a>0.
\]

It remains to show that $b,c\in\mathbb Z$. By Lemma~\ref{dualdeltasigma}, we know that
\[
\pi_{1*}r_1^*(\ell_\delta)\cdot\delta'
=
\pi_{1*}(\ell_\Sigma)\cdot\Sigma'
=
-2.
\]
Thus
\[
\vartheta\cdot\delta'
=
\frac{b}{2}\pi_{1*}r_1^*(\ell_\delta)\cdot\delta'
=
-b
\]
and
\[
\vartheta\cdot\Sigma'
=
\frac{c}{2}\pi_{1*}(\ell_\Sigma)\cdot\Sigma'
=
-c.
\]
Therefore $b,c\in\mathbb Z$, which concludes the proof.

\end{proof}
\begin{rmk}\label{wallelliptic}
Let $a\pi_{1*}r_1^*(\overline{\gamma})+\frac{b}{2}\pi_{1*}r_1^*(\ell_{\delta})+\frac{c}{2}\pi_{1*}(\ell_{\Sigma})$ be the class of the extremal curve obtained from Lemma \ref{extremallemma}.
 By Proposition \ref{extremalray}, the line generated by the dual of this curve class contains a primitive wall divisor. According to Lemmas \ref{ellipticlemma} and \ref{dualdeltasigma}, the dual is $aD_{\gamma}'+\frac{b}{4}\delta'+\frac{c}{4}\Sigma'$.
\end{rmk}
\begin{lemme}\label{mainelliptic}
 Let $E=aD_{\gamma}'+\frac{b}{4}\delta'+\frac{c}{4}\Sigma'$ be the previous wall divisor obtained from Remark \ref{wallelliptic}, rescaled if necessary so that $E$ is primitive in $H^2(M',\Z)$. Moreover, we assume that $E$ satisfies one of the numerical conditions listed in the statement of Theorem \ref{main}. Then: 
 $$E=D_{\gamma}'-\frac{\delta'+\Sigma'}{2}.$$
\end{lemme}
\begin{proof}
 We have:
 $$q_{M'}(E)=-\frac{(b^2+c^2)}{4}.$$
Considering the numerical conditions of Theorem \ref{main}, there are three possibilities:
\begin{itemize}
 \item[(i)]
 $b=\pm2$ and $c=\pm2$ or
 \item[(ii)]
 $b=\pm4$ and $c=0$ or
 \item[(iii)]
 $b=0$ and $c=\pm4$.
\end{itemize}
In the possibilities (ii) and (iii), $q_{M'}(E)=-4$. In this case, following the conditions in Theorem \ref{main}, we know that $E$ has divisibility 2. Hence by Lemma \ref{ellipticlemma} (iii), $a$ is divisible by 2. 
This corresponds to the extremal rays of curves $a\pi_{1*}r_1^*(\overline{\gamma})\pm2\pi_{1*}r_1^*(\ell_{\delta})$ or $a\pi_{1*}r_1^*(\overline{\gamma})\pm2\pi_{1*}(\ell_{\Sigma})$. However, by Lemma \ref{ellipticlemma} (i) these rays cannot be extremal: writing
$$R_0:=\pi_{1*}\bigl(r_1^*(\overline{\gamma})-r_1^*(\ell_\delta)-\ell_\Sigma\bigr),$$
each of them is a sum of a positive multiple of the effective class $R_0$ and non-negative multiples of
$\pi_{1*}r_1^*(\ell_\delta)$ and $\pi_{1*}(\ell_\Sigma)$.

Therefore, $E=aD_{\gamma}'\pm\frac{1}{2}\delta'\pm\frac{1}{2}\Sigma'$. Moreover, the extremal ray associated to $E$ has class:
$$a\pi_{1*}r_1^*(\overline{\gamma})\pm\pi_{1*}r_1^*(\ell_{\delta})\pm\pi_{1*}(\ell_{\Sigma}).$$
Using the same decomposition with $R_0$, extremality forces $a=1$ and both signs to be negative. Hence the only possibility is $R_0$, and its dual is
$D_\gamma'-\frac{\delta'+\Sigma'}2$.
\end{proof}
\begin{rmk}
In Section \ref{endsection}, we will see that $D_{\gamma}'-\frac{\delta'+\Sigma'}{2}$ is indeed a wall divisor.
\end{rmk}
\section{Monodromy orbits}\label{sec:monodromy-orbits}
In this section, we study the orbit of classes in the lattice $H^2(X,\bZ)$ under the action of $\Mon^2(X)$ for an
irreducible symplectic orbifold $X$ of Nikulin-type. Note that since the property of being a wall divisor is
deformation invariant (see Theorem \ref{wall}) we may assume without loss of generality, that $X$ is the Nikulin orbifold $M'$ for a
given K3 surface with a symplectic involution.

The main result of this section is Theorem \ref{thm:9monorb-M'}, which describes a set of representatives in each monodromy orbit.
This will enable us to determine wall divisors for Nikulin-type in the next section by checking only the
representatives.

For completeness, note that we did not determine the precise monodromy orbit of each element. Since we only
used a subgroup of the actual monodromy group, it could happen
that more than one of the elements in Theorem \ref{thm:9monorb-M'} belong to the same orbit.
 
\subsection{Equivalence of lattices} \label{sec:eq-lattices}
The next lemma will be needed in order to apply Propositon \ref{MonoM'} and pass from the monodromy of $K3^{[2]}$-type manifold to the monodromy of Nikulin type orbifolds.
\begin{lemme}\label{lem:twist-gen}
  Let $M$ and $N$ be two unimodular lattices.
  Then $L\coloneqq M\oplus N(2)$ and $\hat{L}\coloneqq M(2)\oplus N$ satisfy the following
  properties:
  There exists a natural identification between lattice automorphisms for both lattices, and
  the rays in both lattices can be naturally identified.
\end{lemme}
\begin{proof}
  Observe that by multiplying the quadratic form of the lattice $L$ by $2$, we obtain a lattice
  $L(2) \iso M(2)\oplus N(4)$, which obviously satisfies that the rays and
  automorphisms are naturally identified  for $L$ and $L(2)$.

  Notice that $N(4)$ can be identified with the sublattice of $N$ consisting of elements of the form $\{2n\,|\,
  n\in N\}$. Therefore, we can naturally include $L (2) \subset \hat{L}$.
  This immediately implies the natural identification of rays in $L$ and $\hat{L}$.

  For the identification of automorphisms, observe that any automorphism $\hat{\varphi}\in \Aut(\hat{L})$
  preserves the sublattice $L(2)$: in fact $L(2)\subset \hat{L}$ consists precisely of those
  (not necessarily primitive) elements whose divisibility is even, and this subset needs to be preserved by
  any automorphism. This yields a natural inclusion $\Aut(\hat{L})\subset \Aut(L(2))\iso
  \Aut(L)$.

  The inverse inclusion is given by the same argument from considering $\hat{L}(2)\subset L$.
\end{proof}
It will be applied for the following specific latties:
\begin{lemme}\label{lem:twist-spec}
We consider:
$$\Lambda\coloneqq \Lambda_{M'}=U(2)^{ 3} \oplus E_8(-1)\oplus (-2) \oplus (-2)$$
and
$$\hat{\Lambda} \coloneqq U^{ 3} \oplus E_8(-2)\oplus (-1) \oplus (-1).$$
There is a natural correspondence between lattice automorphisms for both lattices, and
a natural identification between the rays in both lattices.
\end{lemme}

Fix a K3 surface $S$ with a symplectic involution $\iota$ 
and consider the induced Nikulin orbifold $M'$ associated to $S$.
Further, fix a marking $\varphi_S\colon H^2(S,\bZ) \to \Lambda_{K3}=U^3\oplus E_8(-1)^2$ of $S$ such that
$\iota^*$ corresponds to
swapping the two copies of $E_8(-1)$.
By Remark \ref{rem:invariant-lattice-K3}, the fixed part of $\iota^*$ is isomorphic to
$U^3 \oplus E_8(-2)$.

Note that on $X$, which we choose as the associated orbifold $M'$ to $(S,\iota)$, this induces a marking
$\varphi_{X}\colon H^2(X,\bZ)\to \Lambda=\Lambda_{M'}=U(2)^{ 3} \oplus
E_8(-1)\oplus (-2)^2$ by Theorem \ref{BBform}, where the $U(2)^{ 3} \oplus
E_8(-1)$-part comes from the invariant lattice of the K3 surface (precisely as described in Lemma \ref{lem:twist-gen}), 
and the two $(-2)$-classes correspond to $\frac{\delta' + \Sigma'}{2}$ and $\frac{\delta' - \Sigma'}{2}$.

Therefore,  the sublattice $U^{ 3} \oplus E_8(-2)$  in $\hat{\Lambda}$ can naturally be identified with the fixed part of
$H^2(S,\bZ)$, whereas the generators of square $(-1)$ correspond to $\frac{\hdel + \hSig}{2}$ and
$\frac{\hdel - \hSig}{2}$ for the corresponding elements $\hdel, \hSig \in \hat{\Lambda}$.

With this notation, we can define the  group of $\Mon^2(\hat{\Lambda})$ of monodromy operators for
the lattice $\hat{\Lambda}$: an automorphism $\hat{\varphi} \in \Aut(\hat{\Lambda})$ is in $\Mon^2(\hat{\Lambda})$ if the
corresponding automorphism $\varphi \in \Aut(\Lambda)$ is identified with an element of $\Mon^2(X)$ via the
marking $\varphi_X$.

In the following we will frequently consider the sublattice
$$\hat{\Lambda}_1 \coloneqq U^{ 3} \oplus E_8(-2)\oplus (-2) \oplus (-2) \subset \hat{\Lambda},$$
which replaces the $(-1)\oplus (-1)$-part by the sublattice generated by $\hdel$ and $\hSig$.

Define $$\Mon^2(\hat{\Lambda}_1) \coloneqq \{f \in \Aut (\hat{\Lambda}_1) \,|\, \exists \hat{f} \in \Mon^2(\hat{\Lambda}):
f=\hat{f}|_{\hat{\Lambda}_1}\}.$$

\begin{rmk}
  Note that while there exists an identification 
  $\Mon^2(X) =\Mon^2(\hat{\Lambda})$,
  there exists a natural inclusion  $\Mon^2(\hat{\Lambda}_1) \subseteq \Mon^2(\hat{\Lambda})$ but a priori this
  is not an equality. \TODO{try to understand if this is an equality or not}
\end{rmk}

Note that Proposition \ref{MonoM'} can be reformulated in terms of the lattice $\hat{\Lambda}_1$:
\begin{cor}\label{cor:inheritedMononHat}
  Let $f\in \Mon^2(S^{[2]})$ be a monodromy operator such
  that $f \circ \iota^{[2]*} = \iota^{[2]*} \circ f$ on $H^2(S^{[2]},\bZ)$.
  Let $\hat{f}\in \Aut(\hat{\Lambda}_1)$ be the automorphism defined via the following properties:
  via the marking described above, $\hat{f}$ restricted to $U^{ 3}\oplus E_8(-2)\oplus (-2)$ coincides with the restriction of
  $f$ to the invariant part of the lattice (i.e.~$\hat{f}|_{U^{ 3}\oplus E_8(-2)\oplus (-2)}=
  f|_{H^2(S^{[2]},\bZ)^{\iota^{[2]}}} $) and $\hat{f}(\hSig)=\hSig$.
  Then $\hat{f}\in \Mon^2(\hat{\Lambda}_1)$.
\end{cor}
\begin{proof} This is a straightforward verification: Proposition \ref{MonoM'} gives the inherited monodromy operator
  $f'\in \Mon^2(\Lambda)$ and  $\hat{f}\in \Aut(\hat{\Lambda}_1)$ is precisely the restriction to
  $\hat{\Lambda}_1$ of the corresponding automorphism of $\hat{\Lambda}$ obtained via Lemma \ref{lem:twist-spec}.
\end{proof}

For the proof of Theorem \ref{thm:9monorb-M'} we will study monodromy orbits with respect to successively increasing lattices:
\begin{equation*}\hat{\Lambda}_3\subset \hat{\Lambda}_2 \subset \hat{\Lambda}_1,\quad \text{and finaly}\quad \Lambda_{M'},
\end{equation*}
where $\hat{\Lambda}_3\coloneqq U^{ 3}\oplus (-2)$ with the generator $\hdel$ for the $(-2)$-part,
$\hat{\Lambda}_2\coloneqq U^{ 3}\oplus E_8(-2) \oplus (-2) $, and $\hat{\Lambda}_1$ is as defined above.

Define the following monodromy groups for these lattices.
\begin{align*}
\Mon^2(\hat{\Lambda}_2)&=\{f\in \Aut(\hat{\Lambda}_2)| \exists f_1 \in \Mon^2(\hat{\Lambda}_1) : f=f_1|_{\hat{\Lambda}_2},
f_1|_{\hat{\Lambda}_2^\perp}=\id\}\\
\Mon^2(\hat{\Lambda}_3)&=\{f\in \Aut(\hat{\Lambda}_3)| \exists f_1 \in \Mon^2(\hat{\Lambda}_1) : f=f_1|_{\hat{\Lambda}_3},
f_1|_{\hat{\Lambda}_3^\perp}=\id\}.
\end{align*}
Note that with this definition, there exist natural inclusions
$\Mon^2(\hat{\Lambda}_3)\subseteq \Mon^2(\hat{\Lambda}_2)\subseteq \Mon^2(\hat{\Lambda}_1)$.

\subsection*{Summary of notation}
We summarize the lattices and monodromy groups introduced in this subsection.

\medskip
\noindent\textbf{Lattices.}
\begin{itemize}
\item The K3 lattice and its $\iota^*$-invariant part are
\[
\Lambda_{K3}=U^3\oplus E_8(-1)^2,
\qquad
\Lambda_{K3}^{\iota}=H^2(S,\bZ)^{\iota^*}\simeq U^3\oplus E_8(-2).
\]

\item The Beauville--Bogomolov lattice of the Nikulin orbifold $X$ is
\[
\Lambda=\Lambda_{M'}=U(2)^3\oplus E_8(-1)\oplus(-2)^2
\simeq H^2(X,\bZ).
\]
For the study of its rays, Lemma \ref{lem:twist-spec} replaces it with the auxiliary lattice
\[
\hat{\Lambda}=U^3\oplus E_8(-2)\oplus(-1)^2.
\]
There is no inclusion between $\Lambda$ and $\hat{\Lambda}$; instead, there are natural correspondences
between their rays and between their lattice automorphisms.

\item Inside $\hat{\Lambda}$ we use
\[
\begin{aligned}
\hat{\Lambda}_3&=U^3\oplus\Z\hdel,\\
\hat{\Lambda}_2&=U^3\oplus E_8(-2)\oplus\Z\hdel,\\
\hat{\Lambda}_1&=U^3\oplus E_8(-2)\oplus\Z\hdel\oplus\Z\hSig,
\end{aligned}
\]
with
\[
\hat{\Lambda}_3\subset\hat{\Lambda}_2\subset\hat{\Lambda}_1\subset\hat{\Lambda}.
\]
Here $q(\hdel)=q(\hSig)=-2$ and $\hdel\cdot\hSig=0$. Moreover,
\[
\hat{\Lambda}_2\simeq H^2(S^{[2]},\bZ)^{\iota^{[2]*}},
\]
and $\hat{\Lambda}_1$ is the index-two sublattice of $\hat{\Lambda}$ obtained by replacing the
$(-1)^2$-part with the sublattice generated by $\hdel$ and $\hSig$.
\end{itemize}

\medskip
\noindent\textbf{Monodromy groups.}
\begin{itemize}
\item The geometric monodromy group $\Mon^2(X)$ acts on $\Lambda$. Through the correspondence
\[
\Aut(\Lambda)\simeq\Aut(\hat{\Lambda})
\]
of Lemma \ref{lem:twist-spec}, it defines the group $\Mon^2(\hat{\Lambda})$.

\item The group $\Mon^2(\hat{\Lambda}_1)$ consists of the restrictions to $\hat{\Lambda}_1$ of elements of
$\Mon^2(\hat{\Lambda})$ which preserve $\hat{\Lambda}_1$.

\item For $i=2,3$, the group $\Mon^2(\hat{\Lambda}_i)$ consists of the restrictions of elements of
$\Mon^2(\hat{\Lambda}_1)$ which act trivially on $\hat{\Lambda}_i^\perp$.
\end{itemize}
With these extension--restriction conventions, we have
\[
\Mon^2(\hat{\Lambda}_3)
\subseteq\Mon^2(\hat{\Lambda}_2)
\subseteq\Mon^2(\hat{\Lambda}_1)
\subseteq\Mon^2(\hat{\Lambda})\simeq\Mon^2(X).
\]
Corollary \ref{cor:inheritedMononHat} supplies elements of these groups from monodromy operators of
$S^{[2]}$ commuting with $\iota^{[2]*}$.

\subsection{Monodromy orbits in $\hat{\Lambda}_3$}\label{subsec:monK32-part}

In this subsection, we consider the sublattice $\hat{\Lambda}_3 = U^{ 3} \oplus (-2)\subset \hat{\Lambda}_1$, where the generator of $(-2)$ is the class $\hdel$.

\begin{nota}
For the rest of this article, fix elements $L_i \in U \subset \hat{\Lambda}_1$ of square $2i$ for each $i\in \bZ$.
E.g. one can choose the elements $ie + f$, where $e,f$ is a standard basis for which $U$ has intersection matrix
$\begin{pmatrix}
0&1\\ 1& 0
\end{pmatrix}$.
\end{nota}

\begin{lemme}\label{lem:K3-2-part}
  The $\Mon^2(\hat{\Lambda}_3)$-orbit of a primitive element in $U^{ 3} \oplus (-2)$ is uniquely determined by
  its square and its divisibility.

  More precisely, we prove the following:
  let  $v \in \hat{\Lambda}_3$ be a primitive element.
  Then there exists a monodromy operator $f\in \Mon^2(\hat{\Lambda}_3)$  such that
  $v$ is moved to an element of the following form: 
  
  $f(v)=\left\{
\begin{array}{lll}
L_{i} &\textrm{with\ } i=\frac{1}{2}q(v) & \textrm{if\ }\div(v)=1  \\
2L_{i} -\hdel &\textrm{with\ } i=\frac{1}{8}(q(v)+2) & \, \textrm{if\ } \div(v)=2 \\
\end{array}
\right. $
\end{lemme}
The proof will make use of the following two well-known statements:

the Eichler criterion, which we will frequently use in this section (see \cite[Lemma 3.5]{Gritsenko-Hulek-Sankaran}, originally due to \cite[Chapter 10]{Eichler}\TODO{add a working reference}).
\begin{lemme}\label{lem:Eichler}
  Let $\Gamma$ be a lattice with $U^{ 2} \subseteq \Gamma$. Fix two elements $v,w \in \Gamma$ which
  satisfy
  \begin{enumerate}
  \item $q(v) = q(w)$,
  \item $\div(v)=\div(w) =: r$,
  \item $\frac{v}{r}=\frac {w}{r} \in A_\Gamma \coloneqq \Gamma^\vee / \Gamma$.
  \end{enumerate}
  Then there exists $\varphi \in \Aut(\Gamma)$ such that
  $\varphi(v)=w$, and such that the induced action $\varphi_A$ on the discriminant group $A_\Gamma$ is the identity.
\end{lemme}

Furthermore, recall the following description of the monodromy group of varieties of $K3^{[2]}$-type:
\begin{thm}[{see \cite[Lemma 9.2]{Markman11}}] \label{thm:MonK32}
  For a K3 surface $S$ the monodromy group $\Mon^2(S^{[2]})$ coincides with $O^+(\Lambda_{K3^{[2]}})$, which is the order 2  subgroup of $\Aut(\Lambda_{K3^{[2]}})$ which consists of automorphisms preserving the positive cone.
\end{thm}

\begin{proof}[Proof of Lemma \ref{lem:K3-2-part}]
  The proof of Lemma \ref{lem:K3-2-part} is an immediate consequence of the previous statements.

  First apply the Eichler criterion (Lemma \ref{lem:Eichler}) for $\Gamma =\hat{\Lambda}_3$:
  observe that for a given element $v\in U^{ 3}\oplus (-2)$ the claimed image element has
  the same square and divisibility.

  In the case of $\div(v)=1$, note that $\frac{v}{1} = 0 \in A_{\hat{\Lambda}_3}$ (since $\frac{v}{1} \in
  \hat{\Lambda}_3$). Therefore, the Eichler criterion applies automatically in this case.

  For the case of $\div(v)=2$, note that $v$ can be written as $v=aL + b\hdel$ with a primitive element $L\in
  U^{ 3}$ and $\hdel$ as before.
  The fact that $\div(v)=2$ implies that $a$ is divisible by two.
  Since we assumed that $v$ is primitive, $b$ is odd. Note that $A_{\hat{\Lambda}_3} \iso \bZ/2\bZ$ is
  spanned by the image of $\frac{\hdel}{2}$, since $U$ is unimodular.
  Therefore, we can also apply the Eichler criterion in this case.

  In both cases, the Eichler criterion shows that there is $f \in \Aut(\hat{\Lambda}_3)$ such that
  $f(v)$ coincides with the claimed image. 
  
  Extending this by the identity on the respective orthogonal complements, we will by abuse of notation
  consider $f \in \Aut (\Lambda_{K3^{[2]}})$ resp.~$f\in \Aut(\hat{\Lambda}_1)$.
  Applying Theorem \ref{thm:MonK32}, we observe that up to potentially swapping a sign in one of the copies of $U$, $f\in \Aut (\Lambda_{K3^{[2]}})$ is in fact an
  element of $\Mon^2(\Lambda_{K3^{[2]}})$.
   And therefore, Corollary \ref{cor:inheritedMononHat} implies that $f\in \Aut (\hat{\Lambda}_1)$
  is in $\Mon^2(\hat{\Lambda}_1)$ as claimed.
\end{proof}

\subsection{Monodromy orbits in the lattice $\hat{\Lambda}_2$}
In this section, we study the monodromy group for the lattice $\hat{\Lambda}_2 = U^3\oplus E_8(-2) \oplus
(-2)$. Notice, that via the identifications described in Section \ref{sec:eq-lattices} the lattice
$\hat{\Lambda}_2$ corresponds to the 
${\iota^{[2]}}^*$-invariant lattice of $\Lambda_{K3^{[2]}}$.
Let us refine the methods from the previous section to describe properties of the monodromy orbits in this
lattice $\hat{\Lambda}_2$.
For this we need to deal with the $E_8(-2)$-part of the lattice.

For the basic notions considering discriminant groups of lattices, we refer to \cite{Nikulin}.
Recall that $E_8$ is a unimodular lattice and therefore, the discriminant group $A_{E_8}=0$ is trivial.
Pick a basis of  $E_8(-2)$ for which the intersection matrix is: \TODO{(-2) times the one associated
  to the $E_8$-graph}
$$\begin{pmatrix}
-4 & 2 & 0 & 0 & 0 & 0 & 0 & 0\\
2 & -4 & 2 & 0 & 0 & 0 & 0 & 0\\
0 & 2 & -4 & 2 & 0 & 0 & 0 & 0\\
0 & 0 & 2 & -4 & 2 & 0 & 0 & 0\\
0 & 0 & 0 & 2 & -4 & 2 & 0 & 2\\
0 & 0 & 0 & 0 & 2 & -4 & 2 & 0\\
0 &0 & 0 & 0 & 0 & 2 & -4 & 0 \\
0 &0 &0 & 0 & 2 & 0 & 0 & -4  
\end{pmatrix}.
$$
For the lattice $E_8(-2)$ one can deduce that the discriminant group $A_{E_8(-2)}\iso (\bZ/2\bZ)^{ 8}$ which is
generated by the residue classes of one half of the generators of the lattice.
The quadratic form $q$ on $E_8(-2)$ induces a quadratic form $\bar{q}\colon A_{E_8(-2)}\to \bZ/2\bZ$.
The possible values of $\bar{q}$ on elements of $A_{E_8(-2)}$ are in fact 1 and 0 (these values are achieved
e.g.~by the residue classes of  $\frac{v}{2}$ for lattice elements of $v\in E_8(-2)$ with squares $-4$ and $-8$).

Denote by $\Aut(A_{E_8(-2)})$ the automorphisms of $A_{E_8(-2)}$ preserving the quadratic form $\bar{q}$.
Note that every automorphism $\varphi\in \Aut(E_8(-2))$ induces an element $\bar{\varphi}\in
\Aut(A_{E_8(-2)})$. Therefore, there exists an induced action of $\Aut(E_8(-2))$ on $A_{E_8(-2)}$.
\begin{lemme}\label{lem:A-E8-orbits}
  There exist precisely three $\Aut(E_8(-2))$-orbits in $A_{E_8(-2)}$:
  they are given by $0$, and by non-zero elements $\overline{e_1}, \overline{e_2}\in A_{E_8(-2)}$ with $\bar{q}(\overline{e_1})=1$ respectively $\bar{q}(\overline{e_2})=0$.
\end{lemme}
\begin{proof}
This lemma is a direct consequence of results in \cite[Chapter 4]{Griess}.
Let $$\mathbb{L}_0=\left\{0\right\},\ \mathbb{L}_{-4}=\left\{\left.\alpha\in E_8(-2)\right|\ \alpha^2=-4\right\},\ \mathbb{L}_{-8}=\left\{\left.\alpha\in E_8(-2)\right|\ \alpha^2=-8\right\}.$$
As explained in \cite[just after Notation 4.2.32]{Griess}, the natural map $b:\mathbb{L}_0\cup \mathbb{L}_{-4} \cup \mathbb{L}_{-8}\rightarrow A_{E_8(-2)}$ is surjective. The images $b(\mathbb{L}_{-4})$ and $b(\mathbb{L}_{-8})$ correspond respectively to the elements of square 1 and the non trivial elements of square 0 in $A_{E_8(-2)}$. 

However by \cite[Corollary 4.2.41]{Griess} and \cite[Lemma 4.2.46 (2)]{Griess} respectively, $\Aut(E_8(-2))$ acts transitively on $\mathbb{L}_{-8}$ and on $\mathbb{L}_{-4}$. The result follows.
\end{proof}

Consider $E_8(-2) \subset E_8(-1)\oplus E_8(-1)$ consisting of elements of the form $(e,e)$ for $e\in
E_8(-1)$ and denote the sublattice consisting of elements of the form $(e, -e)$ by $E^a$. Note that this gets
naturally identified with the anti-invariant lattice of $S$ via the marking described in Section
  \ref{sec:eq-lattices}.
\begin{lemme} \label{lem:extension}
  With this notation, any lattice isometry $\varphi \in \Aut(\hat{\Lambda}_2)$ can be extended to a lattice isometry $\Phi \in \Aut(U^3\oplus
  E_8(-1)\oplus E_8(-1)\oplus (-2))$, with the additional property, that $\Phi$ preserves the sublattice
  $E^a$.
\end{lemme}
\begin{proof}
  This is an immediate consequence of \cite[Corollary 1.5.2]{Nikulin} (applied to twice the lattice
  $\hat{\Lambda}_2$) and the surjection $\Aut(E_8(-2)) \surj \Aut(A_{E_8(-2)})$ (which enables us to choose an appropriate
  extension on the orthogonal complement $E^a$ of $\hat{\Lambda}_2$). 
\end{proof}

Fix two elements in $E_8(-2)$: one element $e_1$ of square $-4$ and one element $e_2$ of square $-8$.  Note
that according to Lemma \ref{lem:A-E8-orbits} the residue classes of $\frac{e_1}{2}$ and  $\frac{e_2}{2}$ in $A_{E_8(-2)}$ represent the two
non-zero orbits under the action of the isometry group. For
coherence of the notation, choose the elements such that the residue of $\frac{e_1}{2}$ in the discriminant is
$\overline{e_1}$ and the residue of $\frac{e_2}{2}$ is $\overline{e_2}$.
\begin{prop} \label{prop:6monorb-in-invariant}
  Let $v\in \hat{\Lambda}_2$ be a primitive non-zero element. Denote by $v_{E_8}$  the projection of $v$ to the $E_8(-2)$-part
  of the lattice, and let ${\bar{v}}_{E_8}$ be the image of $\half v_{E_8}$ in the discriminant group $A_{E_8(-2)}$.
  Then there exists a monodromy operator
  $f\in \Mon^2(\hat{\Lambda}_2)$ such that
     $$f(v)=\left\{
\begin{array}{lllll}
1) &L_{i} & \textrm{if\ }\div(v)=1&\textrm{with\ } i=\frac{1}{2}q(v) &  \\
2) &2L_{i} - \hdel & \, \textrm{if\ } \div(v)=2, & q(v)=8i-2,  & \textrm{and\ } {\bar{v}}_{E_8}=0\\
3) &2L_{i} + e_1 & \, \textrm{if\ } \div (v)=2,& q(v)=8i-4&\\
4) &2L_{i+1} + e_2  & \, \textrm{if\ } \div(v)=2, & q(v)=8i  & \\
5) &2L_{i} + e_1 - \hdel & \, \textrm{if\ } \div(v)=2, & q(v)=8i-6  & \\
6) &2L_{i+1} + e_2 -\hdel & \, \textrm{if\ } \div(v)=2, & q(v)=8i-2,  & \textrm{and\ } {\bar{v}}_{E_8}\neq 0.\\
\end{array}
\right. $$
\end{prop}
\begin{remark}
  Note that the values of $q$ and $\div$ uniquely distinguish the representative in the above list, except from the
  cases 2 and 6, where the additional condition on ${\bar{v}}_{E_8}$ is needed to determine the known representative
  of the orbit.
\end{remark}
\TODO{for ourselves, and probably for second version: If I remember correctly, we showed that the groups $\hdel^\perp$ and
          $(e_2+\hdel)^\perp$ in $A_{\hat{\Lambda}_2}$ are not isomorphic, thus showing, that they actually represent
  different orbits with respect to $\Aut(\hat{\Lambda}_2)$}
\begin{proof}
  Let us first observe that if $\div(v)=1$ the exact same proof as for Lemma \ref{lem:K3-2-part} for the case
    of divisibility 1 applies. 

  Therefore, we only need to deal with the case, where $\div(v)=2$.
  Start by observing that the discriminant group of $\hat{\Lambda}_2$ is $A_{E_8(-2)}\times \bZ/2\bZ$.

  For our given element $v\in \hat{\Lambda}_2$, denote by $\bar{v}$ the image of $\frac{v}{2}$ in the discriminant group, and let
  $\bar{v}_{E_8}$ be the $A_{E_8(-2)}$-part of this.
  By Lemma \ref{lem:A-E8-orbits}, there exists $\varphi \in \Aut(E_8(-2))$ such that $\bar{\varphi}(\bar{v}_{E_8}) \in
  A_{E_8(-2)}$ coincides with one of $\{0,\overline{e_1},\overline{e_2}\}$. For the corresponding $\varphi_1
  \in \Aut(E_8(-1))$ consider $(\varphi_1,\varphi_1) \in \Aut(E_8(-1)\oplus  E_8(-1))$, which obviously commutes with
  the swapping of the two factors and induces $\varphi$ on $E_8(-2)\subset E_8(-1)\oplus  E_8(-1)$.
  Extend this to
  $\Phi \in \Aut(\Lambda_{K3^{[2]}})$ via the identity on the other direct summands.
  By Theorem \ref{thm:MonK32} the operator $\Phi\in \Mon^2(S^{[2]})$ is in the monodromy group
  and therefore the induced action on $\hat{\Lambda}_2$ is an element of $\Mon^2(\hat{\Lambda}_2)$ by Proposition \ref{MonoM'}.  By construction, this
  restricts to $\varphi \in \Aut(E_8(-2))$.
  Therefore, up to first applying the above monodromy operator, we may assume that $\bar{v}\in
  \{0,\overline{e_1},\overline{e_2}\}\times \bZ/2\bZ$.

  For the second step, observe that  cases 2) to 6) listed in the proposition correspond 
  precisely to the non-zero elements of $\{0,\overline{e_1},\overline{e_2}\}\times \bZ/2\bZ$.
  By varying the parameter $i$, the elements in the list can furthermore achieve all possible values for
  $q(v)$ with the prescribed residue in $\{0,\overline{e_1},\overline{e_2}\}\times \bZ/2\bZ$.

  Therefore, for our given element $v \in \hat{\Lambda}_2$ with $\bar{v}\in
  \{0,\overline{e_1},\overline{e_2}\}\times \bZ/2\bZ$, we can choose $v_0$ from the above list (for
  appropriate choice of $i$) such that $q(v)=q(v_0)$ and $\bar{v}=\overline{v_0}\in A_{\hat{\Lambda}_2}$ (and
  $\div(v)=\div(v_0)=2$ follows automatically).
  Therefore, by the Eichler criterion (Lemma \ref{lem:Eichler}), there exists an automorphism $\varphi\in \Aut(\hat{\Lambda}_2)$ such that
  $\varphi(v)=v_0$. This can be extended to an automorphism $\Phi \in \Aut(U^3\oplus E_8(-1)\oplus E_8(-1)\oplus
  (-2))$ of the lattice $\Lambda_{K3^{[2]}}$ by Lemma \ref{lem:extension}.
  Observe that up to changing a sign in one of the copies of $U$, we can assume that $\Phi\in
  \Mon^2(\Lambda_{K3^{[2]}})$ by Theorem \ref{thm:MonK32}. Since this monodromy operator commutes with ${\iota^{[2]}}^*$ (it
  preserves the invariant lattice and the anti-invariant lattice by construction) Proposition \ref{MonoM'}
  shows that it induces a monodromy operator on $\Lambda_{M'}$ which in turn corresponds to $\varphi$ extended by the
  identity via Lemma \ref{lem:twist-spec}. Therefore, again up to potentially changing a sign in one of the
  copies of $U$, the automorphism $\varphi \in \Mon^2(\hat{\Lambda}_2)$. This
  completes the proof.
\end{proof}

\subsection{Induced monodromy orbits on the lattice $\hat{\Lambda}_1$}
Recall that
$\hat{\Lambda}_1= U^{ 3} \oplus E_8(-2)\oplus (-2) \oplus \Z(-2)$.

\begin{thm} \label{thm:9monorb}
  Let $v\in \hat{\Lambda}_1$ be a primitive non-zero element. Denote by $v_{E_8}$  the projection of $v$ to the $E_8(-2)$-part
  of the lattice, and, when $\div(v)=2$, let ${\bar{v}}_{E_8}$ be the image of $\half v_{E_8}$ in the discriminant group $A_{E_8(-2)}$.
  Then there exists a monodromy operator
  $f\in \Mon^2(\hat{\Lambda}_1)$ such that
     $$f(v)=\left\{
\begin{array}{lllll}
1) &L_{i} & \textrm{if\ }\div(v)=1&\textrm{with\ } i=\frac{1}{2}q(v) &  \\
2) &2L_{i} - \hdel & \, \textrm{if\ } \div(v)=2, & q(v)=8i-2,  & \textrm{and\ } {\bar{v}}_{E_8}=0\\
3) &2L_{i+1} + e_2 -\hdel & \, \textrm{if\ } \div(v)=2, & q(v)=8i-2,  & \textrm{and\ } {\bar{v}}_{E_8}\neq 0\\
4) &2L_i - \hdel - \hSig & \, \textrm{if\ } \div (v)=2,& q(v)=8i-4,&\textrm{and\ } {\bar{v}}_{E_8}=0\\
5) &2L_{i+1} + e_2-\hdel - \hSig & \, \textrm{if\ } \div (v)=2,& q(v)=8i-4,&\textrm{and\ } \bar{q}({\bar{v}}_{E_8})=0, {\bar{v}}_{E_8}\neq 0\\
6) &2L_{i} + e_1 & \, \textrm{if\ } \div (v)=2,& q(v)=8i-4,&\textrm{and\ } \bar{q}({\bar{v}}_{E_8})=1\\
7) &2L_{i} + e_1 - \hdel & \, \textrm{if\ } \div(v)=2, & q(v)=8i-6,  & \\
8) &2L_i + e_1 - \hdel - \hSig & \, \textrm{if\ } \div(v)=2, & q(v)=8i-8,
       & \textrm{and\ } \bar{q}({\bar{v}}_{E_8})=1\\
9) &2L_{i+1} + e_2  & \, \textrm{if\ } \div(v)=2, & q(v)=8i,  & \textrm{and\ } \bar{q}({\bar{v}}_{E_8})=0.\\
\end{array}
\right. $$
\end{thm}
\TODO{Regrouper les cas 2 et 3 en utilisant le lemme \ref{monolemma}}
\begin{rmk}\label{rem:L0goesaway}
  Observe that whenever $L_0$ is involved in the statement of Theorem \ref{thm:9monorb}, it can be replaced by
  $0$ (apart from case 1) ) since both elements are in the same monodromy orbit.
\end{rmk}

\begin{proof} The proof of this theorem consists of a series of applications of Proposition \ref{prop:6monorb-in-invariant}
  and the existence of the monodromy operator $R_{\frac{\hdel - \hSig}{2}}$ (compare Remark \ref{RdeltaSigma} and notation of Section \ref{notation}).
  
  First note that since $v$ is primitive, it can be expressed as
  $v=k\gamma + a\hdel + b \hSig$, where $\gamma\in U^3\oplus E_8(-2)$ is a primitive element, and $\gcd(a,b,k)=1$.

  First let us assume that  $\div(v)=1$.
  The element $k\gamma + a\hdel \in \hat{\Lambda}_2$ corresponds to $\gcd(k,a)$ times a primitive element of
  divisibility 1 inside  $\hat{\Lambda}_2$.
  Therefore, by Proposition \ref{prop:6monorb-in-invariant} there exists an element $f_1 \in \Mon^2(\hat{\Lambda}_2)$ such that
  $f_1(k\gamma + a\hdel) = \gcd(k,a)\cdot L_{q_1}$ for a suitable choice of $q_1$. By extending $f_1$ to $\Mon^2(\hat{\Lambda}_1)\supseteq
  \Mon^2(\hat{\Lambda}_2)$, observe that
  $f_1(v)=\gcd(k,a) \cdot L_{q_1} + b \hSig$.
  Apply the monodromy operator $R_{\frac{\hdel -\hSig}{2}}$ to obtain $\gcd(k,a) \cdot L_{q_1} + b \hdel$, which is a primitive
  element of divisibility 1 in $\hat{\Lambda}_2$.
  Using once again Proposition \ref{prop:6monorb-in-invariant} find $f_2 \in \Mon^2(\hat{\Lambda}_2)\subseteq \Mon^2(\hat{\Lambda}_1)$
  such that $f_2(\gcd(k,a) \cdot L_{q_1} + b \hdel)=L_{q_2}$.
  
  The composition of these monodromy operators is therefore the claimed $f\in \Mon^2(\hat{\Lambda}_1)$ and concludes
  the proof under the assumption that $\div(v)=1$.

  Therefore, we only need to deal with the case that $\div(v)=2$. Let $\bar{v}$ be the residue class of
  $\frac{v}{2}$ in $A_{\hat{\Lambda}_1}$.
  Let us first work under the additional assumption that  $\gcd(k,a)$ is odd (while still assuming $\div(v)=2$).
   Under this assumption, the element $k\gamma + a\hdel \in \hat{\Lambda}_2$ corresponds to $\gcd(k,a)$ times a
   primitive element $v_1$ in  $\hat{\Lambda}_2$, satisfies that $\bar{v}_1=\bar{v}_{\hat{\Lambda}_2}\in A_{\hat{\Lambda}_2}$, where
   $\bar{v}_1$ is the residue of $\frac{v_1}{2}$, and $\bar{v}_{\hat{\Lambda}_2}$ is the $A_{\hat{\Lambda}_2}$-part of
   $\bar{v}$.
   Then there exists a monodromy operator $f_1\in
   \Mon^2(\hat{\Lambda}_2)\subset\Mon^2(\hat{\Lambda}_1)$ such that $f_1(v_1)$ is one of the cases 2) to 6) from Proposition
   \ref{prop:6monorb-in-invariant}.
   After applying the operator $R_{\frac{\hdel - \hSig}{2}}$, we obtain an element $v_2$ of one of the following
   forms:
        $$v_2=\left\{
\begin{array}{lllll}
  a) &2\gcd(k,a)L_{q_1} &&+ b\hdel &- \gcd(k,a)\hSig \\
b) &2\gcd(k,a)L_{q_1} &+ \gcd(k,a)e_1 &+ b\hdel&\\
c) &2\gcd(k,a)L_{q_1} &+ \gcd(k,a)e_2  &+ b\hdel&\\
d) &2\gcd(k,a)L_{q_1} &+ \gcd(k,a)e_1 &+ b\hdel&-\gcd(k,a)\hSig \\
e) &2\gcd(k,a)L_{q_1} &+ \gcd(k,a)e_2 &+ b\hdel&- \gcd(k,a)\hSig 
\end{array}
\right. $$
for suitable choice of $q_1$.
  Since $\gcd(k,a,b)=1$, note that the $\hat{\Lambda}_2$-component of $v_2$ is primitive unless we are dealing with
  case a) from above and at the same time $b$ is even.
We will separately consider these two situations:

  \smallskip
  \noindent Case 1:
  If the $\hat{\Lambda}_2$-component $v_{2,\hat{\Lambda}_2}$ of $v_2$ is primitive, then one can find a monodromy operator
  moving $v_{2,\hat{\Lambda}_2}$ to one of the cases from Proposition \ref{prop:6monorb-in-invariant}.
  In cases b) and c) from above the resulting element already attains the form of one of the cases claimed in our theorem. 
  In the other cases, apply the operator $R_{\frac{\hdel - \hSig}{2}}$ once again, followed by Proposition
  \ref{prop:6monorb-in-invariant} to conclude the proof of case 1. 

    \smallskip
    \noindent Case 2: Assume that $v_{2,\hat{\Lambda}_2}$ is non-primitive, and therefore we are in case a) from above
    with the 
    additional assumption that $b=2b'$ is even.
    This means that $v_2=2(\gcd(k,a)L_{q_1} + b'\hdel)- \gcd(k,a)\hSig$. Since $\gcd(k,a)$ is odd, $v_{2,\hat{\Lambda}_2}$ is twice
    a primitive element of divisibility 1, and therefore $v_{2,\hat{\Lambda}_2}$ can be moved to an element of the form
    $2L_{q_2}-\gcd(k,a)\hSig$. Applying the operator $R_{\frac{\hdel - \hSig}{2}}$ and using Proposition
  \ref{prop:6monorb-in-invariant} completes the proof of case 2 (since $\gcd(k,a)$ is odd by assumption).

  \smallskip
  The only remaining case, which we have not yet analyzed, is when $\div(v)=2$ and $\gcd(k,a)$ is even.
  Notice that under these assumptions $b$ is odd and in particular $\gcd(k,b)$ is odd.
  Therefore, after application of the operator $R_{\frac{\hdel - \hSig}{2}}$, we find ourselves in the above
  setting, which concludes the final case of the proof. 
\end{proof}

From this we can easily deduce a corresponding statement for the original lattice $\Lambda_{M'}$.
We need to fix some notation in order to formulate the statement.
Consider an irreducible symplectic orbifold $X$ of Nikulin-type, with a given marking
$\varphi\colon H^2(X,\bZ)\overset{\iso}{\to} \Lambda_{M'}$.
Let $L_i^{(2)}\in U(2)$ be an element of square $4i$ (corresponding to the element $L_i \in U$). Furthermore, fix
elements $e_1^{(1)}$ and $e_2^{(1)}\in E_8(-1)$ with squares $q(e_1^{(1)})=-2$ and $q(e_2^{(1)})=-4$ (these elements
correspond to the elements $e_1$ and $e_2 \in E_8(-2)$).

\begin{thm} \label{thm:9monorb-M'}
  Let $v\in \Lambda_{M'}$ be a primitive non-zero element. Denote by $v_{E_8}$  the projection of $v$ to the $E_8(-1)$-part
  of the lattice, and let ${\bar{v}}_{E_8}$ be its image in the $\bZ/4\bZ$-module $E_8(-1)/4E_8(-1)$.
  Then there exists a monodromy operator
  $f\in \Mon^2(X)$ with the following properties.

  If $v$ corresponds to a ray of divisibility 1 in $\hat{\Lambda}_1$ (see below for a checkable condition),
  then
  $$
  \text{\emph{1)}}\qquad  f(v)=L^{(2)}_{i},\ \text{with}\quad \div(v)=2\quad\text{and}\quad q(v)=4i.
  $$
  Otherwise, $v$ corresponds to a ray of divisibility 2 in $\hat{\Lambda}_1$, and
  $$
  f(v)=\left\{
  \begin{array}{lllll}
2) &2L^{(2)}_{i} - \delta' & \, \textrm{if\ } \div(v)=2, & q(v)=16i-4,  & \textrm{and\ }  {\bar{v}}_{E_8}=0\\
3) &2L^{(2)}_{i+1} + 2e_2^{(1)} -\delta' & \, \textrm{if\ } \div(v)=2, & q(v)=16i-4,  & \textrm{and\ } {\bar{v}}_{E_8}\neq 0\\
4) &L^{(2)}_i - \frac{\delta' + \Sigma'}{2} & \, \textrm{if\ } \div (v)=2,& q(v)=4i-2,&\textrm{and\ } {\bar{v}}_{E_8}=0\\
5) &L^{(2)}_{i+1} + e_2^{(1)}-\frac{\delta' + \Sigma'}{2} & \, \textrm{if\ } \div (v)=1,& q(v)=4i-2,&\textrm{and\ }
q(v_{E_8})\equiv 0 \pmod{4}\\
6) &L^{(2)}_{i} + e_1^{(1)}  & \, \textrm{if\ } \div (v)=1,& q(v)=4i-2,&\textrm{and\ } q(v_{E_8})\equiv 2 \pmod{4} \\
7) &2L^{(2)}_{i} + 2e_1^{(1)} - \delta' & \, \textrm{if\ } \div(v)=2, & q(v)=16i-12,  & \textrm{and\ } {\bar{v}}_{E_8}\neq 0\\
8) &L^{(2)}_i + e_1^{(1)} - \frac{\delta' + \Sigma'}{2} & \, \textrm{if\ } \div(v)=1, & q(v)=4i-4,
       & \textrm{and\ } q(v_{E_8})\equiv 2 \pmod{4} \\
9) &L^{(2)}_{i+1} + e_2^{(1)}  & \, \textrm{if\ } \div(v)=1, & q(v)=4i,  & \textrm{and\ } q(v_{E_8})\equiv 0 \pmod{4}.\\
  \end{array}
  \right.
  $$
  The condition that $v$ corresponds to a ray of divisibility 1 in $\hat{\Lambda}_1$ is equivalent to
  satisfying the following three conditions inside $\Lambda$:
  \begin{compactenum}
  \item The restriction $v_{U(2)^3}$ of $v$ to $U(2)^3$ is not divisible by 2,
  \item the restriction $v_{E_8}$ of $v$ to $E_8(-1)$ is divisible by 2, and
  \item the restriction $v_{(-2)\oplus (-2)}$ to  $\langle \frac{\delta'+ \Sigma'}{2} ,  \frac{\delta'-
    \Sigma'}{2}\rangle$ is contained in the sublattice $\langle \delta', \Sigma'\rangle$.
  \end{compactenum}
  
\end{thm}

\begin{proof}
  Similar to Lemma \ref{lem:twist-gen}
  use the inclusion $\hat{\Lambda}_1(2) \subset \Lambda_{M'}$ and notice that under this correspondence $L_i$
  is sent to $L_{i}^{(2)}$, $e_i$ is sent to $2e_i^{(1)}$, $\hdel$ to $\delta'$, and $\hSig$ to $\Sigma'$.
  Then passing to the primitive element in the ray and determining the new square and divisibility gives the
  new cases.

  For the  part of the condition involving $v_{E_8}$, simply check that under the assumptions on $q$ and $\div$ these are
  equivalent to the corresponding ones in $\hat{\Lambda}_1$ from Theorem \ref{thm:9monorb}.

  The same formalism gives a straightforward verification of the characterization when $v$ corresponds
  to a ray of divisibility 1 in $\hat{\Lambda}_1$.
\end{proof}

\begin{cor}
  There are at most 3 monodromy orbits of primitive non-zero
  elements with prescribed square and divisibility (both in $\Lambda_{M'}$ and in $\hat{\Lambda}_1$).
\end{cor}
\begin{proof}
  Since the values of $q$ and $\div$ are given, this can be read off immediately from the statements of Theorems \ref{thm:9monorb} and
  \ref{thm:9monorb-M'}. 
\end{proof}

\begin{rmk}\label{rem:L0goesaway2} Again, one can replace $L_0^{(2)}$ by $0$ in all cases except from Case 1),
  since both elements in question lie in the same monodromy orbit.
\end{rmk}

Let us conclude this section by the following observation:
\begin{cor}\label{cor:Chiara}
  For every element $v\in \Lambda_{M'}$ of square $-4$ and divisibility $2$ the reflection (defined by
  $R_v(\alpha)\coloneqq \alpha -2 \frac{(\alpha,v)_q}{q(v)}v$) gives an element in the monodromy group.
\end{cor}
\begin{proof}
  Begin by observing that this property is equivalent for different elements in the same monodromy orbit.
  Therefore, it suffices to check it for one representative of each orbit. By the list from Theorem
  \ref{thm:9monorb-M'}, the orbits of square $-4$ and divisibility $2$ have one of the following representatives:
  $L_{-1}^{(2)}$ (Case 1), $\delta'$ (Case 2), or $2L_1^{(2)} + 2 e_2^{(1)} - \delta'$ (Case 3).

  The
  associated elements $L_{-1}$, $\delta$, and $2L_1 +  e_2 - \delta$ in the invariant part of
  $\Lambda_{K3^{[2]}}$ all have square $-2$ and thus their reflections correspond to monodromy operators
  on $\Lambda_{K3^{[2]}}$ (e.g.~by Theorem \ref{thm:MonK32}), which commute with $\iota ^{[2]*}$.
  Therefore, Proposition \ref{MonoM'} applies to show that the claimed reflections in $\Lambda_{M'}$ are indeed
  monodromy operators.
\end{proof}
\begin{cor}\label{Lastmonodromy}
  For every element $v\in \Lambda_{M'}$ of square $-2$ and divisibility $2$ the reflection (defined by
  $R_v(\alpha)\coloneqq \alpha -2 \frac{(\alpha,v)_q}{q(v)}v$) gives an element in the monodromy group.
\end{cor}
\begin{proof}
The proof is similar to the one of Corollary \ref{cor:Chiara}. From Theorem \ref{thm:9monorb-M'}, we note that
all such elements $v$ are in the same monodromy orbit (Case 4) which contains the element $\frac{\delta'-\Sigma'}{2}$. Hence the result follows from Remark \ref{RdeltaSigma}.
\end{proof}
\section{Determining the wall divisors}\label{endsection}
In this section, we combine the results from the last sections to prove the main theorem of this paper: Theorem
\ref{main} which gives a complete description of the wall divisors for Nikulin-type orbifolds.

For the proof of the theorem let us start from some $X_0$ which is the Nikulin orbifold associated to some K3
surface $S_0$ obtained by the construction in Section \ref{M'section}.
Fix a marking $\varphi_0 \colon H^2(X_0,\Z) \to \Lambda_{M'}= U(2)^{ 3} \oplus E_8(-1)\oplus (-2) \oplus (-2)$,
where as usual $U(2)^{ 3} \oplus E_8(-1)$ corresponds to the part coming from the invariant lattice of
$S_0$ and the two generators of the $(-2)$-part are $\frac{\delta'+ \Sigma'}{2}$ and $\frac{\delta'-
  \Sigma'}{2}$.
Let us recall the details of this identification:
for the K3 surface $S_0$ with a symplectic involution $\iota$ the $\iota$-anti-invariant part of the lattice  is
isomorphic to  $E_8(-2)$ and one can choose a marking
$\varphi_{S_0}\colon H^2(S_0,\bZ) \to \Lambda_{K3}\iso U^3 \oplus E_8(-1)^2$, such that the  $\iota^*$
acts by exchanging the two copies of $E_8(-1)$. Therefore, the invariant lattice of $\iota$ corresponds to
$\Lambda_{K3}^\iota \iso U^3\oplus E_8(-2)$, where the elements of $E_8(-2)$ are of the form $e+\iota^*(e)$ for elements $e$ in the
first copy of $E_8(-1)$. Similarly the anti-invariant lattice of $\iota$ is $E_8(-2)$ consisting of elements
of the form
$e-\iota^*(e)$. We will denote the anti-invariant part of the lattice by $E^a$.
With this convention, the lattice $U(2)^{3} \oplus E_8(-1)$ corresponds to the
invariant via a twist as described in Lemma \ref{lem:twist-gen} of
$\Lambda_{K3}^{\iota}$.

In order to prove the main theorem, we need to determine for each ray in $\Lambda_{M'}$ whose generator is of
negative Beauville-Bogomolov square, whether it corresponds to a wall divisor for Nikulin-type orbifolds.  Obviously, this notion is invariant under the monodromy action by the deformation
invariance (see Theorem \ref{wall}).
It therefore suffices to pick one representative for each monodromy orbit and to determine it for this choice.

By Lemma \ref{lem:twist-gen}, the rays of $\Lambda_{M'}$ are in (1:1)-correspondence with rays in the lattice
$\hat{\Lambda}_1$, and obviously the property that the generator has negative square coincides in both cases.
Therefore, we only need to deal with the cases from Theorem \ref{thm:9monorb} (respectively Theorem
\ref{thm:9monorb-M'}), for which $i$ is chosen such that the square is negative.

\smallskip
\noindent {\it Case 1:}
As a warm-up, let us start with Case 1 of Theorem \ref{thm:9monorb} separately (i.e.~the ray in question is generated by the element $L_i$ with $i<0$). Note
that $L_i$ naturally corresponds to an element $\varphi_{S_0} ^{-1}L_i\in H^2(S_0,\bZ)$.
Let $(S,\varphi_S)$ be a marked K3 surface such that the Picard lattice of $S$ is $\Pic(S)=\varphi_S^{-1}(L_i
\oplus E^a)$ (which exists by the surjectivity of the period map).
If $i< -1$, then $S$ does not contain any effective curve (since $\Pic(S)$ only has non-zero elements of square smaller
than $-2$). Therefore, we are in the situation of Section \ref{genericM'} and
 one observes that $L_i$ does not correspond to a wall divisor for
Nikulin-type orbifolds if $i<-1$: in fact, $L_i\in \hat{\Lambda}_1$ corresponds to $L_{i}^{(2)}\in \Lambda$, which is
 not a wall divisor by Proposition \ref{exwalls}.
Note that the divisors $L_{i}^{(2)}\in \Lambda$ satisfy $q(L_{i}^{(2)})=4i$, $\div(L_{i}^{(2)})=2$, and $(L_i^{(2)})_{U(2)^3}$ is
not divisible by 2, which confirms Theorem \ref{main} for Case 1 if $i<-1$.
 
 If $i=-1$ (and therefore $q(L_i)=-2$), we are in the situation of Section \ref{onecurve}, and one can deduce
 from Proposition \ref{walldiv1} that $L_{i}^{(2)}\in \Lambda$ (which is precisely $D_C'$) is a wall divisor for Nikulin-type orbifolds, which confirms Theorem
 \ref{main} in this case.
 
\smallskip
\noindent {\it Cases 2, 4:}
As in the proof for Case 1, choose a K3 surface $S$ such that $\Pic(S)=\varphi_S ^{-1}(L_i\oplus E^a)$.
As before, Section \ref{genericM'} applies and Proposition \ref{exwalls} implies that for $i<-1$, the only rays corresponding to wall divisors are
$\delta'$ and $\Sigma'$, and therefore there are no additional  wall divisors of the forms given in Cases 2
and 4 in this example.
Similarly, the results from Section \ref{onecurve} imply that for $i=-1$ the wall divisors are $\delta'$,
$\Sigma'$, $L_{i}^{(2)}$, and $L_{i}^{(2)}-\half(\delta' + \Sigma')$ (compare
Proposition \ref{walldiv1}). Therefore, Case 4 provides precisely a wall divisor of square $-6$ and
divisibility $2$, thus confirming Theorem \ref{main} in this case.

However, for Cases 2 and 4, we also need to consider $i=0$ since the total square will still be negative.  By
Remark \ref{rem:L0goesaway}, the monodromy orbits of $2L_0 + \hdel$ (resp.~$2L_0 + \hdel + \hSig$) coincide
with those of $\hdel$ and $\hdel + \hSig$, and therefore we can instead deform towards a very 
general K3 surface $S$ with a symplectic involution (i.e.~$\Pic(S)=E^a$) and apply the results from Section
 \ref{genericM'} to observe that $\delta'$ is a wall divisor of square $-4$ and divisibility $2$, whereas
 $\half (\delta' + \Sigma')$ is not. 

\smallskip
\noindent {\it Cases 6, 7, and 8:}
Similar to the previous situation, the element $2L_i + e_1$ naturally corresponds to an element
$\varphi_{S_0}^{-1}(2L_i + e_1) \in H^2(S_0,\bZ)$.
Notice, that we are only interested in the cases, where $q(2L_i + e_1)<0$, which corresponds to $i\leq 0$.
Under this condition, the direct sum $(2L_i+ e_1) \oplus E^a$ is a negative definite sublattice of
$\Lambda_{K3}$.
However, notice that this in itself cannot be realized as the Picard lattice of a K3 surface, since it is not
a saturated sublattice:
By definition $e_1\in E_8(-2)$ is an element of square $-4$, where $E_8(-2)$ is part of the invariant lattice.
Therefore, by the above observation, there exists an element $e_1^{(0)}$ in the first copy of $E_8(-1)$ of square
$-2$ such that
$e_1 = e_1^{(0)} + \iota ^*(e_1^{(0)})$. With this notation the element
$2L_i + 2e_1^{(0)} = 2L_i + (e_1^{(0)} + \iota^*(e_1^{(0)})) + (e_1^{(0)} - \iota ^*(e_1^{(0)}))\in  (2L_i+ e_1) \oplus E^a$, but
the element $L_i + e_1^{(0)}$ is not part of this direct sum. In fact, $(L_i + e_1^{(0)})\oplus E^a$ is the
saturation.

With this knowledge, let us choose a marked K3 surface $(S,\varphi_S)$ such that
$\Pic(S) =\varphi_S^{-1}((L_i + e_1^{(0)})\oplus E^a)$.
Note that if $i<0$, then $S$ does not contain any effective curve (since every non-zero element has square
smaller than $-2$).
Therefore, the results from Section \ref{genericM'} apply, and one observes that none of these cases provides
wall divisors.

If $i=0$, then $S$ contains exactly two elements of square $-2$ which are exchanged by $\iota^*$: the elements
$L_0 + e_1^{(0)}$ and $L_0 + \iota^*(e_1^{(0)})$. In this case according to Remark \ref{rem:L0goesaway}, we can choose $L_0=0$.
Thus for $i=0$ we find ourselves in the setting of Section \ref{sec:twocurves} with $e_1^{(0)}=C$. Note that the element $D_C'$
from Section \ref{sec:twocurves} corresponds precisely to the element $e_1^{(1)}$ 
with our
notation.
We can therefore deduce from Proposition \ref{prop:twocurves}, that for $i=0$ the Cases 6 and 7 provide wall
divisors 
($e_1^{(1)}$ 
with square $-2$ and
divisibility $1$, and 
$2e_1^{(1)} - \delta'$
with square $-12$ and divisibility $2$), whereas by Remark \ref{Remark:twocurves} Case 8 does not provide a wall
divisor, thus confirming Theorem \ref{main}.

\smallskip
\noindent {\it Cases 3, 5, and 9:}
Again, the element $2L_{i+1}+ e_2$ corresponds to an element in $\varphi_{S_0}(H^2(S_0,\bZ))$. We need to
consider $i\leq 0$ to cover all possibilities for wall divisors with negative squares.

If $i<0$, then the lattice $(2L_{i+1}+ e_2) \oplus E^a \subseteq \Lambda_{K3}$ is negative definite, and again
its saturation is $(L_{i+1} + e_2^{(0)})\oplus E^a $ for the corresponding element $e_2^{(0)}$ in the first copy of
$E_8(-1)$ (we recall that $e_2^{(0)}$ has square $-4$).
Similar to the above, deform to a marked K3 surface $(S,\varphi_S)$ such that $\varphi_S(\Pic(S))= (L_{i+1} +
e_2^{(0)})\oplus E^a$.
Observe that all non-zero elements of this lattice have squares smaller than $-2$. Therefore, we can apply the
results from Section \ref{genericM'} to observe that we do not find any further wall divisors in these cases.

For the remaining case $i=0$, we need to prove that both $2L_{1}^{(2)} + 2e_2^{(1)}- \delta'$
(with square $-4$ and divisibility $2$)
and $L_{1}^{(2)} + e_2^{(1)} - \frac{\delta' + \Sigma'}{2}$ (with square $-2$ and divisibility $1$)
correspond to wall divisors.
If $i=0$, then $S$ contains exactly one primitive effective element of square $0$: $2L_{1}+ e_2$.
Thus for $i=0$ we find ourselves in the setting of Section \ref{sec:elliptic}. Note that the element $D_\gamma'$
from Section \ref{sec:elliptic} corresponds precisely to the element $L_1^{(2)} + e_2^{(1)}$ with our
notation.
Let $M'$ constructed as in  Section \ref{sec:elliptic}.
From the investigation of the current section, 
we know that a wall divisor on $M'$ has the numerical properties of a wall divisor that we already found or possibly of $2L_{1}^{(2)} + 2e_2^{(1)}- \delta'$ or
$L_{1}^{(2)} + e_2^{(1)} - \frac{\delta' + \Sigma'}{2}$. That is: we have proved that a wall divisor necessarily has one of the numerical properties listed in Theorem \ref{main}. Therefore, Lemma \ref{mainelliptic} shows that $L_{1}^{(2)} + e_2^{(1)} - \frac{\delta' + \Sigma'}{2}$ is a wall divisor. Finally, $2L_{1}^{(2)} + 2e_2^{(1)}- \delta'$ is also a wall divisor by
Lemma \ref{monolemma}.

This concludes the analysis of all possible cases and thus the proof of Theorem \ref{main}.

\section{Application}
\subsection{A general result about the automorphisms of Nikulin-type orbifolds}
\begin{prop}\label{AutM'}
Let $X$ be an orbifold of Nikulin-type and $f$ an automorphism of $X$. If $f^*=\id$ on $H^2(X,\Z)$, then $f=\id$.
\end{prop}
This section is devoted to the proof of this proposition. We will adapt Beauville's proof \cite[Proposition 10]{Beauville1982}.
\begin{lemme}\label{AutS}
Let $S$ be a K3 surface such that $\Pic S=\Z H\oplus^{\bot} E_8(-2)$ with $H^2= 4$. According to Proposition \ref{involutionE8} or from \cite[Proposition 2.3]{Sarti-VanGeemen}, the K3 surface $S$ is endowed with a symplectic involution $\iota$.
Let $f\in\Aut(S)$ such that $f$ commutes with $\iota$. Then $f=\iota$ or $f=\id$.
\end{lemme}
\begin{proof}
We adapt the proof of \cite[Corollary 15.2.12]{HuybrechtsK3}.
The lattice $\Pic(S)$ contains no class of square $-2$, since all its squares are divisible by $4$.
Hence one of the two classes $H$ and $-H$ is ample; replacing $H$ by $-H$ if necessary, we assume that
$H$ is ample. Let $f\in \Aut(S)$ which commutes with $\iota$. Then $f^*$ preserves the invariant and
anti-invariant sublattices of $\iota^*$. Since
\[
\Pic(S)^{\iota^*}=\Z H
\quad\text{and}\quad
\left(\Pic(S)^{\iota^*}\right)^{\bot}=E_8(-2),
\]
we have $f^*(H)=\pm H$. As $f^*$ preserves the ample cone and $H$ is ample, it follows that
$f^*(H)=H$.
By \cite[Corollary 3.3.5]{HuybrechtsK3}, $f$ acts on $T(S)$ (the transcendental lattice of $S$) as $-\id$ or $\id$.
Since the K3 lattice $H^2(S,\Z)$ is unimodular and $\Pic(S)$ and $T(S)$ are primitive orthogonal
complements, their discriminant groups are identified by a canonical anti-isometry
\[
\mu:A_{\Pic(S)}\longrightarrow A_{T(S)}.
\]
The isometry $f^*$ is defined on the whole lattice $H^2(S,\Z)$, hence it preserves the gluing of
$\Pic(S)\oplus T(S)$ and its induced actions on the discriminant groups satisfy
\[
\overline{f^*_{|T(S)}}\circ\mu
=\mu\circ\overline{f^*_{|\Pic(S)}}.
\]
Now $f^*(H)=H$, so the action on the cyclic factor
$A_{\langle H\rangle}\simeq\Z/4\Z$ is the identity. On the other hand, $-\id_{T(S)}$ acts as multiplication
by $-1$ on the corresponding cyclic factor of $A_{T(S)}$, and is therefore incompatible with the above
gluing relation. This forces $f^*_{|T(S)}=\id$. Moreover, we can consider
$f^*_{|E_8(-2)}$ as an isometry of $E_8(-2)$.
By \cite[Theorem 4.2.39]{Griess}, the isometry group of $E_8(-2)$ is finite, hence $f^*_{|E_8(-2)}$ is of finite order.
Therefore by \cite[Chapter 15 Section 1.2]{HuybrechtsK3}, there are only two possibilities for $f$: $\id$ or a symplectic involution.
Moreover, by \cite[Proposition 15.2.1]{HuybrechtsK3}
, there is at most one symplectic involution on $S$. 
\end{proof}
\begin{lemme}\label{commute}
Let $(S,\iota)$ be a K3 surface endowed with a symplectic involution such that $\Pic S=\Z H\oplus^{\bot} E_8(-2)$ with $H^2= 4$.
Let $M'$ be the Nikulin orbifold constructed from $(S,\iota)$ as in Section \ref{M'section}. 
Let $(g,h)\in\Aut(S)^2$ such that $g\times h$ induces a bimeromorphism on $M'$ via the non-ramified cover $$\gamma: S\times S\smallsetminus \left(\Delta_{S^{2}}\cup S_{\iota}\cup(\Fix \iota\times \Fix \iota)\right)\rightarrow M'\smallsetminus \left(\delta'\cup \Sigma' \cup \Sing M' \right)$$
introduced in Section \ref{inv0M'} (i.e., there exists a bimeromorphism $\rho$ on $M'$ such that $\rho\circ\gamma=\gamma\circ (g\times h)$). Then $g$ and $h$ commute with $\iota$.
\end{lemme}
\begin{proof}
It is enough to prove that $g$ commutes with $\iota$; the proof for $h$ being identical. 
Let $$A:=\left\{\left.\eta=\eta_1\circ\eta_2\circ\eta_3\right|\ (\eta_1,\eta_3)\in\left\{\id,\iota\right\}^2\ \text{and}\ \eta_2\in\left\{\id, g, h\right\}\right\}.$$
Let $V:=S\times S\smallsetminus \left(\Delta_{S^{2}}\cup S_{\iota}\cup(\Fix \iota\times \Fix \iota)\right)$;
we consider the following open subset of $V$: 
$$V^{o}:=\left\{\left.(a,b)\in V\ \right|\ g(a)\neq \eta(b),\ g\circ\iota(a)\neq \eta(b),\ \forall \eta\in A\ \text{and}\ a\notin g^{-1}(\Fix \iota)\right\}.$$
Since $g\times h$ induces a bimeromorphism on $M'$, there exists an open subset $\mathcal{W}$ of $S\times S$ such that for all $(a,b)\in \mathcal{W}$:
\small
\begin{align*}
g\times h\left(\left\{(a,b),(b,a),(\iota(a),\iota(b)),\right.\right.&\left.\left.(\iota(b),\iota(a))\right\}\right)\\
&=\left\{(g(a),h(b)),(h(b),g(a)),(\iota\circ g(a),\iota\circ h(b)),(\iota\circ h(b),\iota\circ g(a))\right\}.
\end{align*}
That is:
\begin{align*}
\left\{(g(a),h(b)),(g(b),h(a)),\right.&\left.(g\circ\iota(a),h\circ\iota(b)),(g\circ\iota(b),h\circ\iota(a))\right\}\\&=\left\{(g(a),h(b)),(h(b),g(a)),(\iota\circ g(a),\iota\circ h(b)),(\iota\circ h(b),\iota\circ g(a))\right\}.
\end{align*}
\normalsize
If we choose in addition $(a,b)\in V^{o}$, then
there is only one possibility: 
$$g\circ\iota(a)=\iota\circ g(a).$$
It follows that $g$ commutes with $\iota$ on an open set of $S$, so on all $S$.
\end{proof}
We are now ready to prove Proposition \ref{AutM'}.
\begin{proof}[Proof of Proposition \ref{AutM'}]
We consider an orbifold $X$ of Nikulin-type and an automorphism $f$ of $X$ such that $f^*=\id$. In particular, $f$ is a symplectic automorphism. 
Let $(S,\iota)$ be a K3 surface endowed with a symplectic involution satisfying the hypotheses of Lemma \ref{AutS}; we consider the Nikulin orbifold $M'$ constructed from $(S,\iota)$ as in Section \ref{M'section}. 
By \cite[Lemma 2.17]{Menet-Riess-20}, there exist two markings $\varphi$ and $\psi$ such that $(X,\varphi)$ and $(M',\psi)$ are connected by a sequence of twistor spaces. Moreover, by Remark \ref{twistorinvo}, $f$ extends to an automorphism along all twistor spaces. In particular, $f$ induces an automorphism $f'$ of $M'$. We consider $\gamma$, the non-ramified cover of Lemma \ref{commute}:
$$\gamma: S\times S\smallsetminus \left(\Delta_{S^{2}}\cup S_{\iota}\cup(\Fix \iota\times \Fix \iota)\right)\rightarrow M'\smallsetminus \left(\delta'\cup \Sigma' \cup \Sing M' \right).$$ Since $V=S\times S\smallsetminus \left(\Delta_{S^{2}}\cup S_{\iota}\cup(\Fix \iota\times \Fix \iota)\right)$ is simply connected, it is the universal cover of $U:= M'\smallsetminus \left(\delta'\cup \Sigma' \cup \Sing M' \right)$.

Since $f'^*$ acts as $\id$ on $H^2(M',\Z)$, it fixes the classes of $\delta'$ and $\Sigma'$.
Both $\delta'$ and $\Sigma'$ are prime exceptional divisors and are therefore rigid: each of them is the
unique effective divisor in its cohomology class. Consequently, $f'$ preserves both $\delta'$ and $\Sigma'$
(it also preserves the set $\Sing M'$). Hence $f'$ induces an automorphism of $U$ and then of $V$. Therefore,
it induces a bimeromorphism $\overline{f'}$ of $S\times S$. Let $s_2:S\times S\rightarrow S\times S:
(a,b)\mapsto (b,a)$. By \cite[Theorem 4.1 (d)]{Oguiso}, $\overline{f'}$ can be written as a sequence of
compositions between $s_2$ and automorphisms of the form $g_i\times h_i$, where $g_i$, $h_i$ are in
$\Aut(S)$.
Since we are interested in the automorphism $f'$ on $M'$, without loss of generality, we can assume that $\overline{f'}=g\times h$, with $g$, $h$ in $\Aut(S)$.

Therefore, by Lemma \ref{commute}, $g$ and $h$ commute with $\iota$.
It follows from Lemma \ref{AutS} that $(g,h)\in\left\{\id,\iota\right\}^2$.
So, the unique possibility for $\overline{f'}$ to induce a non-trivial morphism on $U$ is $\overline{f'}=\id\times \iota$ (or $\iota\times\id $).
However, in this case, as seen in Section \ref{inv0M'}, $f'$ would interchange $\delta'$ 	and $\Sigma'$. This is a contradiction with the fact that $f'^*=\id$ on $H^2(M',\Z)$. Therefore, we obtain $f'=\id$ and then $f=\id$.
\end{proof}
\subsection{Construction of a non-standard symplectic involution on orbifolds of Nikulin-type}\label{Application}
Adapting the vocabulary introduced in \cite{Mongardi-2013}, we state the following definition.
\begin{defi}
 Let $Y$ be an irreducible symplectic manifold of $K3^{[2]}$-type endowed with a symplectic involution $\iota$. 
 Let $M'$ be the Nikulin orbifold constructed from $(Y,\iota)$ as in Example \ref{exem}. Let $G\subset \Aut(Y)$ such that all $g\in G$ commute with $\iota$.
 Then $G$ induces an automorphism group $G'$ on $M'$. The group $G'$ is called a \emph{natural automorphism group} on $M'$ and $(M',G')$ is called a \emph{natural pair}.
 
 Let $X$ be an irreducible symplectic orbifold of Nikulin-type and $H\subset \Aut(X)$. The group $H$ will be said \emph{standard} if the pair
 $(X,H)$ is deformation equivalent to a natural pair $(M',G')$; in this case, we say that the pair $(X,H)$ is a \emph{standard pair}.
\end{defi}
\begin{thm}\label{Involution2}
Let $X$ be an irreducible symplectic orbifold of Nikulin-type such that there exists $D\in\Pic (X)$ with $D^2=-2$ and $\div(D)=2$. Then there exists an irreducible symplectic orbifold $Z$ bimeromorphic to $X$ and a non-standard symplectic involution $\iota$ on $Z$ such that:
$$H^2(Z,\Z)^{\iota}\simeq U(2)^3\oplus E_8(-1)\oplus (-2)\ \text{and}\ H^2(Z,\Z)^{\iota\bot}\simeq (-2).$$
\end{thm}
\begin{proof}
By Theorem \ref{main}, $D$ is not a wall divisor, hence there exists $\beta\in\mathcal{B}\mathcal{K}_{X}$ and $g\in\Mon^2_{\Hdg}(X)$ 
 such that $(g(D),\beta)_{q}=0$. Let $f:X\dashrightarrow Z$ be a bimeromorphic map such that $f_*(\beta)$ is a K\"ahler class on $Z$.
 We set $D':=f_*\circ g(D)$.
 By Corollary \ref{Lastmonodromy}, the involution $R_{D'}$ is a Hodge monodromy operator on $H^2(Z,\Z)$.
 Moreover $R_{D'}(f_*(\beta))=f_*(\beta)$. Hence by Theorem \ref{mainHTTO}, 
 there exists an automorphism $\iota$ of $Z$ such that $\iota^*=R_{D'}$. Moreover, by Proposition \ref{AutM'}, $\iota$ is an involution.
 Since $\iota^*=R_{D'}$, we have $H^2(Z,\Z)^{\iota}=D'^{\bot}$ and
 $H^2(Z,\Z)^{\iota\bot}=\Z D'\simeq (-2)$. Moreover, since $(D')^2=-2$ and
 $\div(D')=2$, we have the orthogonal decomposition
 \[
 H^2(Z,\Z)=D'^{\bot}\oplus \Z D'.
 \]
 By Theorem \ref{BBform}, the lattice $D'^{\bot}$ has signature $(3,12)$ and discriminant group
 $(\Z/2\Z)^7$. It is therefore an even $2$-elementary lattice; its parity invariant is necessarily
 $\delta=1$, since its signature difference is $3-12=-9$, which is not divisible by $4$.
 The classification of even $2$-elementary lattices \cite[Theorem 4.3.2]{Nikulin} then gives
 \[
 H^2(Z,\Z)^{\iota}=D'^{\bot}\simeq U(2)^3\oplus E_8(-1)\oplus (-2).
 \]
 
 Now, we show that $\iota$ is non-standard. We assume that $\iota$ is standard and we will find a contradiction. If $\iota$ is standard, there exists
 a natural pair $(M',\iota')$ deformation equivalent to $(Z,\iota)$. 
 Since $\iota'$ is natural, $\iota'^*(\Sigma')=\Sigma'$. Moreover, since $(M',\iota')$ is deformation equivalent to $(Z,\iota)$,
 there exists $D'\in \Pic M'$ such that $q_{M'}(D')=-2$, $\div(D')=2$ and $H^2(M',\Z)^{\iota'\bot}=\Z D'$.
 However, since $\Sigma'\in H^2(M',\Z)^{\iota'}$, we obtain by Theorem \ref{BBform} that:
 $$D'\in \Sigma'^{\bot}\simeq U(2)^3\oplus E_8(-1)\oplus (-4).$$
 For the rest of the proof, we identify $\Sigma'^{\bot}$ with $U(2)^3\oplus E_8(-1)\oplus (-4)$. 
 It follows that $D'$ can be written:
 $$D'=\alpha+\beta,$$
 with $\alpha\in U(2)^3\oplus (-4)$ and $\beta\in E_8(-1)$.
Since $\div(D')=2$, we have 
 $$D'=\alpha+2\beta',$$
 with $\beta'\in E_8(-1)$.
 It follows that $q_{M'}(D')\equiv0\pmod 4$.
 This is a contradiction with $q_{M'}(D')=-2$.
\end{proof}
\TODO{What is the fixed locus of $\iota$ ?}

\TODO{check for the following things: twister -> twistor, k\"ahler -> K\"ahler}
\TODO{look at Nikulin entry} 
\TODO{generic -> very general in many cases}
\TODO{check that notation of $q$ is consistent}

\bibliographystyle{alpha}
\bibliography{Literatur}

\noindent
Gr\'egoire \textsc{Menet}

\noindent
Acad\'emie militaire de Saint-Cyr Co\"etquidan

\noindent
 CReC Saint-Cyr (Centre de recherche de l'acad\'emie militaire de Saint-Cyr) 

\noindent 
 56380 Guer, France.

\noindent
{\tt gregoire.menet@ac-amiens.fr}

\bigskip
    
 \noindent
Ulrike \textsc{Rie}\ss

\noindent
Institute for Theoretical Studies - ETH Z\"urich

\noindent 
Clausiusstrasse 47, Building CLV, Z\"urich (Switzerland)

\noindent
{\tt uriess@ethz.ch}

\end{document}